\documentclass[11pt,reqno,tbtags,a4paper]{amsart}
\usepackage{amssymb}
\usepackage{url}
\usepackage[square,numbers]{natbib}
\bibpunct[, ]{[}{]}{;}{n}{,}{,}
{\makeatletter
\gdef\section{\@startsection{section}{1}%
  \z@{.7\linespacing\@plus\linespacing}{.5\linespacing}%
  {\normalfont\bfseries\centering}}
}

\title[Constrained $U$-statistics, random strings and permutations]
{Asymptotic normality for $m$-dependent and constrained $U$-statistics,
with applications to pattern matching in random strings and permutations}

\date{21 June, 2021;
revised and extended 8 March, 2022}

\author{Svante Janson}
\thanks{Supported by the Knut and Alice Wallenberg Foundation}
\address{Department of Mathematics, Uppsala University, PO Box 480,
SE-751~06 Uppsala, Sweden}
\email{svante.janson@math.uu.se}
\urladdr{http://www.math.uu.se/svante-janson}

\numberwithin{equation}{section}

\renewcommand\le{\leqslant}
\renewcommand\ge{\geqslant}

\allowdisplaybreaks

\theoremstyle{plain}
\newtheorem{theorem}{Theorem}[section]
\newtheorem{lemma}[theorem]{Lemma}

\newtheorem{corollary}[theorem]{Corollary}
\newtheorem{conj}[theorem]{Conjecture}

\theoremstyle{definition}

\newtheorem{problem}[theorem]{Problem}

\newtheorem{exampleqqq}[theorem]{Example}
\newenvironment{example}{\begin{exampleqqq}}
  {\hfill\qedsymbol\end{exampleqqq}}

\newtheorem{remarkqqq}[theorem]{Remark}
\newenvironment{remark}{\begin{remarkqqq}}
  {\hfill\qedsymbol\end{remarkqqq}}

\newenvironment{romenumerate}[1][-10pt]{
\addtolength{\leftmargini}{#1}\begin{enumerate}
 \renewcommand{\labelenumi}{\textup{(\roman{enumi})}}%
 }{\end{enumerate}}

\newenvironment{PQenumerate}[1]{
\begin{enumerate}
 \renewcommand{\labelenumi}{\textup{#1}}%
 }{\end{enumerate}}

\newcounter{oldenumi}
{\setcounter{oldenumi}{\value{enumi}}
\begin{romenumerate} \setcounter{enumi}{\value{oldenumi}}}
{\end{romenumerate}}

\newcounter{thmenumerate}

\newcounter{xenumerate}   

\newcounter{pfcases}
\newcommand\resetcases{\setcounter{pfcases}{0}}

\newcommand\pfcase[1]{\smallskip\noindent\refstepcounter{pfcases}%
 \emph{Case \arabic{pfcases}: #1.}\noindent}

\newcommand{\refT}[1]{Theorem~\ref{#1}}
\newcommand{\refTs}[1]{Theorems~\ref{#1}}

\newcommand{\refL}[1]{Lemma~\ref{#1}}

\newcommand{\refR}[1]{Remark~\ref{#1}}

\newcommand{\refS}[1]{Section~\ref{#1}}
\newcommand{\refSs}[1]{Sections~\ref{#1}}
\newcommand{\refSS}[1]{Section~\ref{#1}}

\newcommand{\refE}[1]{Example~\ref{#1}}

\newcommand{\refApp}[1]{Appendix~\ref{#1}}

\begingroup
  \divide\count255 by 60
  \multiply\count255 by -60
  \advance\count255 by \time
  \ifnum \count255 < 10 \xdef\klockan{\the\count1.0\the\count255}
  \else\xdef\klockan{\the\count1.\the\count255}\fi
\endgroup

\newcommand\nopf{\qed}   

\DeclareMathOperator*{\sumx}{\sum\nolimits^{*}}

\newcommand{\sumk}{\sum_{k=1}^\infty}

\newcommand{\sumin}{\sum_{i=1}^n}

\newcommand{\sumjl}{\sum_{j=1}^\ell}
\newcommand{\sumjb}{\sum_{j=1}^\ellb}
\newcommand{\sumkn}{\sum_{k=1}^n}
\newcommand{\sumkl}{\sum_{k=1}^\ell}

\newcommand\set[1]{\ensuremath{\{#1\}}}

\newcommand\xpar[1]{(#1)}
\newcommand\bigpar[1]{\bigl(#1\bigr)}
\newcommand\Bigpar[1]{\Bigl(#1\Bigr)}

\newcommand\lrpar[1]{\left(#1\right)}
\newcommand\bigsqpar[1]{\bigl[#1\bigr]}
\newcommand\xsqpar[1]{[#1]}
\newcommand\Bigsqpar[1]{\Bigl[#1\Bigr]}

\newcommand\xcpar[1]{\{#1\}}

\newcommand\Bigcpar[1]{\Bigl\{#1\Bigr\}}

\newcommand\abs[1]{\lvert#1\rvert}
\newcommand\bigabs[1]{\bigl\lvert#1\bigr\rvert}
\newcommand\Bigabs[1]{\Bigl\lvert#1\Bigr\rvert}

\def\rompar(#1){\textup(#1\textup)}    
\newcommand\xfrac[2]{#1/#2}

\newcommand\innprod[1]{\langle#1\rangle}

\def\xexp(#1){e^{#1}}
\newcommand\ceil[1]{\lceil#1\rceil}
\newcommand\floor[1]{\lfloor#1\rfloor}

\newcommand\setn{\set{1,\dots,n}}

\newcommand\ntoo{\ensuremath{{n\to\infty}}}
\newcommand\Ntoo{\ensuremath{{N\to\infty}}}

\newcommand\xtoo{\ensuremath{{x\to\infty}}}

\newcommand\norm[1]{\lVert#1\rVert}
\newcommand\bignorm[1]{\bigl\lVert#1\bigr\rVert}
\newcommand\Bignorm[1]{\Bigl\lVert#1\Bigr\rVert}
\newcommand\lrnorm[1]{\left\lVert#1\right\rVert}

\newcommand\punkt{.\spacefactor=1000}    
\newcommand\iid{i.i.d\punkt}    
\newcommand\ie{i.e\punkt}
\newcommand\eg{e.g\punkt}

\newcommand\cf{cf\punkt}
\newcommand{\as}{a.s\punkt}
\newcommand{\aex}{a.e\punkt}

\newcommand{\tend}{\longrightarrow}
\newcommand\dto{\overset{\mathrm{d}}{\tend}}
\newcommand\pto{\overset{\mathrm{p}}{\tend}}
\newcommand\asto{\overset{\mathrm{a.s.}}{\tend}}
\newcommand\eqd{\overset{\mathrm{d}}{=}}

\newcommand\op{o_{\mathrm p}}

\newcommand\bbR{\mathbb R}

\newcommand\bbN{\mathbb N}

\newcommand\bbZ{\mathbb Z}

\newcounter{CC}
\newcounter{cc}

\renewcommand\Re{\operatorname{Re}}

\newcommand\E{\operatorname{\mathbb E{}}}
\renewcommand\P{\operatorname{\mathbb P{}}}

\newcommand\Var{\operatorname{Var}}
\newcommand\Cov{\operatorname{Cov}}

\newcommand\Be{\operatorname{Be}}

\newcommand\sgn{\operatorname{sgn}}

\newcommand\ga{\alpha}
\newcommand\gb{\beta}

\newcommand\gD{\Delta}
\newcommand\gf{\varphi}
\newcommand\gam{\gamma}

\newcommand\gl{\lambda}

\newcommand\go{\omega}

\newcommand\gs{\sigma}

\newcommand\gss{\sigma^2}

\newcommand\eps{\varepsilon}

\renewcommand\phi{\xxx}  

\newcommand\cA{\mathcal A}
\newcommand\cB{\mathcal B}

\newcommand\cD{\mathcal D}

\newcommand\cF{\mathcal F}

\newcommand\cI{\mathcal I}

\newcommand\cL{{\mathcal L}}

\newcommand\cS{{\mathcal S}}

\newcommand\indic[1]{\boldsymbol1\xcpar{#1}}

\newcommand\etta{\boldsymbol1}

\newcommand\qw{^{-1}}
\newcommand\qww{^{-2}}
\newcommand\qq{^{1/2}}
\newcommand\qqw{^{-1/2}}

\newcommand\intoi{\int_0^1}

\newcommand\oi{\ensuremath{[0,1]}}

\newcommand\ooo{[0,\infty)}

\newcommand\setoi{\set{0,1}}

\newcommand\dd{\,\mathrm{d}}
\newcommand\ddx{\mathrm{d}}

\newcommand{\gsf}{$\gs$-field}
\newcommand{\ui}{uniformly integrable}

\newcommand\lhs{left-hand side}
\newcommand\rhs{right-hand side}

\newcommand\xoo{_1^\infty}
\newcommand\xoox[1]{_1^{#1}}
\newcommand\xoon{\xoox{n}}
\newcommand\xooo{\xoo}
\renewcommand\xooo{\relax}

\newcommand\fS{\mathfrak{S}}

\newcommand\bw{\mathbf{w}}
\newcommand\bpi{\boldsymbol{\pi}}
\newcommand\Ustat{$U$-statistic}
\newcommand\Ustats{$U$-statistics}
\newcommand\hX{\widehat X}
\newcommand\hY{\widehat Y}

\newcommand\mdep{$m$-dependent}
\newcommand\elli{{\ell-1}}
\newcommand\nn{^{(n)}}

\newcommand\cDq{{\cD{=}}}
\newcommand\cDz{{\cD'{=}}}
\newcommand\mx{> m}
\newcommand\dx{d_*}

\newcommand\fx{f_*}
\newcommand\nuk{k}
\newcommand\hU{\widehat U}
\newcommand\ubbo[1]{\bigsqpar{#1-\E #1}}
\newcommand\cDoo{\cD_\infty}
\newcommand\tensor{\otimes}
\newcommand\hxi{\hat\xi}
\newcommand\hf{\hat f}
\newcommand\N{\mathsf{N}}
\newcommand\hcI{\widehat{\cI}}
\newcommand\summm[1]{\sum_{\substack{1\le i_1<\dots<i_{#1}\le n\\i_{j+1}-i_j>m}}}
\newcommand\mux{\nu}

\newcommand\AAA[1]{\textup{(A${}_{#1}$)}}
\newcommand\gamx{\gamma}
\newcommand\gamxx{\gamx^2}
\newcommand\FF{\widehat F}
\newcommand\jjj{^{(k,j)}}
\newcommand\jjjx{^{(k,j)*}}
\newcommand\anj[1]{a_{j,#1,n}}
\newcommand\hUnj{\widehat U_{n,j}}
\newcommand\bx{\mathbf{x}}

\newcommand\pthp{$p$-th power}
\newcommand\fnt{\floor{nt}}
\newcommand\nt{_{\floor{nt}}}
\newcommand\TS{\widehat S}
\newcommand\ellb{b}

\newcommand\tg{\widetilde g}
\newcommand\Npmx{{N_{\pm}(x)}}
\newcommand\nuqw{\nu\qw}

\newcommand{\Holder}{H\"older}

\newcommand\CSineq{Cauchy--Schwarz inequality}

\begin{document}

\begin{abstract} 
We study (asymmetric)  $U$-statistics based on a stationary sequence of
$m$-dependent variables; moreover, we consider constrained $U$-statistics,
where the defining multiple sum only includes terms satisfying some
restrictions on the gaps between indices. Results include a law of large
numbers and a central limit theorem,
together with results on
rate of convergence,
moment convergence, functional convergence and a renewal theory version.

Special attention is paid to degenerate cases where, after the standard
normalization, the asymptotic variance vanishes; in these cases non-normal
limits occur after a different normalization.

The results are motivated by applications to pattern matching in random
strings and permutations. We obtain both new results and new proofs of old
results. 
\end{abstract}

\maketitle

\section{Introduction}\label{S:intro}

The  purpose of the present paper is to present some new results for
(asymmetric) \Ustats{} together with some applications. 
(See \refS{SU} for definitions.)
The results include a strong law of large numbers and a
central limit theorem (asymptotic normality), 
together with results on
rate of convergence,
moment convergence, functional convergence and a renewal theory version.

Many results of these types have been proved for \Ustat{s} under different
hypotheses by a large number of authors, from \citet{Hoeffding} and on.
The new feature of the results here, which are motivated by applications
discussed below, is the combination of the following:
\begin{romenumerate}
  
\item We consider, as in \eg{} \cite{SJIII}, \cite{SJ332} and \cite{HanQ} 
but unlike many   other authors, 
\emph{asymmetric \Ustats} and not just the symmetric case.
(See \refR{Rasymm}.)
\item We consider also \emph{constrained \Ustats}, where the summations are
  restricted as in \eqref{Uc} or \eqref{Uc=}.
\item 
The \Ustats{} are based on an underlying sequence that is not
necessarily  \iid{} (as is usually assumed); 
we assume only that the sequence is stationary and \emph{\mdep}.
(This case has been studied earlier by \eg{} \cite{Sen-m}, but not in the
present asymmetric case.)
\end{romenumerate}

The extension to the \mdep{} case might be of interest for some
applications, but for us the main motivation is that it allows us to reduce
the constrained versions to ordinary \Ustats; hence this extension is
implicitly used also when we apply the results for constrained \Ustats{}
based on \iid{} sequences.

\begin{remark}\label{Rincomplete1}
The combination of the three features (i)--(iii) above is new, but 
they have each been considered separately earlier.

In particular, constrained \Ustats{} 
are special cases of the large class of
 \emph{incomplete \Ustats} \cite{Blom}.
These are, in turn, special cases of the even more general 
\emph{weighted \Ustats}, see e.g.\ 
\cite{ShapiroH}, 
\cite{ONeilR}, 
\cite{Major}, 
\cite{RinottR}, 
\cite{HsingW}, 
\cite{Zhou}, 
\cite{HanQ}.  
(These references show  asymptotic normality under various conditions;
some also study degenerate cases with non-normal limits;
\cite{HanQ} includes the asymmetric case.) 
In view of our applications, we consider here only the
constrained case instead of trying to find suitable conditions for general
weights. 

Similarly,
\Ustats{} have been considered by many authors
for more general weakly dependent sequences
than \mdep{} ones.
In particular,
asymptotic normality has been shown under
various types of mixing conditions  by e.g.\
\cite{Sen-*},
\cite{Yoshihara76, 
Yoshihara92}, 
\cite{DehlingW}. 
We are not aware of any paper on asymmetric \Ustats{} with a mixing
condition on the variables.
Such results might be interesting for future research, but again in view of
our applications, we have not pursued this and consider here only the
\mdep{} case.
\end{remark}

There are thus many previous results yielding asymptotic normality
for \Ustat{s} under various condition.
One general feature, found already in the first paper \cite{Hoeffding},
is that there are degenerate cases where the asymptotic
variance vanishes (typically because of some internal cancellations). 
In such cases, the theorems only yield convergence to 0
and do not imply asymptotic normality; indeed, typically a different
normalization yields a non-normal limit.
It is often difficult to calculate the asymptotic variance exactly, and it
is therefore of great  interest to have simple criteria that show that the
asymptotic variance is non-zero.  
Such a criterion is well known for the standard case of 
(unconstrained) \Ustat{s} based on  \iid{} variables \cite{Hoeffding}.
We give corresponding (somewhat more complicated) criteria for the \mdep{} case
studied here, both in the unconstrained and constrained cases.
(This is one reason for considering only the \mdep{} case in the present
paper, and not more general weakly dependent sequences.)
We show the applicability of our criteria in some examples.

We, as many (but not all) of the references cited above, base our proof of
asymptotic normality on
the decomposition method of \citet{Hoeffding}, with appropriate
modifications.
As pointed out by an anonymous referee, an alternative method 
is to use dependency graphs together with Stein's method
which under an extra moment assumption
yields our main results on asymptotic normality 
together
with an upper bound on the
rate of convergence.
We do not use this method in the main parts of the paper,  partly
because it does not seem to yield simple criteria for non-vanishing of the
asymptotic variance; however, as a complement, we use this method to give
some results on rate of convergence.

\subsection{Applications}
The background motivating our general results is given by some parallel
results for 
pattern matching in random strings and in random permutations
that earlier have been shown by different methods, but easily follow from
our results; we describe
these results here and return to them (and some new results)
in \refSs{Sword} and \ref{Sperm}.
Further applications to pattern matching in random permutations restricted
to
two classes of permutations
are given in \cite{SJ367}.

First,
consider a random string $\Xi_n=\xi_1\cdots\xi_n$
consisting of $n$ \iid{} random letters from a
finite  alphabet $\cA$
(in this context, this is known as a \emph{memoryless source}), 
and consider the number of occurences of a given word
$\bw=w_1\dotsm w_\ell$  as a subsequence;
to be precise,
an \emph{occurrence} of $\bw$ in $\Xi_n$ is 
an increasing sequence of indices $i_1<\dots<i_\ell$ in $[n]=\setn$ such that
\begin{align}
  \label{occur}
\xi_{i_1}\xi_{i_2}\dotsm\xi_{i_\ell}=\bw, 
\qquad \text{\ie, $\xi_{i_k}=w_k$ for every $k\in[\ell]$}
.\end{align}
This number, $N_n(\bw)$ say, was studied by
\citet{FlajoletSzV} who proved that $N_n(\bw)$ is asymptotically normal as
\ntoo.

\citet{FlajoletSzV} studied also a \emph{constrained} version, 
where we are given also
numbers $d_1,\dots,d_{\ell-1}\in\bbN\cup\set\infty=\set{1,2,\dots,\infty}$ 
and  count only occurences of $w$ such that
\begin{align}
  \label{constr}
i_{j+1}-i_j\le d_j,
\qquad 1\le j<\ell.
\end{align}
(Thus the $j$th gap in $i_1,\dots,i_\ell$ has length strictly less than $d_j$.)
We write $\cD:=(d_1,\dots,d_{\ell-1})$, 
and let $N_n(\bw;\cD)$ be the number
of occurrences of $\bw$ that satisfy the constraints \eqref{constr}.
It was shown in \cite{FlajoletSzV} that, for any fixed $\bw$ and $\cD$, 
$N_n(\bw,\cD)$ is asymptotically normal as \ntoo.
See also the book by \citet[Chapter 5]{JacSzp}.

\begin{remark}\label{Rfull}
  Note that $d_j=\infty$ means no constraint for the $j$th gap.
In particular,
$d_1=\dots=d_{\ell-1}=\infty$ yields the unconstrained case;
we denote this trivial (but important) constraint $\cD$ by $\cDoo$.

In the other extreme case, 
if $d_j=1$, then $i_j$ and $i_{j+1}$ have to be adjacent.
In particular,
in the completely constrained case 
$d_1=\dots=d_{\ell-1}=1$, then $N_n(\bw;\cD)$ 
counts  occurences of $\bw$ as a
substring $\xi_i\xi_{i+1}\cdots\xi_{i+\ell-1}$.
Substring counts have been studied by many authors;
some references with central limit theorems or local limit theorems under
varying conditions are
\cite{BenderK}, 
\cite{RegSzp98}, 
\cite{NicodemeSF},
\cite[Proposition IX.10, p.~660]{FlajoletS}.
See also
\cite[Section 7.6.2 and Example 8.8]{Szp} 
and
\cite{JacSzp}; the latter book discusses not only substring and subsequence
counts but also other versions of substring matching problems in random strings.

Note also that the case when all $d_i\in\set{1,\infty}$ means that
$\bw$ is a concatenation $\bw_1\dotsm\bw_b$ (with $\bw$ broken at positions
where $d_i=\infty$), such that an occurence 
now is an occurence of each $\bw_i$ as a substring, with these substrings 
in order and non-overlapping, 
and with arbitrary gaps in between. (A special case of the 
\emph{generalized subsequence problem} in \cite[Section 5.6]{JacSzp};
the general case can be regarded as a sum of such counts over a set of $\bw$.)
\end{remark}

There are similar results for random permutations.
Let $\fS_n$ be the set of the $n!$ permutations of $[n]$.
If $\pi=\pi_1\cdots\pi_n\in\fS_n$  and $\tau=\tau_1\cdots\tau_\ell\in\fS_\ell$,
then an \emph{occurrence} of the pattern $\tau$ in $\pi$ is 
an increasing sequence of indices $i_1<\dots<i_\ell$ in $[n]=\setn$ such that
the order relations in $\pi_{i_1}\cdots\pi_{i_\ell}$ are the same as in
$\tau_1\cdots\tau_\ell$, \ie, $\pi_{i_j}<\pi_{i_k}\iff \tau_j<\tau_k$.

Let $N_n(\tau)$ be the number of occurences of $\tau$ in $\bpi$ when 
$\bpi=\bpi\nn$
is uniformly random in $\fS_n$. \citet{Bona-Normal} proved that $N_n(\tau)$ is
asymptotically normal as \ntoo, for any fixed $\tau$.

Also for permutations, one can consider, and count, constrained occurrences
by again imposing the restriction \eqref{constr} for some
$\cD=(d_1,\dots,d_{\ell-1})$.  
In analogy with strings, we let $N_n(\tau,\cD)$ be the number of constrained
occurences of $\tau$ in $\bpi\nn$ when $\bpi\nn$ is uniformly random in $\fS_n$.
This random number seems to mainly have been studied in the
case when each $d_i\in\set{1,\infty}$, \ie, some $i_j$ are required to be
adjacent to the next one -- such constrained patterns are in the permutation
context known as
\emph{vincular} patterns.
\citet{Hofer} proved asymptotic normality of $N_n(\tau,\cD)$ as \ntoo,
for any fixed $\tau$ and vincular $\cD$. 
The extreme case with $d_1=\dots=d_\elli=1$ was earlier treated by
\citet{Bona3}.
Another (non-vincular)  case that has been studied is
\emph{$d$-descents}, given by $\ell=2$,  $\tau=21$ and $\cD=(d)$;
\citet{Bona-d} shows asymptotic normality
and \citet{Pike} gives a rate of convergence.
 
We unify these results by considering \Ustat{s}.
It is well known and easy to see
that the number $N_n(\bw)$ of unconstrained occurences of a given subsequence
$\bw$ in a random string $\Xi_n$
can be written as an asymmetric $U$-statistic; see 
\refS{Sword} and \eqref{Uword} for details.
There are  general results on asymptotic 
normality of \Ustat{s} that extend the basic result by
\cite{Hoeffding} to the asymmetric case,   see \eg{}
\cite[Corollary 11.20]{SJIII}, \cite{SJ332}.
Hence,
asymptotic normality of $N_n(\bw)$ 
follows directly from these general results.
Similarly, it is well known that the pattern count $N_n(\tau)$
in a random permutation also can be written
as a $U$-statistic, 
see \refS{Sperm} for details,
and again this can be used to prove
asymptotic normality. 
(See \cite{SJ287}, with an alternative proof by this method of the result
by \citet{Bona-Normal}.)

The constrained case is different, since the constrained pattern counts are
not \Ustat{s}. However, they can be regarded as \emph{constrained \Ustats},
which we define in \eqref{Uc} below in analogy with 
the constrained counts above.
As said above, we show in the present paper general limit theorems for such
constrained \Ustats, which thus immediately apply to the constrained pattern
counts discussed above in random strings and permutations.

The basic idea in the proofs is that a constrained \Ustat{} based on a
sequence $(X_i)$ can be written (possibly up to a small error) as an
unconstrained \Ustat{} based on another sequence $(Y_i)$ of random
variables,
where the new sequence $(Y_i)$ is \mdep{} (with a different $m$)
if $(X_i)$ is. (However, even if $(X_i)$ is independent, $(Y_i)$ is in
general not; this is our main motivation for considering \mdep{} sequences.)
The unconstrained \mdep{} case then is treated by standard methods from the
independent case, with appropriate modifications.

\refS{Sprel} contains some preliminaries.
The unconstrained and constrained \Ustat{s} are defined in \refS{SU}, where
also the main theorems are stated.
The degenerate case, when the asymptotic variance in the
central limit theorem \refT{TUM}, \ref{TUMD}, or \ref{TUUN} vanishes,
is discussed later in \refS{S0}, when more notation has been introduced; 
\refTs{T0}, \ref{T0D} and \ref{TUUN0}, repectively,
give criteria that can be used to show that the asymptotic variance is
non-zero in an application. 
On the other hand, \refE{E0} shows that the degenerate
case can occur in new ways for constrained \Ustat{s}.

The reduction to the unconstrained case and some other lemmas are given in
\refS{Slem}, and then the proofs of the main theorems are  completed in
\refSs{Smean}--\ref{SLLN} and \ref{Srate}--\ref{Srenew}.
\refS{Sword} gives applications to the  problem on pattern matching in
random  strings discussed above.
Similarly, \refS{Sperm} gives applications to pattern matching in random
permutations.
Some further comments and open problems
are given in \refS{Sfurther}.
The appendix contains some further results on subsequence counts in 
random strings.

\vfill
\section{Preliminaries}\label{Sprel}

\subsection{Some notation}\label{SSnot}

A \emph{constraint} is, as in \refS{S:intro},  
a sequence $\cD=(d_1,\dots,d_\elli)\in (\bbN\cup\set\infty)^\elli$, for some
given $\ell\ge1$.
Recall that the special constraint
$(\infty,\dots,\infty)$ is denoted by $\cDoo$.
Given a constraint $\cD$, define $b=b(\cD)$ by
\begin{align}\label{b}
b=b(\cD):
=\ell-|\set{j:d_j<\infty}|
=1+|\set{j:d_j=\infty}|
.\end{align}
We say that $b$ is the number of \emph{blocks} defined by $\cD$, see
further \refS{Slem} below.

For a random variable $Z$, and $p>0$, we let 
$\norm{Z}_p:=\bigpar{\E\xsqpar{|Z|^p}}^{1/p}$.

We use $\dto$, $\pto$, and $\asto$, for
convergence of random variables in distribution, probability, and almost
surely (a.s.), respectively.
For a sequence of random variables $(Z_n)$, and a sequence $a_n>0$, we write
$Z_n=\op(a_n)$  when $Z_n/a_n \pto0$.

Unspecified limits are as \ntoo.
$C$ denotes unspecified constants, which may be different at each occurrence.
($C$ may depend on parameters that are regarded as fixed, for example the
function $f$ below; this will be clear from the context.)

We use the convention $\binom nk:=0$ if $n<0$. (We will always have $k\ge0$.) 
Some further standard notation:
$[n]:=\setn$.
 $\max\emptyset:=0$.
 All functions are tacitly assumed to be measurable.

\subsection{\mdep{} variables}\label{SSm}
For reasons mentioned in the introduction, 
we will consider  \Ustat{s} not only based on sequences of
independent random variables, but also based on \mdep{} variables.

Recall that a (finite or infinite) sequence of random variables
$(X_i)_{i}$ is  \emph{$m$-dependent}
if the
two families $\set{X_i}_{i\le k}$ and $\set{X_i}_{i> k+m}$ of random
variables are independent of each other for every $k$.
(Here, $m\ge0$ is a given integer.)
In particular, 0-dependent is the same as independent; thus the 
important independent case is included as the special case $m=0$ below.

It is well known that if $(X_i)_{i\in I}$ is \mdep, and
$I_1,\dots,I_r\subseteq I$ are sets of indices such that
$\operatorname{dist}(I_j,I_k):=\inf\set{|i-i'|:i\in I_j,i'\in I_k}>m$ 
when $j\neq k$, then the families (vectors) of random variables
$(X_i)_{i\in I_1}$, \dots, $(X_i)_{i\in I_r}$ are mutually independent of
each other. (To see this, note first that it suffices to consider the case when
each $I_j$ is an interval; then use the definition and induction on $r$.)
We will use this property without further comment.

In practice, \mdep{} sequences usually occur as \emph{block factors}, \ie{} 
they can be expressed as
\begin{align}\label{block}
  X_i:=h(\xi_i,\dots,\xi_{i+m})
\end{align}
for some \iid{} sequence $(\xi_i)$ of random variables (in some measurable
space $\cS_0$), and a fixed function $h$ on $\cS_0^{m+1}$.
(It is obvious that \eqref{block} then defines a 
stationary \mdep{} sequence.)

\section{\Ustat{s} and main results}\label{SU}

Let $X_1,X_2,\dots$ be a sequence of random variables, taking values in
some measurable space $\cS$,
and let $f:\cS^\ell\to\bbR$ be a (measurable) function of $\ell$ variables,
for some
$\ell\ge1$. 
Then the corresponding \emph{\Ustat} is 
the (real-valued) random variable defined for each $n\ge0$  by
\begin{equation}
  \label{U}
U_n=U_n(f)
=U_n\bigpar{f;(X_i)\xooo}
:=\sum_{1\le i_1<\dots<i_\ell\le n} f\bigpar{X_{i_1},\dots,X_{i_\ell}}
.\end{equation}
\Ustat{s} were introduced by \citet{Hoeffding}, who proved a general central
limit theorem; the present paper gives an extension of his result that builds on
his methods.

\begin{remark}
Of course, for the definition \eqref{U} it suffices to have a finite
sequence $(X_i)\xoon$, but we will in the present paper only
consider the initial segments of an infinite sequence.
\end{remark}

\begin{remark}
Many authors, including \citet{Hoeffding},
define $U_n$ by dividing
the sum in \eqref{U}
  by $\binom n\ell$,
  the number of terms in it.
  We find it more convenient for our purposes
to use the unnormalized version above.
\end{remark}

\begin{remark}\label{Rasymm}
  Many authors, including \citet{Hoeffding},
assume that $f$ is a symmetric function of its $\ell$ variables.
In this case, the order of the variables does not matter, and we can in
\eqref{U} sum
over all sequences $i_1,\dots,i_\ell$ of $\ell$ distinct elements of $\setn$,
up to an obvious factor of $\ell!$. 
(\cite{Hoeffding} gives both versions.)
Conversely, if we sum over all such sequences, we may without loss of
generality
assume that $f$ is symmetric.
However, in the present paper 
(as in several earlier papers by various authors)
we consider the general case of \eqref{U}
without assuming symmetry, which we for
emphasis  call an \emph{asymmetric $U$-statistic}. 
(This is essential in our applications to pattern matching.)
Note that for independent $(X_i)_1^n$, the asymmetric case can be
reduced to the symmetric case by the trick
in \cite[Remark 11.21, in particular (11.20)]{SJIII},
see also \cite[(15)]{SJ287} and \eqref{Usym} below.
However, this trick does not work in the \mdep{} or constrained cases
studied here, so we cannot use it here.
\end{remark}

As said in the introduction, we also consider 
\emph{constrained \Ustat{s}}.
Given a constraint $\cD=(d_1,\dots,d_\elli)$, 
we define the {constrained \Ustat}
\begin{equation}  \label{Uc}
  U_n(f;\cD)
=  U_n(f;\cD;(X_i)\xooo)
:=\sum_{\substack{1\le i_1<\dots<i_\ell\le n\\i_{j+1}-i_j\le d_j}}
  f\bigpar{X_{i_1},\dots,X_{i_\ell}},
\qquad n\ge0
,\end{equation}
where we thus impose the constraints \eqref{constr} on the indices.

We define further the \emph{exactly constrained \Ustat}
\begin{equation}
  \label{Uc=}
  U_n(f;\cDq)
=  U_n(f;\cDq;(X_i)\xooo)
:=\sum_{\substack{1\le i_1<\dots<i_\ell\le n\\i_{j+1}-i_j=d_j
      \text{ if }d_j<\infty}}
  f\bigpar{X_{i_1},\dots,X_{i_\ell}},
\qquad n\ge0
,\end{equation}
where we thus specify each gap either exactly or (when $d_j=\infty$) 
not at all.
In the vincular case, when all $d_j$ are either 1 or $\infty$,
there is no difference and we have $U_n(f;\cD)=U_n(f;\cDq)$.

Note that, trivially, each constrained \Ustat{} can be written as a sum of 
exactly constrained \Ustats:
\begin{align}\label{Ucsum}
  U_n(f;\cD)=\sum_{\cD'} U_n(f;\cDz),
\end{align}
where we sum over all constraints $\cD'=(d'_1,\dots, d'_\ell)$ with
\begin{align}\label{cd'}
  \begin{cases}
1\le d_j'\le d_j,& d_j<\infty,
\\
    d'_j=\infty, & d_j=\infty.
  \end{cases}
\end{align}

\begin{remark}\label{Rincomplete2}
As said in the introduction,
  the [exactly] constrained \Ustats{} thus belong to the large class of
  \emph{incomplete \Ustats} \cite{Blom}, 
where the summation in \eqref{U} is restricted
  to some, in principle arbitrary, subset of the set of all $\ell$-tuples 
$(i_1,\dots,i_\ell)$ in $[n]$.
\end{remark}

The standard setting, in \cite{Hoeffding} and many other papers,
is to assume that the underlying random variables  $X_i$ are \iid;
we consider in the present paper a more general case, 
and we will assume only that
$X_1,X_2,\dots$ is an infinite stationary \emph{$m$-dependent} sequence, 
for some fixed integer $m\ge0$;
See \refSS{SSm} for the definition, and recall in particular that the
special case $m=0$ 
yields the case of independent variables $X_i$.

We will consider limits as \ntoo.
The sequence $X_1,X_2,\dots$ (and thus the space $\cS$ and the integer $m$)
and the function $f$ (and thus $\ell$)
will be fixed, and do not depend on $n$.

We will throughout assume the following  moment condition
for $p=2$; at a few places (always explicitly stated)
we also assume it for some larger $p$: 
\vskip6pt
\begin{PQenumerate}{\AAA{p}}
  \item\label{ALp} \qquad
$\E|f(X_{i_1},\dots,X_{i_\ell})|^p<\infty$\qquad for every $i_1<\dots<i_\ell$.
\end{PQenumerate}
\medskip
Note that in the independent case ($m=0$), it suffices to verify \ref{ALp} 
for a single sequence $i_1,\dots,i_\ell$, for example $1,\dots,\ell$.
In general, it suffices to verify \ref{ALp} for all sequences with $i_1=1$ and
$i_{j+1}-i_j\le m+1$ for every $j\le\elli$, since the stationarity and
$m$-dependence imply
that every larger gap can be
reduced to $m+1$ without changing the distribution of
$f(X_{i_1},\dots,X_{i_\ell})$. 
Since there is only a finite number of such sequences,  
it follows that
that \AAA2 is
equivalent to the uniform bound
\begin{align}
  \label{al2}
\E|f(X_{i_1},\dots,X_{i_\ell})|^2\le C 
\qquad\text{for every $i_1<\dots<i_\ell$},
\end{align}
and similarly for \ref{ALp}.

\subsection{Expectation and law of large numbers}

We first make an elementary observation on the expectations
$\E U_n(f;\cD)$ and $\E U_n(f;\cDq)$. These can  be calculated
exactly by taking the expectation inside the sums in \eqref{Uc} and 
\eqref{Uc=}. 
In the independent case, all terms have the same expectation, so it remains
only to count the number of them.
In general,
because of the $m$-dependence of $(X_i)$, the 
expectations of the terms in \eqref{Uc=} are not all equal, but
most of them coincide, and it is still easy to find the asymptotics.

\begin{theorem}
  \label{TE}
Let $(X_i)\xoo$ be a stationary \mdep{} sequence of random variables with
values in a measurable space $\cS$, let $\ell\ge1$, and let
$f:\cS^\ell\to\bbR$ satisfy 
\AAA2.
Then, as \ntoo, 
with $\mu$ given by \eqref{mu} below,
\begin{align}
  \E U_n(f)& = \binom{n}{\ell}\mu + O\bigpar{n^{\ell-1}}
= \frac{n^\ell}{\ell!}\mu + O\bigpar{n^{\ell-1}}
\label{te0}
.\end{align}
More generally,
let $\cD=(d_1,\dots,d_{\ell-1})$ be a constraint, and let $b:=b(\cD)$. 
Then, as \ntoo, 
for some real numbers $\mu_\cD$ and $\mu_{\cDq}$ given by \eqref{muD} and
\eqref{muD=}, 
\begin{align}
  \E U_n(f;\cD)& = \frac{n^b}{b!}\mu_\cD + O\bigpar{n^{b-1}},\label{te}
\\
\E U_n(f;\cDq)& = \frac{n^b}{b!}\mu_\cDq + O\bigpar{n^{b-1}}\label{te=}
.\end{align}

If $m=0$, \ie, the sequence $(X_i)\xoo$ is \iid, then, moreover,
\begin{align}\label{0mu}
  \mu&=\mu_\cDq=\E f(X_1,\dots, X_\ell),
\\\label{0muD}
\mu_\cD& = \mu\prod_{j:d_j<\infty} d_j 
=\prod_{j:d_j<\infty} d_j \cdot \E f(X_1,\dots, X_\ell)
.\end{align}
\end{theorem}
The straightforward proof is given in \refS{Smean},
where we also give formulas for $\mu_\cD$ and $\mu_\cDq$
in the general case, 
although in an application it might be simpler to find the leading term
of the expectation directly.

Next, we have a corresponding strong law of large numbers,
proved in \refS{SLLN}.
This extends well known results in  the independent case, see
\cite{Sen-LLN,Hoeffding-LLN,SJ332}.

\begin{theorem}\label{TLLN}
Let $(X_i)\xoo$ be a stationary \mdep{} sequence of random variables with
values in a measurable space $\cS$, let $\ell\ge1$, and let
$f:\cS^\ell\to\bbR$ satisfy 
\AAA2.
Then,  as \ntoo, 
with $\mu$ given by \eqref{mu},
\begin{align}
  n^{-\ell} U_n(f) &\asto  \frac{1}{\ell!}\mu \label{tlln10}
.\end{align}
More generally,
let $\cD=(d_1,\dots,d_{\ell-1})$ be a constraint, and let $b:=b(\cD)$. 
Then, as \ntoo, 
\begin{align}
  n^{-b} U_n(f;\cD) &\asto  \frac{1}{b!}\mu_\cD, \label{tlln1}
\\
  n^{-b} U_n(f;\cDq) &\asto  \frac{1}{b!}\mu_\cDq,\label{tlln1=}
\end{align}
where $\mu_\cD$ and $\mu_{\cDq}$, as in \refT{TE}, are given by \eqref{muD}
and \eqref{muD=}.

Equivalently, 
\begin{align}
  n^{-\ell} \ubbo{U_n(f)} &\asto  0,\label{tlln20}
\\
  n^{-b} \ubbo{U_n(f;\cD)} &\asto  0,\label{tlln2}
\\
  n^{-b} \bigsqpar{U_n(f;\cDq)-\E U_n(f;\cDq)} &\asto  0.\label{tlln2=}
\end{align}
\end{theorem}

\begin{remark}\label{RLLN}
For convenience, we assume \AAA2 in \refT{TLLN} as in the rest of the paper,
which leads to a simple proof. 
We conjecture that the theorem 
holds assuming only \AAA1 (i.e., finite first moments)
instead of \AAA2, as in \cite{Hoeffding-LLN,SJ332} for the
independent case.
\end{remark}

\subsection{Asymptotic normality}
We have the following theorems yielding asymptotic
normality.
The proofs are given in  \refS{Spf}. 

The first theorem is for the unconstrained case, and extends the basic
theorem by \citet{Hoeffding} for symmetric \Ustat{s} based on independent
$(X_i)\xoo$ to the  asymmetric and \mdep{} case.
Note that both these extensions have earlier been treated, but separately.
For symmetric \Ustat{s} in the
\mdep{} setting, asymptotic normality was proved by \citet{Sen-m}
(at least assuming a third moment);
moreover, bounds on the rate of convergence
(assuming a moment condition)
were given by \citet{MalevichA}.
The asymmetric case with independent $(X_i)\xoo$ has been treated \eg{}
in 
\cite[Corollary 11.20]{SJIII} and \cite{SJ332}; furthermore, 
as said in \refR{Rasymm},
for  independent $(X_i)$, the asymmetric case can be
reduced to the symmetric case by the method in 
\cite[Remark 11.21]{SJIII}.

\begin{theorem}\label{TUM}
Let $(X_i)\xoo$ be a stationary \mdep{} sequence of random variables with
values in a measurable space $\cS$, let $\ell\ge1$, and let
$f:\cS^\ell\to\bbR$ satisfy 
\AAA2.
Then, as \ntoo, 
\begin{align}\label{tum1}
  \Var\bigsqpar{U_n(f)} / n^{2\ell-1}\to\gss
\end{align}
for some $\gss=\gss(f)\in\ooo$,
and
\begin{align}\label{tum2}
  \frac{U_n(f)-\E U_n(f)}{n^{\ell-1/2}} \dto \N(0,\gss).
\end{align}
\end{theorem}

The second theorem extends \refT{TUM} to the
constrained cases.

\begin{theorem}\label{TUMD}
Let $(X_i)\xoo$ be a stationary \mdep{} sequence of random variables with
values in a measurable space $\cS$, let $\ell\ge1$, and let
$f:\cS^\ell\to\bbR$ satisfy 
\AAA2.
Let $\cD=(d_1,\dots,d_{\ell-1})$ be a constraint, and let $b:=b(\cD)$. 
Then, as \ntoo, 
\begin{align}\label{tumd1}
  \Var\bigsqpar{U_n(f;\cD)} / n^{2b-1}\to\gss
\end{align}
  for some $\gss=\gss(f;\cD)\in\ooo$,
and
\begin{align}\label{tumd2}
  \frac{U_n(f;\cD)-\E U_n(f;\cD)}{n^{b-1/2}} \dto \N(0,\gss).
\end{align}

The same holds, with some (generally different) $\gss=\gss(f;\cDq)$,
for the exactly constrained
$U_n(f;\cDq)$.
\end{theorem}

\begin{remark}\label{RCW}
It follows immediately by
the Cram{\'e}r--Wold device  \cite[Theorem 5.10.5]{Gut}
(i.e., considering linear combinations),
that \refT{TUM} extends in the obvious way
to joint convergence for any finite number of 
different $f:\cS^\ell\to\bbR$,
with $\gss$ now a covariance matrix. 
Moreover, the proof shows that this holds also for a family 
of different $f$ with (possibly) different $\ell\ge1$.

Similarly, \refT{TUMD} extends 
to joint convergence for any finite number of 
different $f$ (possibly with different $\ell$ and $\cD$);
this follows by the proof below, which reduces the results to \refT{TUM}.
\end{remark}

\begin{remark}
The asymptotic variance $\gss$ in \refTs{TUM} and \ref{TUMD} can  
be calculated explicitly, see \refR{Rgss}.
\end{remark}

\begin{remark}
  Note that it is possible that the asymptotic variance $\gss=0$ in
  \refTs{TUM} and \ref{TUMD}; in this case, \eqref{tum2} and \eqref{tumd2} 
just give convergence in probability to 0.
This degenerate case is discussed in \refS{S0}.
\end{remark}

\begin{remark}
We do not 
consider extensions to triangular arrays
where $f$ or $X_i$ (or both) depend on $n$.
In the symmetric \mdep{} case,
such a result (with fixed $\ell$ but possibly increasing $m$, under suitable
conditions)
has been shown by \cite{MalevichA}, with a bound on the rate of
convergence.
In the independent case, 
results for triangular arrays
are given by \eg{} \cite{RubinVitale} and \cite{SJ48};
see also \cite{SJ349} for the special case of substring counts
$N_n(\bw)$ with $\bw$ depending on $n$ (and growing in length).
It seems to be an interesting (and challenging) open problem
to formulate useful general theorems
for constrained \Ustats{} in such settings.
\end{remark}

\subsection{Rate of convergence}\label{SSrate}

Under stronger moment assumptions on $f$,
an alternative method of proof 
(suggested by a referee)
yields the asymptotic normality 
in \refTs{TUM} and \ref{TUMD}
together with an upper bound
on the rate of convergence, provided $\gss>0$.

In the following theorem of Berry--Esseen type 
we assume, for simplicity,  that $f$ is bounded
(as it is in our applications in \refSs{Sword}--\ref{Sperm});
see further \refR{Rrate1}.
Let $d_K$ denote the Kolmogorov distance between distributions; 
recall that for two distributions $\cL_1,\cL_2$
with distribution functions $F_1(x)$ and $F_2(x)$,
$d_K=d_K(\cL_1,\cL_2):=\sup_x|F_1(x)-F_2(x)|$; we use also the notation 
$d_K(X,\cL_2):=d_K(\cL(X),\cL_2)$ for a random variable $X$.

\begin{theorem}\label{Trate}
  Suppose in addition to the hypotheses in \refT{TUM} or \ref{TUMD} that 
$\gss>0$ and that $f$ is bounded. Then,
\begin{align}\label{trate}
d_K\Bigpar{\frac{U_n-\E U_n}{\sqrt{\Var U_n}}, \N(0,1)}
=O\bigpar{n\qqw},
\end{align}
where $U_n$ denotes $U_n(f)$, $U_n(f,\cD)$ or $U_n(f;\cDq)$.
\end{theorem}

In the symmetric and unconstrained case, this (and more) was shown
by \citet{MalevichA}.
The proof of \refT{Trate}
is given in \refS{Srate}, together with further remarks.

\subsection{Moment convergence}

\refTs{TUM} and \ref{TUMD} include convergence of the first (trivially)
and second moments in \eqref{tum2} and \eqref{tumd2}.
This extends to higher moments under a corresponding moment condition on $f$.
(The unconstrained case with independent $X_i$
was shown in \cite[Theorem 3.15]{SJ332}.)

\begin{theorem}\label{Tmom}
Suppose in addition to the hypotheses in \refTs{TUM} or \ref{TUMD} 
that \AAA{p} holds for some real $p\ge2$.
Then all absolute and ordinary moments of order up to $p$ converge 
in \eqref{tum2} or \eqref{tumd2}.
\end{theorem}

The proof is given in \refS{Smom},
where we also give related
estimates for maximal functions.

\subsection{Functional limit theorems}\label{SSfun}

We can extend \refT{TUMD} to functional convergence.
For unconstrained \Ustat{s},
this was done by \citet{MillerSen}
in the classical case of independent $X_i$ and symmetric $f$;
the asymmetric case is \cite[Theorem 3.2]{SJ332};
furthermore, \citet{Yoshihara92} proved the case of dependent $X_i$
satisfying a suitable mixing condition 
(assuming a technical condition on $f$ besides symmetry).

\begin{theorem}  \label{TG}
Suppose that \AAA2 holds.
Then as \ntoo, with $b=b(\cD)$,
in $D\ooo$,
\begin{align}\label{tg1}
  \frac{U\nt(f;\cD)-\E U\nt(f;\cD)}{n^{b-1/2}}\dto Z(t),\qquad t\ge0,
\end{align}
where $Z(t)$ is a continuous centered Gaussian process.
Equivalently, in $D\ooo$, 
\begin{align}\label{tg2}
  \frac{U\nt(f;\cD)-(\mu_\cD/b!)n^bt^b}{n^{b-1/2}}&\dto Z(t),\qquad t\ge0
.\end{align}

The same holds for exact constraints.
Moreover, the results hold jointly for any finite set of $f$ and $\cD$
(possibly with different $\ell$ and $b$), with limits $Z(t)$ depending on
$f$ and $\cD$.
\end{theorem}

The proof is given in \refS{Sfun}.

\begin{remark}\label{RZ}
A comparison between \eqref{tg1} and \eqref{tumd2} yields
$Z(t)\sim \N(0,t^{2b-1}\gss)$, 
with $\gss$ as in \refT{TUMD}.
Equivalently, $\Var Z(t) = t^{2b-1}\gss$, which can be calculated by \refR{Rgss}.
Covariances $\Cov\bigpar{Z(s),Z(t)}$ 
can be calculated by the same method
and \eqref{Zt} in the proof; we leave the details to the reader.
Note that these covariances
determine the distribution of the process $Z$.
\end{remark}

\subsection{Renewal theory}\label{SSrenew}

Assume further that $h:\cS\to\bbR$ is
another (fixed) measurable function, with
\begin{align}\label{nu}
  \nu:=\E h(X_1)>0.
\end{align}
We define
\begin{align}\label{Sn}
  S_n=S_n(h):=\sumin h(X_i),
\end{align}
and, for $x>0$,
\begin{align}
  N_-(x)&:=\sup\set{n\ge0:S_n\le x},\label{N-}
\\
  N_+(x)&:=\inf\set{n\ge0:S_n> x}.\label{N+}
\end{align}
$N_-(x)$ and $N_+(x)$ are finite a.s.\ by
the law of large numbers for $S_n$ \eqref{llnh}; see further \refL{LN}.
We let $N_\pm(x)$ denote either $N_-(x)$ or $N_+(x)$, in statements and
formulas that are valid for both.

\begin{remark}\label{Rhf}
  In \cite{SJ332}, we consider instead of $h(x)$, more generally, a function
  of several variables, and define $N_\pm$ using the corresponding \Ustat{}
  instead of $S_n$. We believe that the results of the present paper can be
  extended  to that setting, but we have not pursued this,
and leave it as an open problem.
\end{remark}

\begin{remark}\label{R>=}
If $h(X_1)\ge0$ a.s., which often is assumed in renewal theory, 
then $N_+(x)=N_-(x)+1$. However, if $h$ may be negative (still assuming
\eqref{nu}), then $N_-(x)$ may be larger than $N_+(x)$.
Nevertheless, the difference is typically small, and we obtain the same
asymptotic results for both $N_+$ and $N_-$.
(We can also obtain the same results if we instead use $S_n<x$ or $S_n\ge x$
in the definitions.)
\end{remark}

In this situation, we have the following limit theorems, which extend
results in \cite{SJ332}.
Proofs are given in \refS{Srenew}.
For an application, see \cite{SJ367}.

\begin{theorem}\label{TUUN}
With the assumptions and notations of \refT{TUMD}, 
assume \AAA2, and
suppose also that
$\mux:=\E h(X_1)>0$
and $\E h(X_1)^2<\infty$.
Then, with notations as above, 
as \xtoo,
\begin{equation}\label{cvtauu}
\frac{U_{N_\pm(x)}(f;\cD)-\mu_\cD{\mux}^{-b}{b!}\qw x^{b}}
{x^{b-1/2}} 
\dto \N\bigpar{0,\gamx^2},
\end{equation}
for some $\gamxx=\gamxx(f;h;\cD)\ge0$.

The same holds for exact constraints.
Moreover, the results hold jointly for any finite set of $f$ and $\cD$
(possibly with different $\ell$ and $b$).
\end{theorem}

\begin{theorem} \label{TUUN2}
Suppose in addition to the hypotheses in \refT{TUUN} that $h(X_1)$ is
integer-valued and that $(X_i)\xoo$ are independent.
Then \eqref{cvtauu} holds also conditioned on $S_{N_-(x)}=x$ 
for integers $x\to\infty$.
\end{theorem}
We consider here tacitly only $x$ such that
$\P\bigpar{S_{N_-(x)}=x}>0$.

\begin{remark}
 We prove \refT{TUUN2} only for independent $X_i$ (which, in any
case, is our main interest as said in the introduction.)
It seems likely that the result can be extended to at least some
\mdep{} $(X_i)$,
using a modification of the proof below and the \mdep{} renewal theorem
(under some conditions) 
\cite[Corollary 4.2]{AlsmeyerH01}, but we have not pursued this.
\end{remark}

\begin{theorem}\label{TUUNp}
Suppose in addition to the hypotheses in \refT{TUUN} that \AAA{p} holds and
$\E\bigsqpar{|h(X_1)|^p}<\infty$ for every $p<\infty$.
Then all moments converge in \eqref{cvtauu}.

Under the additional hypothesis in \refT{TUUN2}, this holds also conditioned
on  $S_{N_-(x)}=x$.  
\end{theorem}

\begin{remark}
In \refT{TUUNp}, unlike \refT{Tmom}, we assume $p$-th moments for all $p$,
and conclude convergence of all moments. If we only want to show convergence
for a given $p$, some sufficient moment conditions on $f$ and $h$ can be
derived from the proof, but we do not know any sharp results and have not
pursued this.
Cf.\ \cite[Remark 6.1]{SJ332} and the references there.
\end{remark}

\section{Some lemmas}\label{Slem}
We give here some lemmas that will be used in the proofs in later sections.
In particular, they will enable us to reduce the constrained cases to the
unconstrained one.

Let $\cD=(d_1,\dots,d_\elli)$ be a given constraint.
Recall that $b=b(\cD)$ is given by \eqref{b},
and let $1=\gb_1<\dots<\gb_b$ be the indices in $[\ell]$ 
just after the unconstrained gaps; in other words, $\gb_j$ are defined by
$\gb_1:=1$ and $d_{\gb_j-1}=\infty$ for $j=2,\dots,b$.
For convenience we also define   $\gb_{b+1}:=\ell+1$.
We say that the constraint $\cD$ separates the index set $[\ell]$ 
into the $b$ \emph{blocks} $B_1,\dots,B_b$,
where $B_k:=\set{\gb_k,\dots,\gb_{k+1}-1}$.
Note that the constraints \eqref{constr} thus are constraints on $i_j$ for
$j$ in each block separately.

\begin{lemma}
  \label{L1}
Let $(X_i)\xoo$ be a stationary \mdep{} sequence of random variables with
values in $\cS$, let $\ell\ge1$, and let $f:\cS^\ell\to\bbR$ satisfy
\AAA2.
Let $\cD=(d_1,\dots,d_{\ell-1})$ be a constraint.
Then
\begin{align}\label{l1}
  \Var \bigsqpar{U_n(f;\cD)} = O\bigpar{n^{2b(\cD)-1}},
\qquad n\ge1.
\end{align}
Furthermore,
\begin{align}\label{l1=}
  \Var \bigsqpar{U_n(f;\cD)-U_{n-1}(f;\cD)} = O\bigpar{n^{2b(\cD)-2}},
\qquad n\ge1.
\end{align}

Moreover, the same estimates hold for $U_n(f;\cDq)$.
\end{lemma}

\begin{proof}
The definition \eqref{Uc} yields
  \begin{align}\label{b1}
\Var\bigsqpar{U_n(f;\cD)}&
=\sum_{\substack{1\le i_1<\dots<i_\ell\le n\\i_{k+1}-i_k\le d_k}}
\sum_{\substack{1\le j_1<\dots<j_\ell\le n\\j_{k+1}-j_k\le d_k}}
\Cov\bigpar{f\bigpar{X_{i_1},\dots,X_{i_\ell}},f\bigpar{X_{j_1},\dots,X_{j_\ell}}}
.  \end{align}
  Let $\dx$ be the largest finite $d_j$ in the constraint $\cD$, \ie,
  \begin{align}\label{D}
\dx:=\max_j \set{d_j:d_j<\infty}.    
  \end{align}
The constraints imply that for each block $B_q$ and all  indices $k\in B_q$, 
coarsely,
\begin{align}\label{b2}
0\le i_k-i_{\gb_q}\le \dx\ell
\qquad\text{and}
\qquad 0\le j_k-j_{\gb_q}\le \dx\ell
.\end{align}
It follows that if $|i_{\gb_r}-j_{\gb_s}|>\dx\ell+m$ for all $r,s\in[b]$,
then $|i_\ga-j_\gb|>m$ for all $\ga,\gb\in[\ell]$. 
Since $(X_i)\xoo$ is \mdep, this implies
that the two random vectors
$\bigpar{X_{i_1},\dots,X_{i_\ell}}$ and $\bigpar{X_{j_1},\dots,X_{j_\ell}}$
are independent, and thus the corresponding term in \eqref{b1} vanishes.

Consequently, we only have to consider terms in the sum in \eqref{b1}
such that 
\begin{align}\label{b3}
|i_{\gb_r}-j_{\gb_s}|\le \dx\ell+m  
\end{align}
for some $r,s\in[b]$. 
For each of the $O(1)$ choices of $r$ and $s$,
we can choose $i_{\gb_1},\dots,i_{\gb_b}$ in at most $n^b$ ways; then $j_{\gb_s}$
in $O(1)$ ways such that \eqref{b3} holds; then the remaining $j_{\gb_q}$
in $O(n^{b-1})$ ways; then, finally, all remaining $i_k$ and $j_k$ in $O(1)$
ways because of \eqref{b2}.  
Consequently, the number of non-vanishing terms in \eqref{b1} is
$O(n^{2b-1})$.
Moreover, each term is $O(1)$ by \eqref{al2} and the \CSineq, and thus
\eqref{l1} follows.

For \eqref{l1=}, we note that $U_n(f;\cD)-U_{n-1}(f;\cD)$ is the sum in
\eqref{Uc} with the extra restriction $i_\ell=n$.
Hence, its variance can be expanded as in \eqref{b1}, with the extra
restrictions $i_\ell=j_\ell=n$.
We then argue as above, but note that \eqref{b2} and $i_\ell=n$ imply that
there are only $O(1)$ choices of $i_b$, and hence $O(n^{b-1})$ choices of
$i_1,\dots,i_b$. We thus obtain $O\bigpar{n^{2b-2}}$ non-vanishing terms in
the sum, and \eqref{l1=} follows.

The argument for the exactly constrained $U_n(f;\cDq)$ is the same (and
slightly simpler). (Alternatively, we could do this case first, and
then use \eqref{Ucsum} to obtain the results for $U_n(f;\cD)$.)
\end{proof}

The next lemma is the central step in the reduction to the unconstrained case.

\begin{lemma}
  \label{L2}
Let $(X_i)\xoo$, $f:\cS^\ell\to\bbR$, 
and  $\cD=(d_1,\dots,d_{\ell-1})$ be
as in \refL{L1},  
and let
\begin{align}\label{ld}
  D:=\sum_{j:d_j<\infty}d_j.
\end{align}
Let $M>D$ and define 
\begin{align}\label{YX}
Y_i:=(X_i,X_{i+1},\dots,X_{i+M-1})\in \cS^M ,
\qquad i\ge1
.\end{align}
Then there exists a function $g=g_{\cDq}:(\cS^M)^b\to\bbR$ such that
for every $n\ge 0$,
\begin{align}\label{lx}
U_n(f;\cDq;(X_i)\xooo) = \sum_{j_1<\dots<j_b\le n-D}g\bigpar{Y_{j_1},\dots,Y_{j_b}}
=U_{n-D}\bigpar{g;(Y_i)\xooo},
\end{align}
with $U_{n-D}(g):=0$ when $n<D$.
Furthermore,
\begin{align}\label{ly}
  \E \bigabs{g\bigpar{Y_{j_1},\dots,Y_{j_b}}}^2<\infty,
\end{align}
for every $j_1<\dots<j_b$.
\end{lemma}

\begin{proof}
  For each block $B_q=\set{\gb_q,\dots,\gb_{q+1}-1}$ defined by $\cD$,
let 
\begin{align}\label{jedl}
\ell_q&:=|B_q|=\gb_{q+1}-\gb_q,
\\\label{jedt}
t_{qr}&:=\sum_{j=1}^{r-1}d_{\gb_q+j-1},
\qquad r=1,\dots,\ell_q,
\\\label{jedu}
u_q&:=t_{q,\ell_q}=\sum_{j=1}^{\gb_{q+1}-\gb_q-1}d_{\gb_q+j-1},
\\\label{jedv}
v_q&:=\sum_{k<q}u_k
.\end{align}
Note that $t_{q1}=0$ for every $q$ and that $t_{qr},u_q<\infty$. (We stop
the summation in \eqref{jedu} just before the next infinite $d_j$, which
occurs for 
$j=\gb_{q+1}-1$ provided $q<b$.) 
Note also that 
\begin{align}
  \label{yoda}
u_b+v_b=\sum_{k\le b} u_k=D.
\end{align}
We then rewrite \eqref{Uc=} as,
letting
$\nuk_q:=i_{\gb_q}$ and
grouping the arguments of $f$ according to the blocks of $\cD$
(using an obvious notation for this),
\begin{align}\label{luke}
U_n(f;\cDq)
&= \sum_{\substack{1\le \nuk_1<\nuk_2<\dots<\nuk_b\le n-u_{b},\\ \nuk_{q+1}>\nuk_q+u_q}}
f\bigpar{(X_{\nuk_1+t_{1r}})_{r=1}^{\ell_1} ,\dots,(X_{\nuk_b+t_{br}})_{r=1}^{\ell_b} 
}
.\end{align}
Change summation variables by  $\nuk_q=j_q+v_q$.
Then \eqref{luke} yields, recalling \eqref{jedv}--\eqref{yoda},
\begin{align}\label{leia}
U_n(f;\cDq)
= \sum_{1\le j_1<j_2<j_b\le n-D}
f\bigpar{(X_{j_1+v_1+t_{1r}})_{r=1}^{\ell_1} ,\dots,(X_{j_b+v_b+t_{br}})_{r=1}^{\ell_b} 
}
.\end{align}
Define, for
$y_i=(y_{ik})_{k=1}^M\in \cS^M$,
\begin{align}\label{han}
  g(y_1,\dots,y_b)
=f\bigpar{(y_{1,v_1+t_{1r}+1})_{r=1}^{\ell_1} ,\dots,(y_{b,v_b+t_{br}+1})_{r=1}^{\ell_b} }.
\end{align}
(Note that 
$v_j+t_{jr}+1\le v_j + u_j +1 \le D+1\le M$.)
We have $Y_j=(X_{j+k-1})_{k=1}^M$, and thus \eqref{han} yields
 \begin{align}\label{solo}
g(Y_{j_1},\dots,Y_{j_b})
=f\bigpar{(X_{j_1+v_1+t_{1r}})_{r=1}^{\ell_1} ,\dots,(X_{j_b+v_b+t_{br}})_{r=1}^{\ell_b}} .
 \end{align}
Consequently, \eqref{lx} follows from
 \eqref{leia} and \eqref{solo}.

Furthermore, \eqref{ly} follows from \eqref{solo} and \AAA2.
\end{proof}

\begin{lemma}
  \label{L3}
Let $(X_i)\xoo$
and  $\cD=(d_1,\dots,d_{\ell-1})$ be
as in \refL{L1},
and let $M$ and $Y_i$ be as in \refL{L2}.
For every $f:\cS^\ell\to\bbR$
such that \AAA2 holds,
there  exist  functions $g_{\cD},g_{\cDq}:(\cS^M)^b\to\bbR$ such that
\eqref{ly} holds for both, and
\begin{align}\label{obi}
\Var\Bigsqpar{U_n\bigpar{f;\cD ;(X_i)\xooo}-U_{n}\bigpar{g_{\cD};(Y_i)\xooo}}
&=O\bigpar{n^{2b(\cD)-2}},
\\\label{wan}
\Var\Bigsqpar{U_n(f;\cDq ;(X_i)\xooo)-U_{n}\bigpar{g_{\cDq};(Y_i)\xooo}}
&=O\bigpar{n^{2b(\cD)-2}}
.\end{align}
\end{lemma}
\begin{proof}
  First, letting $g_{\cDq}$ be as in \refL{L2}, we have by \eqref{lx},
  \begin{align}\label{mcduff}
    U_n(f;\cDq ;(X_i)\xooo)-U_{n}\bigpar{g_{\cDq};(Y_i)\xooo}
&=
U_{n-D}\bigpar{g_{\cDq}}-U_{n}\bigpar{g_{\cDq}}
\notag\\&
=-\sum_{k=1}^q\Bigpar{ U_{n-k+1}\bigpar{g_{\cDq}}-U_{n-k}\bigpar{g_{\cDq}}}
.  \end{align}
Thus \eqref{wan} follows by \eqref{l1=} in \refL{L1} applied to $g_{\cDq}$, 
the trivial constraint $\cDoo$
 (\ie, no constraint),
and
$(Y_i)\xoo$. 

Next, we recall \eqref{Ucsum} and define
\begin{align}\label{glamis}
  g_{\cD}= 
\sum_{\cD'} g_{\cDz},
\end{align}
again summing over all constraints $\cD'$ satisfying \eqref{cd'}.
This is a finite sum, 
and by \eqref{Ucsum} and \eqref{glamis},

\begin{align}\label{cawdor}
U_n(f;\cD ;(X_i)\xooo)-U_{n}\bigpar{g_{\cD};(Y_i)\xooo}
= \sum_{\cD'}
\bigpar{U_n(f;\cDz ;(X_i)\xooo)-U_{n}\bigpar{g_{\cDz};(Y_i)\xooo}}
\end{align}
and thus \eqref{obi} follows from \eqref{wan}.
\end{proof}

To avoid some of the problems caused by dependencies between the
$X_i$,
we follow \citet{Sen-m} and introduce
another type of constrained \Ustat{s}, where we require the
gaps beteen the summation indices to be large, instead of small as in
\eqref{Uc}. 
We need only one case, and  define
\begin{equation}  \label{Uc>}
  U_n(f;\mx)
:=\sum_{\substack{1\le i_1<\dots<i_\ell\le n\\i_{j+1}-i_j>m}}
  f\bigpar{X_{i_1},\dots,X_{i_\ell}},
\qquad n\ge0
,\end{equation}
summing only over terms where all gaps $i_{j+1}-i_j>m$, $j=1,\dots,\ell-1$.
(The advantage is that in each term in \eqref{Uc>}, the variables
$X_{i_1},\dots,X_{i_\ell}$ are independent.)


\begin{lemma}
  \label{L>}
Let $(X_i)\xoo$ and $f:\cS^\ell\to\bbR$
be as in \refL{L1}.
Then,
\begin{align}\label{l4}
  \Var\bigpar{U_n(f)-U_n(f;\mx)}
=O\bigpar{n^{2\ell-3}}.
\end{align}
\end{lemma}

\begin{proof}
  We can express the type of constrained \Ustat{} in \eqref{Uc>}
as a combination of
constrained \Ustat{s} of the previous type
by the following inclusion--exclusion argument:
\begin{align}
    \label{bada}
  U_n(f;\mx)
&=\sum_{1\le i_1<\dots<i_\ell\le n}
  f\bigpar{X_{i_1},\dots,X_{i_\ell}} \prod_{j=1}^\elli\indic{i_{j+1}-i_j>m}
\notag\\
&=\sum_{1\le i_1<\dots<i_\ell\le n}  f\bigpar{X_{i_1},\dots,X_{i_\ell}}
  \prod_{j=1}^\elli\Bigpar{1-\indic{i_{j+1}-i_j\le m}}
\notag\\
&=\sum_{J\subseteq[\ell-1]}(-1)^{|J|}\sum_{1\le i_1<\dots<i_\ell\le n}
  f\bigpar{X_{i_1},\dots,X_{i_\ell}} \prod_{j\in J}\indic{i_{j+1}-i_j\le m}
\notag\\
&=\sum_{J\subseteq[\ell-1]}(-1)^{|J|}U_n(f;\cD_J),
\end{align}
where we sum over the $2^\elli$ subsets $J$ of $[\elli]$, and use the
  constraints
  \begin{align}\label{cdj}
    \cD_J:=(d_{Jj})_{j=1}^\elli
\qquad \text{with}
\qquad d_{Jj}=
    \begin{cases}
      m,&j\in J,\\ 
      \infty,& j\notin J.
    \end{cases}
  \end{align}
We have $b(\cD_{J})=\ell-|J|$, and thus $b(\cD_J)<\ell$ unless $J=\emptyset$.
Moreover, $\cD_\emptyset=(\infty,\dots,\infty)=\cDoo$, 
and thus means no constraint,
so $U_n(f;\cD_\emptyset)=U_n(f)$, the unconstrained \Ustat.
Consequently, by \eqref{bada} and \refL{L1},
\begin{align}\label{oy}
  \Var\bigpar{U_n(f)-U_n(f;\mx)}
=\Var\Bigpar{\sum_{J\neq\emptyset}(-1)^{|J|-1}U_n(f;\cD_J)}
=O\bigpar{n^{2\ell-3}},
\end{align}
which proves the estimate \eqref{l4}.
\end{proof}

\subsection{Triangular arrays}\label{SStri}
We will also use a central limit theorem for 
\mdep{} triangular arrays satisfying the Lindeberg condition,
 which we  state as \refT{TM} below.
The theorem is implicit in \citet{Orey}; it follows from his theorem there
exactly as his corollary, which however is stated for a sequence and not for
a triangular array.
See also \citet[Theorem 2.1]{Peligrad}, which 
contains the theorem below (at
least for $\gss>0$; the case $\gss=0$ is trivial),
and is much more general in that it only assumes
strong mixing instead of $m$-dependence.

Recall that
a \emph{triangular array} is an array
$(\xi_{ni})_{1\le i\le n<\infty}$ of random variables, such that the
variables $(\xi_{ni})_{i=1}^n$ in a single row
are defined on a common probability space. 
(As usual, it is only for convenience that we require that the $n$th row has
length $n$; the results extend to arbitrary lengths $N_n$.)
We are here mainly interested in the case when each row is an \mdep{}
sequence;
in this case, we say that $(\xi_{ni})$ is an \emph{\mdep{} triangular array}.
(We make no assumption on the relation between variables in different rows;
these may even be defined on different probability spaces.)

\begin{theorem}[\citet{Orey}]\label{TM}
  Let $(\xi_{ni})_{1\le i\le n<\infty}$ be an \mdep{} triangular array of
  real-valued random variables with
  $\E\xi_{ni}=0$. 
Let $\TS_n:=\sumin\xi_{ni}$. Assume that, as \ntoo,
\begin{align}\label{m1}
  \Var \TS_n \to\gss\in\ooo,
\end{align}
that $\xi_{ni}$ satisfy the Lindeberg condition
\begin{align}\label{m2}
  \sumin \E\bigsqpar{\xi_{ni}^2\indic{|\xi_{ni}|>\eps}} \to 0,
\qquad\text{for every $\eps>0$},
\end{align}
and that
\begin{align}\label{m3}
  \sumin \Var \xi_{ni} = O(1).
\end{align}
Then, as \ntoo,
\begin{align}
  \label{mm}
\TS_n\dto \N(0,\gss).
\end{align}
\nopf
\end{theorem}

Note that \refT{TM} extends the standard Lindeberg--Feller 
central limit theorem for triangular arrays with row-wise independent variables
(see \eg{} \cite[Theorem 7.2.4]{Gut}), to which it reduces when $m=0$.

\begin{remark}
In fact, the assumption \eqref{m3} is not needed in \refT{TM}, see
\cite{SJ361}. However, it is easily verified in our case 
(and many other applications), so we need only this classical result.
\end{remark}

\section{The expectation}\label{Smean}

The expectation of a (constrained) \Ustat{s}, and in particular its leading
term, is easily found from the definition. Nevertheless, we give a detailed
proof of \refT{TE}, for completeness and for later reference.

\begin{proof}[Proof of \refT{TE}]
  Consider first the unconstrained case. We take expectations in \eqref{U}.
The sum in \eqref{U} has
$\binom{n}{\ell}$ terms.
We consider first the terms that satisfy the restriction
$i_{j+1}> i_j+m$ for every $j\in[\elli]$. (I.e., the terms in \eqref{Uc>}.)
As noted above, in each such term, 
the variables $X_{j_1},\dots,X_{j_\ell}$ are independent.
Hence, let 
$(\hX_i)_1^\ell$ be an \emph{independent} sequence of random variables in
$\cS$, each with the same distribution as $X_1$ (and thus as each $X_j$), and
define
\begin{equation}\label{mu}
  \mu:=\E f(\hX_1,\dots,\hX_\ell).
\end{equation}
Then
\begin{align}\label{mu2}
  \mu = \E f\bigpar{X_{i_1},\dots,X_{i_\ell}}
\end{align}
for every sequence of indices $i_1,\dots,i_\ell$ with $i_{j+1}>i_j+m$ for all
$j\in[\elli]$. 
Moreover, the number of terms in \eqref{U} that do \emph{not} satisfy these
constraints is $O\bigpar{n^{\elli}}$,
and their expectations are uniformly $O(1)$ as a
consequence of \eqref{al2}.
Thus, \eqref{te0} follows from \eqref{U}.

Next, consider the exactly constrained case.
We use \refL{L2}
and then apply the unconstrained case just treated 
to $g$ and $(Y_i)\xooo$;
this yields
\begin{align}\label{eu=3}
 \E U_n\bigpar{f;\cDq}
=
\E U_{n-D}\bigpar{g;(Y_i)\xooo}
=\binom{n-D}{b} \E g(\hY_1,\dots,\hY_b)+ O\bigpar{n^{b-1}}
\end{align}
with $\hY_1,\dots,\hY_b\eqd Y_1$ independent. Using \eqref{solo}, and the
notation there, this
yields \eqref{te=} with
 \begin{align}\label{muD=}
\mu_{\cDq}:=\E g(Y_{j_1},\dots,Y_{j_b})
=\E f\bigpar{(X_{j_1+v_1+t_{1r}})_{r=1}^{\ell_1},\dots,(X_{j_b+v_b+t_{br}})_{r=1}^{\ell_b}} 
, \end{align}
for any sequence $j_1,\dots,j_b$ with $j_{k+1}-j_k\ge m+M$ for all $k\in[b-1]$.
(Note that $(Y_i)\xoo$ is $(m+M-1)$-dependent.)

Finally, the constrained case \eqref{te} follows by \eqref{te=} and the
decomposition \eqref{Ucsum}, with
\begin{align}\label{muD}
  \mu_{\cD}:=\sum_{\cD'}\mu_{\cDz},
\end{align}
summing over all $\cD'$ satisfying \eqref{cd'}.

In the independent case $m=0$, the results above simplify. First, for the
unconstrained case, the formula for $\mu$ in \eqref{0mu} is a special case
of \eqref{mu2}.  
Similarly, in the exactly unconstrained case, \eqref{muD=} yields the
formula for $\mu_\cDq$ in \eqref{0mu}. Finally, \eqref{0mu} shows that
$\mu_\cDq$ does not depend on $\cD$, and thus all terms in the sum in
\eqref{muD} are equal to $\mu$. Furthermore, it follows from \eqref{cd'}
that the number of terms in the sum is $\prod_{d_j<\infty}d_j$, and
\eqref{0muD} follows.

Alternatively, in the independent case, all terms in the sums in \eqref{U},
\eqref{Uc} and \eqref{Uc=} have the same expectation $\mu$ given by
\eqref{0mu}, and the result follows by counting the number of terms. In
particular, exactly,
\begin{align}
  \E U_n(f)=\binom n\ell \mu
\end{align}
and, with $D$ given by \eqref{ld},
\begin{align}
  \E U_n(f;\cDq)=\binom {n-D}{b} \mu
.\end{align}
\end{proof}

\section{Asymptotic normality}\label{Spf}

The general idea to prove \refT{TUM} 
is to use the projection method by \citet{Hoeffding},
together with  modifications as in \cite{Sen-m} to treat $m$-dependent variables
and modifications as in  \eg{} \cite{SJ332} to treat the asymmetric case.
We then obtain the constrained version \refT{TUMD} by reduction to the
unconstrained case.

\begin{proof}[Proof of \refT{TUM}]
We first note that by \refL{L>},
it suffices to prove \eqref{tum1}--\eqref{tum2} for $U_n(f;\mx)$.
(This uses standard arguments with Minkowski's inequality and
Cram\'er--Slutsky's theorem \cite[Theorem 5.11.4]{Gut}, respectively; 
we omit the details. The same arguments are used several times below without
comment.) 

As commented above, 
the variables inside each term in the sum in \eqref{Uc>} are independent;
this enables us to use Hoeffding's decomposition for the independent case, 
which we (in the present, asymmetric case) define as follows. 

As in \refS{Smean},
let $(\hX_i)_1^\ell$ be an {independent} sequence of random variables in
$\cS$, each with the same distribution as $X_1$. 
Recall $\mu$ defined in \eqref{mu}, and, for $i=1,\dots,\ell$,
define the function $f_i$ as the one-variable projection
\begin{align}
f_i(x):
= &\E f\bigpar{\hX_1,\dots,\hX_{i-1},x,\hX_{i+1},\dots,\hX_\ell}-\mu.
\label{fi1}
\intertext{
Equivalently,
}
  f_i(\hX_i)
=&\E \bigpar{f(\hX_1,\dots,\hX_\ell)\mid \hX_i}-\mu 
\label{fi2}
.\end{align}
(In general, $f_i$ is defined only $\cL(\hX_i)$-\aex, but it does not matter which
version we choose.)
Define also the residual function
\begin{align}
\fx(x_1,\dots,x_d)
&:=
f(x_1,\dots,x_d) - \mu - \sumjl f_j(x_j).
\label{hi}
\end{align}

Note that the variables $f_i(X_j)$ are centered by \eqref{mu} and \eqref{fi2}:
\begin{align}\label{h0}
\E f_i(X_j)=\E f_i(\hX_i)=0.
\end{align}
Furthermore, \AAA2 implies that 
$f_i(\hX_i)$, and thus  each $f_i(X_j)$,
is square integrable.

The essential property of $\fx$ is that, 
as an immediate consequence of the definitions and \eqref{h0}, 
its one-variable projections vanish:
\begin{align}
\E \bigpar{\fx(\hX_1,\dots,\hX_\ell)\mid \hX_i=x}
= \E \fx\bigpar{\hX_1,\dots,\hX_{i-1},x,\hX_{i+1},\dots,\hX_\ell}
=0
\label{fi0}
.\end{align}

We assume from now on for simplicity that $\mu=0$; the general case follows
by replacing $f$ by $f-\mu$.
Then \eqref{Uc>} and \eqref{hi} yield, by counting the terms where $i_j=k$
for given $j$ and $k$,
\begin{align}\label{maya}
  U_n(f;\mx)
&=
\sum_{\substack{1\le i_1<\dots<i_\ell\le n\\i_{j+1}-i_j>m}}
 \Bigpar{\sumjl f_j(X_{i_j})+ \fx\bigpar{X_{i_1},\dots,X_{i_\ell}}}
\notag\\&
=\sumjl \sumkn \binom{k-1-(j-1)m}{j-1}\binom{n-k-(\ell-j)m}{\ell-j} f_j(X_k)
+U_n(\fx;\mx)
.\end{align}

Let us first dispose of the last term in \eqref{maya}.
Let $i_1<\dots<i_\ell$ and $j_1<\dots<j_\ell$ be two sets of indices such that
the constraints $i_{k+1}-i_k>m$ and $j_{k+1}-j_k>m$ in \eqref{Uc>} hold.
First, as in the proof of \refL{L1},
if also $|i_\ga-j_\gb|>m$ for all $\ga,\gb\in[\ell]$, then
all $X_{i_\ga}$ and $X_{j_\gb}$ are independent; thus
$\fx(X_{i_1},\dots,X_{i_\ell})$ and $\fx(X_{j_1},\dots,X_{j_\ell})$ are
independent, and
\begin{align}\label{tun}
  \E\bigsqpar{\fx(X_{i_1},\dots,X_{i_\ell})\fx(X_{j_1},\dots,X_{j_\ell})}
=  \E \fx(X_{i_1},\dots,X_{i_\ell}) \E \fx(X_{j_1},\dots,X_{j_\ell})
=0.
\end{align}
Moreover, suppose that 
$|i_\ga-j_\gb|>m$ for all but one pair $(\ga,\gb)\in[\ell]^2$, 
say for $(\ga,\gb)\neq(\ga_0,\gb_0)$.
Then the pair $(X_{i_{\ga_0}},X_{j_{\gb_0}})$ is independent of
all the variables 
$\set{X_{i_\ga}:\ga\neq\ga_0}$ and $\set{X_{j_\gb}:\gb\neq\gb_0}$,
and all these are mutually independent.
Hence, 
recalling \eqref{fi0}, \as
\begin{align}\label{katun}
&  \E\bigsqpar{\fx(X_{i_1},\dots,X_{i_\ell})\fx(X_{j_1},\dots,X_{j_\ell})
\mid X_{i_{\ga_0}},X_{j_{\gb_0}}}
\\\notag&
\qquad= \E\bigsqpar{ \fx\xpar{X_{i_1},\dots,X_{i_\ell}}\mid X_{i_{\ga_0}}} 
\E \bigsqpar{\fx(X_{j_1},\dots,X_{j_\ell})\mid X_{j_{\gb_0}}}
=0.
\end{align}
Thus, taking the expectation, we find that unconditionally
\begin{align}\label{baktun}
  \E\bigsqpar{\fx(X_{i_1},\dots,X_{i_\ell})\fx(X_{j_1},\dots,X_{j_\ell})}
=0.
\end{align}
Consequently, if we expand $\Var\bigsqpar{U_n(\fx;\mx)}$ in analogy with
\eqref{b1}, then all terms where $|i_\ga-j_\gb|\le m$ for at most one pair
$(\ga,\gb)$ will vanish. 
The number of remaining terms, \ie, those with at least two such pairs
$(\ga,\gb)$, is $O(n^{2\ell-2})$, and
each term is $O(1)$, by \AAA2 and the \CSineq. Consequently,
\begin{align}\label{tzol}
  \Var\bigsqpar{U_n(\fx;\mx)}
=O\bigpar{n^{2\ell-2}}.
\end{align}
Hence, we may ignore the final term $U_n(\fx;\mx)$ in \eqref{maya}.

We turn to the main terms in \eqref{maya}, i.e., the double sum; 
we denote it by $\hU_n$ and write it as
\begin{align}\label{jsb}
\hU_n=\sumjl \sumkn a_{j,k,n}f_j(X_k),
\end{align}
where we thus define
\begin{align}\label{ajkn}
  a_{j,k,n}&:=
 \binom{k-1-(j-1)m}{j-1}\binom{n-k-(\ell-j)m}{\ell-j} 
\notag\\&\phantom:
=\frac{1}{(j-1)!\,(\ell-j)!} k^{j-1}(n-k)^{\ell-j} + O(n^{\ell-2})
,\end{align}
where the $O$ is uniform over all $k\le n$ and $j\le\ell$.
Define the polynomial functions, for $j=1,\dots,\ell$,
\begin{align}\label{psi}
  \psi_j(x):=\frac{1}{(j-1)!\,(\ell-j)!} x^{j-1}(1-x)^{\ell-j},
\qquad x\in\bbR.
\end{align}
Then \eqref{ajkn} yields,
again uniformly for all $k\le n$ and $j\le\ell$, 
\begin{align}\label{wolf}
a_{j,k,n}=n^{\ell-1}\psi_j(k/n)+O\bigpar{n^{\ell-2}}
.\end{align}

The expansion \eqref{jsb} yields
\begin{align}\label{krk}
  \Var \hU_n
=
\sum_{i=1}^\ell \sum_{j=1}^\ell 
\sumkn \sum_{q=1}^na_{i,k,n}a_{j,q,n}\Cov\bigsqpar{f_i(X_k),f_j(X_q)} ,
\end{align}
where all terms with $|k-q|>m$ vanish because the sequence $(X_i)$ is \mdep.
Hence, with $r_-:=\max\set{-r,0}$ and $r_+:=\max\set{r,0}$,
\begin{align}\label{kyrie}
  \Var \hU_n
=
\sum_{i=1}^\ell \sum_{j=1}^\ell 
\sum_{r=-m}^m
\sum_{k=1+r_-}^{n-r_+}a_{i,k,n}a_{j,k+r,n}\Cov\bigsqpar{f_i(X_k),f_j(X_{k+r})} .
\end{align}
The covariance in \eqref{kyrie} is independent of $k$;
we thus define, for any $k>r_-$,
\begin{align}\label{gijr}
  \gam_{i,j,r}:=\Cov\bigsqpar{f_i(X_k),f_j(X_{k+r})}
\end{align}
and obtain
\begin{align}\label{elei}
  \Var \hU_n
=
\sum_{i=1}^\ell \sum_{j=1}^\ell 
\sum_{r=-m}^m\gam_{i,j,r}
\sum_{k=1+r_-}^{n-r_+}a_{i,k,n}a_{j,k+r,n}.
\end{align}
Furthermore, by \eqref{wolf},
\begin{align}\label{amadeus}
n^{2-2\ell} \sum_{k=1+r_-}^{n-r_+}a_{i,k,n}a_{j,k+r,n}
&= \sum_{k=1+r_-}^{n-r_+}
  \bigpar{\psi_i(k/n)+O(n\qw)}\bigpar{\psi_j(k/n)+O(n\qw)}
\notag\\&
= \sum_{k=1+r_-}^{n-r_+} \bigpar{\psi_i(k/n)\psi_j(k/n)+O(n\qw)}
\notag\\&
= \sum_{k=1}^{n} \psi_i(k/n)\psi_j(k/n)+O(1)
\notag\\&
= \int_0^n \psi_i(x/n)\psi_j(x/n)\dd x+O(1)
\notag\\&
= n\int_0^1 \psi_i(t)\psi_j(t)\dd t+O(1)
.\end{align}
Consequently, \eqref{elei} yields
\begin{align}\label{agnus}
n^{1-2\ell}  \Var \hU_n
=
\sum_{i=1}^\ell \sum_{j=1}^\ell 
\sum_{r=-m}^m\gam_{i,j,r}
\intoi\psi_i(t)\psi_j(t)\dd t
+O\bigpar{n\qw}.
\end{align}
Since \eqref{l4}, \eqref{maya}, and \eqref{tzol} yield
\begin{align}\label{fit}
\Var\bigsqpar{  U_n(f)-\hU_n} = O\bigpar{n^{2\ell-2}},
\end{align}
the result \eqref{tum1} follows from \eqref{agnus}, with
\begin{align}\label{gss}
\gss
=
\sum_{i=1}^\ell \sum_{j=1}^\ell 
\sum_{r=-m}^m\gam_{i,j,r}
\intoi\psi_i(t)\psi_j(t)\dd t
.\end{align}

Next, we use \eqref{jsb} and write
\begin{align}\label{yngve}
n^{\frac12-\ell} \hU_n
=\sumkn Z_{kn},
\end{align}
with
\begin{align}\label{frej}
  Z_{kn}:=\sum_{j=1}^\ell n^{\frac12-\ell}a_{j,k.n}f_j(X_k)
.\end{align}
Since $Z_{kn}$ is a function of $X_k$, it is evident that $(Z_{kn})$ is an
\mdep{} triangular array with centered variables. Furthermore, 
$\E Z_{kn}=0$ as a consequence of \eqref{h0}.

We apply \refT{TM} to $(Z_{kn})$, so $\TS_n=n^{\frac12-\ell}\hU_n$ by \eqref{yngve},
and verify first its conditions.
The condition \eqref{m1} holds by 
\eqref{agnus} and \eqref{gss}.
Write $Z_{kn}=\sumjl Z_{jkn}$ with
\begin{align}
  Z_{jkn}:= n^{\frac12-\ell}a_{j,k.n}f_j(X_k)
\end{align}
Since \eqref{ajkn} yields $|a_{j,k,n}|\le n^{\ell-1}$,
we have, for $\eps\ge0$, 
\begin{align}\label{van}
  \E\bigsqpar{Z_{jkn}^2\indic{|Z_{jkn}|>\eps}}
\le n\qw\E \bigsqpar{|f_j(X_k)|^2
\indic{\abs{f_j(X_k)}>\eps n\qq}}
\end{align}
The distribution of $f_j(X_k)$ does not depend on $k$, and thus the Lindeberg
condition \eqref{m2}  for each triangular array $(Z_{jkn})_{k,n}$ follows
from \eqref{van}.
The Lindeberg condition \eqref{m2} for $(Z_{nk})_{k,n}$ then follows easily.
Finally, taking $\eps=0$ in \eqref{van} yields $\E Z_{jkn}^2\le Cn\qw$, and thus 
$\E Z_{kn}^2\le Cn\qw$, which shows \eqref{m3}.

We have shown that \refT{TM} applies, and thus, recalling \eqref{yngve} and
\eqref{h0},
\begin{align}\label{as}
n^{\frac12-\ell} \bigpar{ \hU_n-\E \hU_n}
= n^{\frac12-\ell} \hU_n
=\sumkn Z_{kn}
  \dto \N(0,\gss).
\end{align}
The result \eqref{tum2} now follows from \eqref{as} and \eqref{fit}.
\end{proof}

\begin{proof}[Proof of \refT{TUMD}]
\refL{L3} implies that it suffices to consider 
$U_n\bigpar{g;(Y_i)\xooo}$ instead of $U_n(f;\cD)$
or $U_n(f;\cDq)$.
Note that the definition \eqref{YX} implies that $(Y_i)\xoo$ is a stationary 
$m'$-dependent sequence, with $m':=m+M-1$. 
Hence, the result follows from \refT{TUM} applied to $g$ and $(Y_i)\xoo$.
\end{proof}

\begin{remark}\label{Rbeta}
The integrals in \eqref{gss} are standard Beta integrals
\cite[5.12.1]{NIST}; we have
\begin{align}\label{beta}
  \intoi\psi_i(t)\psi_j(t)\dd t&
=\frac{1}{(i-1)!\,(j-1)!\,(\ell-i)!\,(\ell-j)!}
\intoi t^{i+j-2}(1-t)^{2\ell-i-j}\dd t
\notag\\&
=\frac{(i+j-2)!\,(2\ell-i-j)!}
  {(i-1)!\,(j-1)!\,(\ell-i)!\,(\ell-j)!\,(2\ell-1)!}.
\end{align}
\end{remark}

\begin{remark}\label{Rgss}
In the unconstrained case \refT{TUM}, the asymptotic variance $\gss$ is given by
\eqref{gss} together with \eqref{gijr}, \eqref{fi1} and \eqref{beta}.

In the constrained cases, the proof above
shows that $\gss$ is given by \eqref{gss} applied to 
the function $g$ given by \refL{L3} and
$(Y_i)\xoo$ given by \eqref{YX} (with $M=D+1$ for definiteness);
note that this also entails replacing $\ell$ by $b$ and $m$ by $m+M-1=m+D$
in the formulas above.
In particular, in the exactly constrained case \eqref{Uc=},
it follows from \eqref{fi1} and \eqref{han} that, 
with $y=(x_1,\dots,x_M)\in\cS^M$
and other notation as in \eqref{jedl}--\eqref{jedv} and \eqref{muD=},
\begin{align}\label{kut}
g_i(x_1,\dots,x_M)
=\E
f\bigpar{(X_{j_1+v_1+t_{1r}})_{r=1}^{\ell_1} 
,\dots,(x_{1+v_i+t_{ir}})_{r=1}^{\ell_i},
\dots,(X_{j_b+v_b+t_{br}})_{r=1}^{\ell_b}}
-\mu_\cDq 
,\end{align}
where the $i$th group of variables consists of the given $x_i$, and
the other $b-1$ groups contain variables $X_i$,
and $j_1,\dots,j_b$ is any sequence of indices that has large enough gaps:
$j_{i+1}-j_i> m+M-1= m+D$. 

In the constrained case \eqref{Uc}, $g=g_\cD$ is obtained as the sum
\eqref{glamis}, and thus each $g_i$ is a similar sum of functions that can
be obtained as \eqref{kut}. (Note that $M:=D+1$ works in \refL{L2}
for all terms by \eqref{cd'}.)
Then, $\gss$ is given by \eqref{gss} (with
substitutions as above).
\end{remark}

\section{Law of large numbers}\label{SLLN}

\begin{proof}[Proof of \refT{TLLN}]
  Note first that if $R_n$ is any sequence of random variables such that
  \begin{align}\label{r1}
    \E R_n^2=O\bigpar{n\qww},
  \end{align}
then Markov's inequality and the Borel--Cantelli lemma show that 
$R_n\asto0$.

We begin with the unconstrained case, $\cD=\cDoo=(\infty,\dots,\infty)$.
We may assume, as in the proof of \refT{TUM}, that $\mu=0$.
Then \eqref{fit} holds, and thus by the argument just given,
and recalling that $\E \hU_n=0$ by \eqref{jsb} and \eqref{h0},
\begin{align}\label{rov}
  n^{-\ell} \bigsqpar{U_n(f)-\E U_n(f) - \hU_n}\asto 0.
\end{align}
Hence, to prove \eqref{tlln20}, it suffices to
prove $n^{-\ell}\hU_n\asto0$.

For simplicity, we fix $j\in[\ell]$, and define, with $f_j$ as above given
by \eqref{fi1},
\begin{align}\label{sjn}
  S_{jn}
=S_{jn}(f)
:=S_n(f_j)
:=\sumkn f_j(X_k)
\end{align}
and, using partial summation,
\begin{align}\label{oden}
  \hU_{jn}:=\sumkn a_{j,k,n}f_j(X_k)
=\sum_{k=1}^{n-1}(a_{j,k,n}-a_{j,k+1,n})S_{jk} + a_{j,n,n}S_{jn}.
\end{align}
The sequence $(f_j(X_k))_k$ is \mdep, stationary and with  $\E|f_j(X_k)|<\infty$.
As is well known, the strong law of large number holds
for stationary \mdep{} sequences with finite means.
(This follows by considering the subsequences $(X_{(m+1)n+q})_{n\ge0}$,
which for each fixed $q\in[m+1]$ is an \iid{} sequence.)
Thus, by \eqref{sjn} and \eqref{h0},
\begin{align}\label{tor}
S_{jn}/n\asto \E f_j(X_k)=0.  
\end{align}
In other words, \as{} $S_{jn}=o(n)$, and thus also
\begin{align}\label{frigg}
  \max_{1\le k\le n}|S_{jk}| =o(n)\qquad \text{a.s.}
\end{align}
Moreover, \eqref{ajkn} implies $a_{j,k,n}-a_{j,k+1,n}=O(n^{\ell-2})$.
Hence, \eqref{oden} yields
\begin{align}
n^{-\ell} \hU_{jn}
=\sum_{k=1}^{n-1} O(n\qww)\cdot S_{jk} + O(n\qw) \cdot S_{jn}
\end{align}
and thus, using \eqref{frigg},
\begin{align}
\bigabs{n^{-\ell} \hU_{jn}}
\le C n\qw \max_{k\le n}| S_{jk}|
=o(1)\qquad\text{a.s.}
\end{align}
Consequently,
\begin{align}
  n^{-\ell}\hU_n  = \sumjl n^{-\ell}\hU_{jn }\asto 0,
\end{align}
which together with \eqref{rov} yields the
desired result \eqref{tlln20}.

Next, for an exact constraint $\cDq$, we use \refL{L2}.
Then \eqref{lx} together with the just shown result applied to $g$ and $(Y_i)$
yields
\begin{align}
  n^{-b}\ubbo{U_n(f;\cDq)} 
=  n^{-b}\ubbo{U_{n-D}(g)} 
\asto0.
\end{align}
This proves \eqref{tlln2=}, and \eqref{tlln2} follows by \eqref{Ucsum}.

Finally, using \refT{TE}, \eqref{tlln10}--\eqref{tlln1=} are equivalent to
\eqref{tlln20}--\eqref{tlln2=}.
\end{proof}

\section{The degenerate case}\label{S0}

As is well known, even in the original symmetric and independent case studied in
\cite{Hoeffding}, the asymptotic variance $\gss$ in \refT{TUM} may vanish
also in non-trivial cases. In such cases, \eqref{tum2} is still valid, but
says only that the \lhs{} converges to 0 in probability.
In the present section, we characterize this degenerate case in \refTs{TUM}
and \ref{TUMD}.
Note that in applications, it is frequently natural to guess that
$\gss>0$, but this is sometimes surprisingly difficult to prove.
One purpose of the theorems below is to assist in showing $\gss>0$; see the
applications in \refSs{Sword} and \ref{Sperm}.

For an unconstrained \Ustat{} and an independent sequence $(X_i)\xoo$ (the
case $m=0$ of \refT{TUM}), 
it is known, and not difficult to see, 
that $\gss=0$ if and only if every projection $f_i(X_1)$
defined by \eqref{fi1} vanishes a.s.,
see \cite[Corollary 3.5]{SJ332}. 
(This is included in the theorem below
by taking $m=0$ in \ref{T0gam},
and it is also the correct interpretation of
\ref{T0Y} when $m=0$.)
In the \mdep{} case, the situation is similar, but somewhat more
complicated, as shown by the following theorem.
Note that $S_n(f_j)$ defined in \eqref{bl2} below equals $S_{jn}$; for later
applications we find this change of notation convenient.

\begin{theorem}\label{T0}
With assumptions and notation as in \refT{TUM},
define also 
$f_i$ by \eqref{fi1},
$\gam_{i,j,r}$ by \eqref{gijr} and
$S_{jn}$ by \eqref{sjn}.
Then,
the following are equivalent.
\begin{romenumerate}
  
\item \label{T0gss}
  \begin{align}\label{t0gss}
    \gss=0.
  \end{align}

\item \label{T0Un}
  \begin{align}\label{t0un}
\Var U_n = O\bigpar{n^{2\ell-2}}.
  \end{align}

\item \label{T0gam}
  \begin{align}\label{t0gam}
    \sum_{r=-m}^m\gam_{i,j,r}=0,\qquad \forall i,j\in[\ell].
  \end{align}

\item \label{T0Cov}
  \begin{align}
\Cov\bigsqpar{S_{in},S_{jn}}/n \to0\text{ as \ntoo}    
\quad\forall i,j\in[\ell].
  \end{align}

\item \label{T0Var}
  \begin{align}
\Var\bigsqpar{S_{jn}}/n \to0\text{ as \ntoo}    
\quad\forall j\in[\ell].
  \end{align}

\item \label{T0Y}
For each $j\in[\ell]$ there exists a stationary sequence
$(Z_{j,k})_{k=0}^\infty$ of\/ 
$(m-1)$-dependent random variables such that \as
\begin{align}\label{fy}
  f_j(X_k)=Z_{j,k}-Z_{j,k-1},
\qquad k\ge1.
\end{align}
\end{romenumerate}

Moreover, suppose that the sequence $(X_k)\xoo$ is a block factor given by
\eqref{block} for some function $h$ and \iid{} $\xi_i$, and that $\gss=0$.
Then, in \ref{T0Y}, we may take $Z_{j,k}$ as block factors
\begin{align}\label{bl1}
  Z_{j,k}=\gf_j(\xi_{k+1},\dots,\xi_{k+m}),
\end{align}
for some functions $\gf_j:\cS_0^m\to\bbR$.
Hence, for every $j\in[\ell]$ and $n\ge1$,
\begin{align}\label{bl2}
S_n(f_j)
:=\sumkn f_j(X_k)
=Z_{j,n}-Z_{j,0}
=\gf_j(\xi_{n+1},\dots,\xi_{n+m})
-\gf_j(\xi_{1},\dots,\xi_{m}),
\end{align}
and thus $S_{n}(f_j)$
is independent of $\xi_{m+1},\dots,\xi_{n}$ 
for every
$j\in[\elli]$ and $n>m$.
\end{theorem}

To prove \refT{T0}, 
we begin with a well known algebraic lemma; for completeness we include a proof.
\begin{lemma}\label{LAB}
  Let $A=(a_{ij})_{i,j=1}^\ell$ and $B=(b_{ij})_{i,j=1}^\ell$ be symmetric real
  matrices such that $A$ is positive definite and $B$ is positive
  semidefinite. Then 
  \begin{align}\label{lab}
    \sum_{i,j=1}^\ell a_{ij}b_{ij}=0
\iff b_{ij}=0\quad\forall i,j\in[\ell].
  \end{align}
\end{lemma}
\begin{proof}
 Since $A$ is positive definite, there exists an orthonormal 
basis $(v_k)_1^\ell$ 
in $\bbR^\ell$ consisting of eigenvectors of $A$, 
in other words $Av_k=\gl_kv_k$; furthermore,
the eigenvalues $\gl_k>0$.
Write $v_k=(v_{ki})_{i=1}^\ell$. We then have
\begin{align}\label{aij}
  a_{ij}=\sumkl \gl_k v_{ki}v_{kj}.
\end{align}
 Thus
\begin{align}\label{abij}
 \sum_{i,j=1}^\ell a_{ij}b_{ij}
=\sumkl \gl_k \sum_{i,j=1}^\ell b_{ij}v_{ki}v_{kj}
=\sumkl \gl_k \innprod{v_k,Bv_k}.
\end{align}
Since $B$ is positive semidefinite, all terms in the last sum are $\ge0$, so
the sum is 0 if and only if every term is, and thus
  \begin{align}\label{abij0}
    \sum_{i,j=1}^\ell a_{ij}b_{ij}=0
\iff \innprod{v_k,Bv_k}=0 \quad \forall k\in[\ell].
  \end{align}
By the \CSineq{} for the semidefinite bilinear form $\innprod{v,Bw}$
(or, alternatively by using $\innprod{v_k\pm v_n,B(v_k\pm v_n)}\ge0$)
it follows that this condition implies 
$\innprod{v_k,Bv_n}=0$ for any $k,n\in[\ell]$, and thus
  \begin{align}\label{abij00}
    \sum_{i,j=1}^\ell a_{ij}b_{ij}=0
\iff \innprod{v_k,Bv_n}=0 \quad \forall k,n\in[\ell].
  \end{align}
Since $(v_k)_1^\ell$ is a basis, this is further equivalent to $\innprod{v,Bw}=0$
for any $v,W\in\bbR^\ell$, and thus to $B=0$. This yields \eqref{lab}.
\end{proof}

\begin{proof}[Proof of \refT{T0}]
  The $\ell$ polynomials $\psi_j$, $j=1,\dots,\ell$, of degree $\ell-1$
defined by \eqref{psi}
are linearly independent (e.g., since the matrix of their coefficients in the
standard basis $\set{1,x,\dots,x^{\ell-1}}$ is upper triangular with non-zero
diagonal elements). Hence, the Gram matrix
$A=(a_{ij})_{i,j}$  with
\begin{align}\label{gram}
  a_{ij}:=\intoi \psi_i(t)\psi_j(t)\dd t
\end{align}
  is positive definite.

We have by \eqref{sjn}, similarly to \eqref{krk}--\eqref{elei},
\begin{align}\label{qwe}
  \Cov\bigpar{S_{in},S_{jn}}&
=\sumkn \sum_{q=1}^n \Cov\bigsqpar{f_i(X_k),f_j(X_q)}
= \sum_{r=-m}^m\sum_{k=1+r_-}^{n-r_+}\Cov\bigsqpar{f_i(X_k),f_j(X_{k+r})}
\notag\\&
= \sum_{r=-m}^m(n-|r|)\Cov\bigsqpar{f_i(X_k),f_j(X_{k+r})} 
= \sum_{r=-m}^m(n-|r|)\gam_{i,j,r}
\end{align}
and thus, as \ntoo,
\begin{align}\label{blim}
  \Cov\bigpar{S_{in},S_{jn}}/n&
\to \sum_{r=-m}^m\gam_{i,j,r}
=: b_{ij}.
\end{align}
Note that \eqref{gss} can be written
\begin{align}\label{gss2}
  \gss=\sum_{i,j=1}^\ell b_{ij}a_{ij}.
\end{align}
The covariance matrices $\bigpar{\Cov(S_{in},S_{jn})}_{i,j=1}^\ell$
are positive semidefinite, and thus so is the limit $B=(b_{ij})$ defined by
\eqref{blim}. 
Hence \refL{LAB} applies and yields, using \eqref{gss2} and the definition
of $b_{ij}$ in \eqref{blim}, the equivalence 
\ref{T0gss}$\iff$\ref{T0gam}.

Furthermore, \eqref{blim} yields
\ref{T0gam}$\iff$\ref{T0Cov}.

The implication
\ref{T0Cov}$\implies$\ref{T0Var} is trivial, and the converse follows by the
\CSineq.

If \ref{T0gam} holds, then \eqref{agnus} yields $\Var \hU_n =
O\bigpar{n^{2\ell-2}}$ , and \ref{T0Un} follows by \eqref{fit}. 
Conversely, \ref{T0Un}$\implies$\ref{T0gss} by \eqref{tum1}.

Moreover, for $m\ge1$,
\ref{T0Var}$\iff$\ref{T0Y} holds by \cite[Theorem 1]{SJ286},
recalling $\E f_j(X_k)=0$ by \eqref{h0}.
(Recall also that any stationary sequence $(W_k)\xoo$ 
of real random variables
can be extended to a
doubly-infinite stationary sequence $(W_k)_{-\infty}^\infty$.)
The case $m=0$ is trivial, since then \ref{T0Var} is equivalent to 
$\Var f_j(X_k)=0$ and thus $f_j(X_k)=0$ a.s.\ by \eqref{h0},
while \ref{T0Y} should be interpreted to mean that \eqref{fy} holds for some
non-random $Z_{j,k}=z_j$.

Finally, suppose that $(X_i)\xoo$ is a block factor. 
In this case,
\cite[Theorem 2]{SJ286} shows that $Z_{j,k}$ can be chosen 
as in  \eqref{bl1}. (Again, the case $m=0$ is trivial.)
Then \eqref{bl2} is an immediate consequence of \eqref{fy}--\eqref{bl1}.
\end{proof}

\begin{remark}\label{R0Y}
  It follows from the proof in \cite{SJ286} that in \ref{T0Y}, we can choose
  $Z_{jk}$ such that also the random vectors $(Z_{jk})_{j=1}^\ell$, $k\ge0$, form a
  stationary 
$(m-1)$-dependent sequence.
\end{remark}

\begin{theorem}\label{T0D}
With assumptions and notation as in \refT{TUMD}, define also $g_i$,
$i\in[b]$,  as in
\refR{Rgss}, \ie, by \eqref{kut} in the exactly constrained case and
otherwise
as a sum of such terms over all $\cD'$ given by \eqref{cd'}.
Let also (again as in \refR{Rgss}) 
$D$ be given by \eqref{ld} and $Y_k$ by \eqref{YX} with $M=D+1$.
Then $\gss=0$ if and only if for every $j\in[b]$, 
there exists a stationary sequence
$(Z_{j,k})_{k=0}^\infty$ of 
$(m+D-1)$-dependent random variables such that \as{}
\begin{align}\label{skam}
 g_j(Y_k)=Z_{j,k}-Z_{j,k-1},
\qquad k\ge1.
\end{align}

Moreover, if the sequence $(X_i)\xoo$ is independent and $\gss=0$, then
there exist  functions $\gf_j:\cS^{D}\to\bbR$
such that
\eqref{skam} holds with
\begin{align}\label{skams}
  Z_{j,k}=\gf_j(X_{k+1},\dots,X_{k+D}),
\end{align}
and consequently \as
\begin{align}\label{skamma}
  S_{n}(g_j)
:=\sumkn g_j(Y_k)
=\gf_j(X_{n+1},\dots,X_{n+D})
-\gf_j(X_{1},\dots,X_{D}),
\end{align}
and thus $S_{n}(g_j)$ is independent of $X_{D+1},\dots,X_n$ for every
$j\in[\elli]$ and $n>D$.
\end{theorem}

\begin{proof}
As in the proof of \refT{TUMD},  it suffices to
consider $U_n(g)$ with $g$ given by  \refL{L3} (with $M=D+1$).
The first part then is an immediate  consequence of
\refT{T0}\ref{T0gss}$\Leftrightarrow$\ref{T0Y} applied to $g$
and $Y_i:=(X_i,\dots,X_{i+D})$, 
with appropriate
substitutions $\ell\mapsto b$ and $m\mapsto m+D$.
 
The second part follows similarly
by the last part of \refT{T0}, with $\xi_i=X_i$; 
note that then $(Y_i)$ is a block
factor as in \eqref{block}, with $m$ replaced by $D$.
\end{proof}

\begin{remark}
  Of course, under the assumptions of \refT{T0D}, also the other
  equivalences in \refT{T0} hold with the appropriate interpretations,
substituting $g$ for $f$ and so on.
\end{remark}

We give an example of a constrained \Ustat{} where $\gss=0$ in a somewhat
non-trivial way. 

\begin{example}\label{E0}
  Let $(X_i)\xoo$ be an infinite \iid{} symmetric random binary string,
\ie, $\cS=\setoi$ and $X_i\sim\Be(1/2)$ are i.i.d.
Let
\begin{align}\label{xf}
  f(x,y,z):=\indic{xyz=101}-\indic{xyz=011}
\end{align}
and consider the constrained \Ustat{}
\begin{align}\label{xu}
  U_n(f;\cD) = \sum_{1\le i<i+1<j\le n} f\bigpar{X_i,X_{i+1},X_j},
\end{align}
which thus has constraint $\cD=(1,\infty)$. 
(In this case, $U_n(f,\cD)=U_n(f;\cDq)$.)
Note that \eqref{xu} is a difference of two constrained subsequence counts.

Although the function \eqref{xf} might look non-trivial and innocuous at
first glance,  this turns out to be a degenerate case.
In fact, it is easily verified that
\begin{align}\label{x3}
  f(x,y,z)=(x-y)z,
\qquad x,y,z\in\setoi.
\end{align}
Hence, 
with $m=0$, $D=1$ and $M=D+1=2$,
\eqref{muD=} yields
\begin{align}\label{x4}
\mu_\cD=  \mu_\cDq
=\E g(Y_1,Y_3)=\E f(X_1,X_2,X_4)=0
\end{align}
while \eqref{kut} yields 
\begin{align}\label{x51}
  g_1(x,y)&=\E f(x,y,X_4)
=\E\bigsqpar{(x-y)X_4}
= \tfrac12(x-y),
\\\label{x52}
g_2(x,y)&=\E f(X_1,X_2,y)
=\E\bigsqpar{(X_1-X_2)y}
= 0.
\end{align}
Thus $g_2$ vanishes but not $g_1$. Nevertheless,
$g_1(Y_k)=g_1(X_k,X_{k+1})=\frac12(X_k-X_{k+1})$ is of the type in
\eqref{skam}--\eqref{skams}
(with $Z_{1,k}:=-\frac12X_{k+1}$).
Hence, \refT{T0D}
shows that $\gss=0$, and thus
\refT{TUMD} and \eqref{te} yield
$n^{-3/2}U_n(f;\cDq)\pto0$.

In fact, in this example we have by \eqref{x3}, for $n\ge3$,
\begin{align}\label{x6}
  U_n(f;\cD)&
=\sum_{j=3}^n\sum_{i=1}^{j-2}(X_i-X_{i+1})X_j
=\sum_{j=3}^n X_j(X_1-X_{j-1})
\notag\\&
=X_1\sum_{j=3}^n X_j  - \sum_{j=3}^n X_{j-1}X_j. 
\end{align}
Hence, by the law of large numbers for stationary \mdep{} sequences,
\begin{align}\label{x7}
n\qw  U_n(f;\cD)
\asto X_1\E X_2 - \E\bigsqpar{X_2X_3}
=\tfrac12 X_1 -\tfrac14
=\tfrac12\bigpar{X_1-\tfrac12}.
\end{align}
As a consequence, $n\qw U_n(f;\cD)$ has a non-degenerate limiting
distribution.
Note that this example differs in several respects from the degenerate cases
that may occur for standard \Ustats, 
\ie{}  unconstrained \Ustats{} based on independent $(X_i)\xooo$.
In this example, \eqref{x7} shows that
the asymptotic distribution is a linear
transformation 
of a Bernoulli variable, and is thus neither normal, nor of the type 
that appears 
as limits of degenerate standard \Ustats.
(The latter are polynomials in independent normal variables, in general
infinitely many, see 
\eg{}
\refT{Tworddeg} and, in general,
\cite{RubinVitale} and \cite[Chapter~11]{SJIII}.)
Moreover, the \as{} convergence to a non-degenerate limit is unheard of for
standard \Ustats, where the limit is mixing.
\end{example}

\subsection{The degenerate case in renewal theory}\label{SSDegRen}

In the renewal theory setting in \refT{TUUN}, the degenerate case is
characterized by a modified version of the conditions above.

\begin{theorem}\label{TUUN0}
  With the assumptions and notations of \refT{TUUN}, 
let $g_i$, $i\in[b]$, be as in 
\refT{T0D} and
\refR{Rgss}.
Then,
$\gamxx=0$ if and only if for every $j\in[b]$, the function
\begin{align}
\tg_j(y):=g_j(y)+\mu_\cD-{\frac{\mu_\cD}{\nu}}h(y_1),   
\qquad y=(y_1,\dots,y_b)\in\cS^b,
\end{align}
satisfies the condition \eqref{skam}.
Moreover,
if the sequence $(X_i)\xoo$ is independent and $\gamxx=0$, 
then the functions $\tg_j$ also satisfy
\eqref{skams}--\eqref{skamma}.
\end{theorem}

The proof is given in \refS{Srenew}.
Note that 
$\E\tg_j(Y_1)=0$
for each $j\in[b]$
by \eqref{h0} and \eqref{nu}.

\section{Rate of convergence}\label{Srate}

We use here a different method than in the rest of the paper.

\begin{proof}[Proof of \refT{Trate}]
We consider $U_n(f,\cD)$; the argument for
$U_n(f;\cDq)$ is identical, and
$U_n(f)$ is a special case.

Let $\cI$ denote the set of all indices $(i_1,\dots,i_\ell)$ in the sum
\eqref{Uc}; thus \eqref{Uc} can be written 
$U_n(f;\cD)=\sum_{I\in\cI} Z_I$, where
$Z_{i_1,\dots,i_\ell}:=f(X_{i_1},\dots,X_{i_\ell})$.
Note that the size $|\cI| \sim C n^b$ for some $C>0$, where $b=b(\cD)$.

We define a graph $\hcI$
with vertex set $\cI$ by putting an edge
between $I=(i_1,\dots,i_\ell)$ and $I'=(i'_1,\dots,i'_\ell)$ if and only if
$|i_j-i'_k|\le m$ for some $j,k\in\set{1\dots,\ell}$.
Let $\gD$ be 1 +  the maximum degree of the graph $\hcI$; it is easy to see
that $\gD=O(n^{b-1})$. 
Moreover, it follows from the $m$-dependence of $(X_i)$ that
$\hcI$ is a \emph{dependency graph} for the random variables
$(Z_I)_I$, meaning that if $A$ and $B$ are two disjoint subsets
of $\cI$ such that there is no  edge between $A$ and $B$, then the two
random vectors $(Z_I)_{I\in A}$ and $(Z_I)_{I\in B}$ are independent.

The result now follows from \cite[Theorem 2.2]{Rinott}, which in our
notation yields the bound,
with $\gss_n:=\Var U_n \sim \gss n^{2b-1}$
and $B:=2\sup|f|$ which implies $|Z_I-\E Z_I|\le B$ a.s. for every $I\in\cI$,
\begin{align}
  d_K&
\le \frac{1}{\gs_n}
\Bigcpar{(2\pi)\qqw \gD B 
+ 16\Bigpar{\frac{|\cI|\gD}{\gss_n}}\qq \gD B^2
+ 10\Bigpar{\frac{|\cI|\gD}{\gss_n}} \gD B^3
}
\notag\\&
\le C \frac{\gD}{\gs_n} 
\le C n\qqw,
\end{align}
since $|\cI|\gD\le C n^{b+b-1}\le C \gss_n$ and $B$ is a constant.
(Alternatively, one could use the similar bound in \cite[Theorem 2.1]{Fang}.)
\end{proof}

\begin{remark}\label{Rrate1}
  The assumption in \refT{Trate}
that $f$ be bounded can be relaxed to the 6th moment
  condition \AAA6
by using \cite[Theorem 2.1 instead of Theorem 2.2]{Rinott}
together with \Holder's inequality and straightforward estimates.

The similar bound \cite[Corollary 2]{BaldiR}
gives the weaker estimate $d_K=O\bigpar{n^{-1/4}}$, assuming
again that $f$ is bounded; 
this can be relaxed to 
\AAA4
by instead using
\cite[Theorem 1]{BaldiR}.

If we instead of the Kolmogorov distance use the Wasserstein distance $d_W$
(see \eg{} \cite[pp.~63--64]{ChenGS} 
for several equivalent definitions, and for several alternative names),
the estimate $d_W=O\bigpar{n\qqw}$ follows similarly 
from \cite[Theorem 1]{BKR}, assuming only 
the third moment condition
\AAA3;  
we omit the details.
(Actually, \cite{BKR} does not state the result for the Wasserstein distance
but for a weaker version called bounded Wasserstein distance; 
however, the same proof yields estimates for $d_W$.)
See also 
\cite[Theorem 3 and Remark 3]{Raic}, which yield the same estimate under 
\AAA3, and furthermore imply convergence in distribution assuming only
\AAA2. (This thus yields an alternative proof of \refTs{TUM} and \ref{TUMD}.)
\end{remark}

Returning to the Kolmogorov distance, 
we do not believe that the moment assumption \AAA6 in \refR{Rrate1} is best
possible. 
For unconstrained and symmetric \Ustat{s},
\citet[Theorem 2]{MalevichA} has shown bounds for the Kolmogorov distance,
which in particular show that then \AAA3 is sufficient to yield 
$d_K=O\bigpar{n\qqw}$; we conjecture that the same holds in our, more
general, setting.

\begin{conj}
  \refT{Trate} holds assuming only \AAA3 (instead of $f$ bounded).
\end{conj}

\begin{remark}\label{Rrate2}
If we do not care about the rate of convergence, 
we can for bounded $f$ alternatively 
obtain convergence in distribution in \eqref{trate}, and thus in
\eqref{tum2} and \eqref{tumd2}, by 
\cite[Theorem 2]{SJ58} using 
the dependency graph $\hcI$ in the proof of \refT{Trate}.
This can easily be extended to any $f$ satisfying the second moment
condition \AAA2 by a standard truncation argument.
\end{remark}

\section{Higher moments and maximal functions}
\label{Smom}

To prove \refT{Tmom}, we will show estimates for maximal functions that also
will be used in \refSs{Sfun} and \ref{Srenew}.
Let $p\ge2$ be fixed throughout the section; explicit and implicit constants
may thus depend on $p$.
We let 
\begin{align}\label{Ux}
U_n^*(f):=\max_{j\le n}|U_j(f)|,   
\end{align}
and use similar notation for maximal
functions of other sequences of random variables.

We use another decomposition of $f$ and $U_n(f)$
which was used in \cite{SJ332} for the independent
case ($m=0$); unlike Hoeffding's decomposition in \refS{Spf}, 
it focuses on the order of the arguments.

Recall from \refS{Smean} that $(\hX_i)_1^\ell$ are \iid{} with the
same distribution as $X_1$.
Let $\FF_0:=\mu$ defined in \eqref{mu} and, for $1\le k\le \ell$,
\begin{align}
  \FF_k(x_1,\dots,x_k)&:=\E f\bigpar{x_1,\dots,x_k,\hX_{k+1},\dots,\hX_\ell},
\label{FF}
\\
  F_k(x_1,\dots,x_k)&:=\FF_k\bigpar{x_1,\dots,x_k)-\FF_{k-1}(x_1,\dots,x_{k-1}}
.\label{F}
\end{align}
(These are defined at least for $\cL(X_1)$-\aex{} $x_1,\dots,x_k\in\cS$,
which is enough for our purposes.)
In other words, a.s.,
\begin{align}
  \label{FFF}
\FF_k(\hX_1,\dots,\hX_k)
=\E\bigpar{f(\hX_1,\dots,\hX_\ell)\mid \hX_1,\dots,\hX_k},
\end{align}
and 
thus 
$\FF_k(\hX_1,\dots,\hX_k)$, $k=0,\dots,\ell$, is a martingale, with
the martingale differences 
$F_k(\hX_1,\dots,\hX_k)$, $k=1,\dots,\ell$.
Hence, or directly from \eqref{FF}--\eqref{F},
for \aex{} $x_1,\dots,x_{k-1}$,
\begin{equation}
  \label{EFk}
\E F_k(x_1,\dots,x_{k-1},\hX_k)=0.
\end{equation}
Furthermore, if \ref{ALp} holds, 
then by \eqref{FFF} and Jensen's inequality,
\begin{align}
  \label{ub1}
\norm{\FF_k(\hX_1,\dots,\hX_k)}_p
\le\norm{f(\hX_1,\dots,\hX_\ell)}_p\le C,
\end{align}
and thus by \eqref{F},
\begin{align}
  \label{ub2}
\norm{F_k(\hX_1,\dots,\hX_k)}_p
\le2\norm{f(\hX_1,\dots,\hX_\ell)}_p\le C.
\end{align}

\begin{lemma}\label{LA}
Suppose that \ref{ALp} holds for some $p\ge2$,
and that $\mu=0$.
Then
\begin{align}\label{la}
  \bignorm{U^*_n(f;>m)}_p \le C n^{\ell-1/2}.
\end{align}
\end{lemma}

\begin{proof}
We argue as in \cite[Lemmas 4.4 and 4.7]{SJ332} with some minor differences.
By \eqref{FF}--\eqref{F}, 
$f(x_1,\dots,x_\ell)
=\FF_\ell(x_1,\dots,x_\ell)
=\sumkl F_k(x_1,\dots,x_k)$
for \aex{} $x_1,\dots,x_\ell$,
and thus, a.s.,
\begin{equation}
  \begin{split}
U_n(f;>m)
&=\sumkl \summm{k} \binom{n-i_k-(\ell-k)m}{\ell-k} 
F_k\xpar{X_{i_1},\dots,X_{i_k}} 
\\
&=\sumkl \sumin \binom{n-i-(\ell-k)m}{\ell-k}\bigpar{U_i(F_k;>m)-U_{i-1}(F_k;>m)}
\\
&=U_n(F_\ell;>m)
+\sum_{k=1}^{\ell-1} \sum_{i=1}^{n-1}\binom{n-i-(\ell-k)m-1}{\ell-k-1}U_i(F_k;>m),
  \end{split}
\end{equation}
using a summation by parts and the identity 
$\binom{n'}{\ell-k}-\binom{n'-1}{\ell-k}=\binom{n'-1}{\ell-k-1}$.
In particular,
\begin{align}\label{kum}
|U_n(f;>m)|
&\le |U_n(F_\ell;>m)|+\sum_{k=1}^{\ell-1} \sum_{i=1}^{n-1} 
\binom{n-i-(\ell-k)m-1}{\ell-k-1}
U^*_n(F_k;>m)
\notag\\
&= |U_n(F_\ell;>m)|+\sum_{k=1}^{\ell-1} \binom{n-(\ell-k)m-1}{\ell-k} U^*_n(F_k;>m)
\notag\\&
\le \sum_{k=1}^{\ell} n^{\ell-k} U^*_n(F_k;>m).
\end{align}
Since the \rhs{} is weakly increasing in $n$, it follows that, a.s.,
\begin{equation}\label{kul}
U^*_n(f;>m)
\le \sum_{k=1}^{\ell} n^{\ell-k} U^*_n(F_k;>m).
\end{equation}
We thus may consider each $F_k$ separately.
Let $1\le k\le\ell$, and let
\begin{align}\label{gDU}
\gD U_n(F_k;>m):=U_n(F_k;>m)-U_{n-1}(F_k;>m)
.\end{align}
By the definition \eqref{Uc>},
$\gD U_n(F_k;>m)$
is a sum of $\binom{n-(k-1)m-1}{k-1}\le n^{k-1}$ terms
$F_k(X_{i_1},\dots,X_{i_{k-1}},X_n)$ that all have the same distribution
as $F_k\xpar{\hX_1,\dots,\hX_k}$,
and thus by Minkowski's inequality and \eqref{ub2},
\begin{equation}\label{swab}
\norm{\gD U_n(F_k;>m)}_p
\le {n}^{k-1}\norm{F_k(\hX_1,\dots,\hX_k)}_p 
\le C{n}^{k-1}
.\end{equation}
Furthermore, in each such term $F_k(X_{i_1},\dots,X_{i_{k-1}},X_n)$ 
we have $i_{k-1}\le n-m-1$.
Hence, if we let $\cF_i$ be the \gsf{} generated by $X_1,\dots,X_i$, then,
 by $m$-dependence,
$X_n$ is independent of $\cF_{i_{k-1}}$, whence
\eqref{EFk} implies
\begin{align}
  \E \bigpar{F_k\xpar{X_{i_1},\dots,X_{i_{k-1}},X_n}\mid \cF_{n-m-1}}
&=
  \E \bigpar{F_k\xpar{X_{i_1},\dots,X_{i_{k-1}},X_n}\mid X_{i_1},\dots,X_{i_{k-1}}}
\notag\\&=0.
\end{align}
Consequently,
\begin{align}\label{annaw}
  \E\bigpar{\gD U_n(F_k;>m)\mid \cF_{n-m-1}}=0.
\end{align}
In the independent case $m=0$ treated in \cite{SJ332}, 
this means that $U_n(F_k)$ is a martingale.
In general, we may as a substitute  split $U_n(F_k;>m)$
as a sum of $m+1$ martingales.
For $j=1,\dots,m+1$ and $i\ge1$, let
\begin{align}
\gD M_i\jjj&:=\gD U_{(i-1)(m+1)+j}(F_k;>m),\label{uba}
\\\label{ubb}
  M_i\jjj
&:=\sum_{q=1}^i\gD M_q\jjj
=\sum_{q=1}^i\gD U_{q(m+1)+j}(F_k;>m)
.\end{align}
Then, \eqref{annaw} implies that $(M_i\jjj)_{i\ge0}$ is a martingale, for
each $k$ and $j$. Hence, Burkholder's inequality 
\cite[Theorem 10.9.5]{Gut} yields, for the maximal function $M_n\jjjx$,
\begin{align}\label{ub3}
  \norm{M_n\jjjx}_p
\le C \Bignorm{\Bigpar{\sum_{i=1}^n |\gD M_i\jjj|^2}\qq}_p
= C \Bignorm{{\sum_{i=1}^n |\gD M_i\jjj|^2}}_{p/2}\qq.
\end{align}
Furthermore, Minkowski's inequality yields (since $p/2\ge1$),
using also \eqref{uba} and \eqref{swab}, for $n\ge1$,
\begin{align}\label{ub4}
 \Bignorm{{\sum_{i=1}^n |\gD M_i\jjj|^2}}_{p/2}
\le \sum_{i=1}^n \bignorm{|\gD M_i\jjj|^2}_{p/2}
= \sum_{i=1}^n \bignorm{\gD M_i\jjj}_{p}^2
\le C n^{1+2(k-1)}
.\end{align}
Combining \eqref{ub3} and \eqref{ub4} yields
\begin{align}\label{ub5}
  \norm{M_n\jjjx}_p
\le C n^{k-1/2}
.\end{align}
 It follows from \eqref{uba}--\eqref{ubb} that
\begin{align}\label{ub6}
U_n(F_k;>m) =\sum_{j=1}^{m+1} M\jjj_{\floor{(n-j)/(m+1)}+1}  .
\end{align}
Hence (coarsely),
\begin{align}\label{ub7}
U^*_n(F_k;>m) \le\sum_{j=1}^{m+1} M\jjjx_{n},
\end{align}
and thus \eqref{ub5} and Minkowski's inequality yield
\begin{align}\label{ub8}
  \bignorm{U^*_n(F_k;>m)}_p
\le C n^{k-1/2}
,\end{align}
for $k=1,\dots,\ell$.

The result \eqref{la} now follows from \eqref{ub8} and \eqref{kul} 
by a final application
of Minkowski's inequality.  
\end{proof}

\begin{theorem}\label{TQ}
  Suppose that \ref{ALp} holds for some $p\ge2$.
Then, with $b=b(\cD)$,
\begin{align}\label{tq1}
  \bignorm{\max_{j\le n}\bigabs{U_j(f;\cD)-\E U_j(f;\cD)}}_p 
&= O\bigpar{n^{b-1/2}},
\\\label{tqq}
  \bignorm{\max_{j\le n}\bigabs{U_j(f;\cD)-\frac{\mu_\cD}{b!} j^b }}_p 
&= O\bigpar{n^{b-1/2}},
\\\label{tq2}
  \bignorm{U^*_n(f;\cD)}_p& = O\bigpar{n^{b}}
.\end{align}

The same results hold for an exact constraint $\cDq$.
\end{theorem}

\begin{proof}
We 
use induction on $b$.
We split the induction step into three cases.

\resetcases
\pfcase{no constraint, i.e., $\cD=\cDoo$ and $b=\ell$}
By \eqref{bada}--\eqref{cdj}, 
\begin{align}\label{tq3}
U_n(f)=U_n(f;\cD_\emptyset)= U_n(f;\mx)
-\sum_{J\neq\emptyset}(-1)^{|J|}U_n(f;\cD_J).
\end{align}
Thus,
\begin{align}\label{tq4}
U^*_n(f) 
\le U^*_n(f;\mx)+\sum_{J\neq\emptyset}U^*_n(f;\cD_J).
\end{align}

Suppose first that $\mu=0$; then  
\refL{LA} applies to $U^*_n(f;\mx)$. 
Furthermore, 
the induction hypothesis applies to each term in the sum in \eqref{tq4},
since $b(\cD_J)=\ell-|J|\le \ell-1=b-1$.
Hence, Minkowski's inequality yields
\begin{align}\label{tq6}
\bignorm{U^*_n(f)}_p 
\le \bignorm{U^*_n(f;\mx)}_p+\sum_{J\neq\emptyset}\bignorm{U^*_n(f;\cD_J)}_p
\le C n^{\ell-1/2}+ C n^{\ell-1}
.\end{align}
When $\mu=0$, 
\eqref{te0} yields
\begin{align}
  \label{tq7}
\E U_n(f)=O\bigpar{n^{\ell-1}}.
\end{align}
Now
\eqref{tq1} follows from \eqref{tq6} and \eqref{tq7}, 
which shows \eqref{tq1}
when $\mu=0$. The general case follows by considering $f-\mu$;
this does not affect $U_n(f)-\E U_n(f)$.

Finally, 
both \eqref{tqq} and
\eqref{tq2} follow from \eqref{tq1} and \eqref{te0}.

\pfcase{an exact constraint $\cDq$, $b<\ell$}
An immediate consequence of \eqref{lx} in \refL{L2} and Case 1 applied to
$g$; note that 
$g$ too satisfies \ref{ALp} by \eqref{solo}.

\pfcase{a constraint $\cD$, $b<\ell$}
A consequence of Case 2 by \eqref{Ucsum} and \eqref{muD}.
\end{proof}

\begin{lemma}\label{LUI}
 Suppose that \ref{ALp} holds for some $p\ge2$.
Let $b:=b(\cD)$.
Then
the sequences
\begin{align}\label{tui3}
n^{1/2-b}\max_{j\le n}\bigabs{U_j(f;\cD)-\E U_j(f;\cD)}
\qquad\text{and}\qquad
n^{-b}U^*_n(f;\cD)
\qquad(n\ge1)
\end{align}
are uniformly \pthp{} integrable.

The same holds for an exact constraint $\cDq$.
\end{lemma}
\begin{proof}
We consider the second sequence in \eqref{tui3}; the proof for the first
sequence 
differs only notationally.

We have so far let $f$ be fixed, so the constants above may depend on $f$.
However, it is easy to see that the proof of \eqref{tq2} yields
\begin{align}\label{tui8}
    \bignorm{U^*_n(f;\cD)}_p\le C_p  
\max_{i_1<\dots< i_\ell}\norm{f(X_{i_1},\dots,X_{i_\ell})}_p n^{b}, 
\end{align}
with $C_p$ independent of $f$ (but depending on $p$).
(Note that we only have to consider a finite set of
indices $(i_1,\dots,i_\ell)$, as discussed above \eqref{al2}.

Truncate $f$, and define, for  $B>0$,
$f_B(\bx):=f(\bx)\indic{|f(\bx)|\le B}$.
Then \eqref{tui8} yields
\begin{align}\label{tui9}
    \bignorm{n^{-b} U^*_n(f-f_B;\cD)}_p\le C_p  \eps(B),
  \end{align}
where
  \begin{align}
\eps(B):=
\max_{i_1<\dots< i_\ell}
\norm{f(X_{i_1},\dots,X_{i_\ell})\indic{|f(X_{i_1},\dots,X_{i_\ell})|>B}}_p 
\to 0
\end{align}
as $B\to\infty$. 

Let $q:=2p$.
Since $f_B$ is bounded, we may apply \eqref{tui8} (or \refT{TQ}) 
with $p$ replaced by $q$
and obtain
\begin{align}\label{sup}
\sup_n  \bignorm{n^{-\ell}U^*_n(f_B;\cD)}_{2p} <\infty.
\end{align}
Hence, for any $B$,
the sequence $n^{-\ell} U^*_n(f_B;\cD)$ is uniformly \pthp{} integrable.
Since
$ U^*_n(f;\cD) \le  U^*_n(f_B;\cD)+ U^*_n(f-f_B;\cD)$,
the result now follows from the following simple observation.
\end{proof}

\begin{lemma}\label{LLUI}
Let $1\le p<\infty$.
Let $(\xi_n)_{n\ge1}$ be a sequence of  random variables.
Suppose that for every $\eps>0$, there exist random variables 
$\eta^\eps_n$ and $\zeta^\eps_n$, $n\ge1$, such 
that 
\begin{romenumerate}
\item \label{LLUI1}
$|\xi_n|\le\eta^\eps_n+\zeta^\eps_n$,  
\item \label{LLUI2}
the sequence $(|\eta^\eps_n|^p)_n$ is \ui,
\item \label{LLUI3}
$\norm{\zeta^\eps_n}_p\le\eps$.
\end{romenumerate}
Then $\xpar{|\xi_n|^p}_{n}$ is uniformly integrable.
\end{lemma}
\begin{proof}
Since \ref{LLUI1} implies $|\xi_n|^p \le 2^p|\eta^\eps_n|^p+2^p|\zeta^\eps_n|^p$,
it suffices (by appropriate substitutions) to consider the case $p=1$.
This is a simple exercise, using for example \cite[Theorem 5.4.1]{Gut}.
\end{proof}

\begin{proof}[Proof of \refT{Tmom}]
  An immediate consequence of \refTs{TUM}--\ref{TUMD} and the uniform
  integrability given by \refL{LUI}.
\end{proof}

\section{Functional convergence}\label{Sfun}

We begin by improving \eqref{la} in a special situation.
(We consider only $p=2$.)

\begin{lemma}\label{LB}
  Suppose that \AAA2 holds and that $\mu=0$ and
$f_i(X_i)=0$ a.s.\ for every   $i=1,\dots,\ell$. 
Then 
\begin{align}\label{lb}
  \bignorm{U^*_n(f;>m)}_2 \le C n^{\ell-1}.
\end{align}
\end{lemma}

\begin{proof}
Note that \eqref{maya} and \eqref{tzol} immediately give this estimate for 
$\norm{U_n(f;>m)}_2$. To extend it to the maximal function
$U_n^*(f;>m)$,
we reuse the proof of \refL{LA} (with $p=2$), and analyse the terms 
$U^*_n(F_k;>m)$ further.
First,
by \eqref{FFF}, \eqref{fi2} and the assumptions,
for every $k\in[\ell]$,
\begin{align}\label{lb1}
  \E\bigpar{\FF_k(\hX_1,\dots,\hX_k)\mid\hX_k}
=
  \E\bigpar{f(\hX_1,\dots,\hX_\ell)\mid\hX_k}
=f_k(\hX_k)+\mu=0
.\end{align}

In particular, for $k=1$, \eqref{lb1} yields
$\FF_1(\hX_1)=0$ a.s.,
and thus 
\begin{align}
  \label{lb0}
U^*_n(F_1;>m)=0 \text{\quad  a.s.} 
\end{align}

For $k\ge2$, as said in the proof of \refL{LA},
$\gD U_n(F_k;>m)$
is a sum of $\le n^{k-1}$ terms
$F_k(X_{i_1},\dots,X_{i_{k-1}},X_n)$. 
Consider two such terms
$F_k(X_{i_1},\dots,X_{i_{k-1}},X_n)$ and
$F_k(X_{i'_1},\dots,X_{i'_{k-1}},X_n)$,
and suppose that $|i_j-i'_{j'}|>m$ for all $j,j'\in[k-1]$.
Then all variables $X_{i_j}, X_{i'_{j'}}$, and $X_n$ are independent,
and thus a.s.
\begin{align}
&\E\bigsqpar{F_k(X_{i_1},\dots,X_{i_{k-1}},X_n)F_k(X_{i'_1},\dots,X_{i'_{k-1}},X_n)
\mid X_n}  
\notag\\&\quad
=
\E\bigsqpar{F_k(X_{i_1},\dots,X_{i_{k-1}},X_n)\mid X_n}  
\E\bigsqpar{F_k(X_{i'_1},\dots,X_{i'_{k-1}},X_n)\mid X_n}  
=0,
\end{align}
by \eqref{lb1}.
Hence, taking the expectation,
\begin{align}\label{lb5}
&\E\bigsqpar{F_k(X_{i_1},\dots,X_{i_{k-1}},X_n)F_k(X_{i'_1},\dots,X_{i'_{k-1}},X_n)}  
=0,
\end{align}
unless $|i_j-i'_{j'}|\le m$ for some pair $(j,j')$.
For each $(i_1,\dots,i_{k-1})$, there are only $O(n^{k-2})$ such
$(i'_1,\dots,i'_{k-1})$, and for each of these, the expectation in \eqref{lb5}
is $O(1)$ by \AAA2 and the \CSineq. Consequently,
summing over all appearing
$(i_1,\dots,i_{k-1})$ and $(i'_1,\dots,i'_{k-1})$,
\begin{align}\label{lb7}
  \E \bigsqpar{\bigabs{\gD U_n(F_k;>m)}^2}&
=\sum_{i_1,\dots,i_{k-1},i'_1,\dots,i'_{k-1}}
\E\bigsqpar{F_k(X_{i_1},\dots,X_{i_{k-1}},X_n)F_k(X_{i'_1},\dots,X_{i'_{k-1}},X_n)}  
\notag\\&
= O\bigpar{n^{k-1}\cdot n^{k-2}}=O\bigpar{n^{2k-3}}.
\end{align}
We have gained a factor of $n$ compared to \eqref{swab}.
Hence, recalling \eqref{uba} and 
using \eqref{lb7} in \eqref{ub3}--\eqref{ub4} (which for $p=2$
essentially just is Doob's inequality), we improve \eqref{ub5} to
\begin{align}\label{ub5+}
  \norm{M_n\jjjx}_2
\le C n^{k-1}
.\end{align}
Finally, \eqref{ub5+} and \eqref{ub7} yield
\begin{align}\label{ub8+}
  \bignorm{U^*_n(F_k;>m)}_2
\le C n^{k-1}
,\end{align}
for $2\le k\le \ell$;
this holds trivially for $k=1$ too by \eqref{lb0}.
The result
follows by \eqref{kul} and \eqref{ub8+}.
\end{proof}

\begin{proof}[Proof of \refT{TG}]
We prove \eqref{tg1}; then \eqref{tg2} follows by \eqref{te}.
By replacing $f$ by $f-\mu$, we may assume that $\mu=0$.

Consider first the unconstrained case.
We argue as in \cite{SJ332}, with minor modifications.
We use \eqref{maya}, which we write as, cf.\ \eqref{jsb},
\begin{align}\label{tg3}
  U_n\xpar{f;>m}=
\sum_{j=1}^\ell\sum_{i=1}^n \anj{i} f_j(X_i)+ U_n(\fx;>m),
\end{align}
with, as in \eqref{ajkn},
\begin{align}\label{tg4}
  \anj{i}:=\binom{i-1-(j-1)m}{j-1}\binom{n-i-(\ell-j)m}{\ell-j}.
\end{align}
\refL{LB} applies to $\fx$ and shows that
\begin{align}\label{unf}
\norm{U^*_n(\fx;>m)}_2=O\bigpar{n^{\ell-1}}=o\bigpar{n^{\ell-1/2}},  
\end{align}
which implies that  the last term in \eqref{tg3} is negligible,
so we
concentrate on the sum.
Define $\gD \anj{i}:=\anj{i+1}-\anj{i}$
and, using a summation by parts,
\begin{align}\label{tg5}
  \hUnj:=\sumin \anj{i} f_j(X_i)
=\anj{n}S_n(f_j) -\sum_{i=1}^{n-1}\gD\anj{i}S_i(f_j).
\end{align}
Donsker's theorem extends to \mdep{} stationary sequences
\cite{Billingsley56},
and thus,
as \ntoo,
\begin{align}\label{tg6}
  n\qqw S\nt(f_j)\dto W_j(t)
\qquad\text{in $D\ooo$}
,\end{align}
for a continuous centered Gaussian process $W_j$ (a suitable multiple of
Brownian motion); 
furthermore, as is easily seen, 
this holds
jointly for $j=1,\dots,\ell$.
Moreover, 
define
\begin{equation}\label{psi2}
\psi_j(s,t):=\frac{1}{(j-1)!\,(\ell-j)!}s^{j-1}(t-s)^{\ell-j}.
\end{equation}
(Thus $\psi(s,1)=\psi(s)$ defined in \eqref{psi}; the present homogeneous
version is more convenient here.)
Let $\psi'_j(s,t):=\frac{\partial}{\partial s}\psi(s,t)$.
Then,
straightforward calculations (as in \cite[Lemma 4.2]{SJ332})
show that, extending \eqref{ajkn},
\begin{align}\label{tg7}
  \anj{i}&=\psi_j(i,n)+O(n^{\ell-2}),
\\\label{tg8}
\gD\anj{i}&=\psi'_j(i,n)
+O\bigpar{n^{\ell-3}+ n^{\ell-2}\indic{i\le m \text{ or }i\ge n-m}
}
\end{align}
uniformly for all $n,j,i$ that are relevant;
moreover, the error terms with negative powers, i.e., $n^{\ell-2}$ for
$\ell=1$ and $n^{\ell-3}$ for $\ell\le2$, vanish.

By the Skorohod coupling theorem \cite[Theorem~4.30]{Kallenberg},
we may assume that the convergences \eqref{tg6} hold a.s., and
similarly, see \eqref{unf}, a.s.
\begin{align}\label{tg23}
n^{1/2-\ell}U^*_n(\fx;>m)\to 0.
\end{align}
It then follows from \eqref{tg5}--\eqref{tg8} and the homogeneity of $\psi_j$
that a.s.,
uniformly for $t\in[0,T]$ for any fixed $T$,
\begin{align}\label{tg9}
  n^{1/2-\ell}\hU_{\fnt,j}&
=n^{1-\ell}\psi_j(\fnt,\fnt) W_j(t)
-\sum_{i=1}^{\fnt-1}n^{1-\ell}\psi'_j(i,\fnt)W_j(i/n)
+o(1)
\notag\\&
=\psi_j(t,t) W_j(t)
-\frac{1}{n}\sum_{i=1}^{\fnt-1}\psi'_j(i/n,t)W_j(i/n)
+o(1)
\notag\\&
=\psi_j(t,t) W_j(t)
-\int_0^t\psi'_j(s,t)W_j(s)\dd s
+o(1)
.\end{align}
Summing over $j\in[\ell]$, we obtain by \eqref{tg3}, \eqref{tg5},
\eqref{tg9}, and \eqref{tg23}, 
a.s.\ uniformly for $t\in[0,T]$ for any $T$,
\begin{align}\label{tg10}
  n^{1/2-\ell}U\nt(f;>m)&
=Z(t)
+o(1)
,\end{align}
where
\begin{align}\label{Zt}
Z(t):=\sumjl \Bigpar{\psi_j(t,t) W_j(t)
-\int_0^t\psi'_j(s,t)W_j(s)\dd s}
,\end{align}
which obviously is a centered Gaussian process.
We can rewrite \eqref{tg10} as
\begin{align}\label{tr1}
  n^{1/2-\ell}U\nt(f;>m)\to Z(t)
\qquad\text{in $D\ooo$}
.\end{align}

Finally we use \eqref{tq3}, which implies
\begin{align}
  \max_{k\le n} \bigabs{U_k(f)-U_k(f;>m)}
\le \sum_{J\neq\emptyset} U^*_n(f;\cD_J)
\end{align}
and thus, by \refT{TQ},
recalling $b(\cD_J)=\ell-|J|\le\ell-1$,
\begin{align}\label{tg12}
\bignorm{\max_{k\le n} \bigabs{U_k(f)-U_k(f;>m)}}_2&
\le \sum_{J\neq\emptyset} \bignorm{U^*_n(f;\cD_J)}_2
\le \sum_{J\neq\emptyset} C n^{b(\cD_J)}
\le C n^{\ell-1}.
\end{align}
It follows that,
in each $D[0,T]$ and thus in $D\ooo$,
\begin{align}\label{tr2}
n^{1/2-\ell}\bigpar{  U\nt(f)-U\nt(f;>m)}
\pto 0
.\end{align}
Furthermore, recalling the assumption $\mu=0$, 
$\E U_n(f)=O\bigpar{n^{\ell-1}}$ by \eqref{te0},
and thus
\begin{align}\label{tr13}
n^{1/2-\ell}\E U\nt(f)\to0 \qquad \text{in $\cD\ooo$}.  
\end{align}

The result \eqref{tg1} in the unconstrained case follows from 
\eqref{tr1}, \eqref{tr2} and \eqref{tr13}.

Joint convergence for several $f$ (in the unconstrained case)
follows by the same proof.

Finally, as usual, 
the exactly constrained case  follows by \eqref{lx} in \refL{L2} and
the constrained case then follows by \eqref{Ucsum}, 
using
joint convergence for all $g_{\cDz}$, 
with notation  as
in \eqref{Ucsum} and \refL{L2}.
To obtain joint convergence for several $f$ and $\cD$, we only have to
choose $M$ in \eqref{YX} large enough to work for all of them.
\end{proof}

\section{Renewal theory}\label{Srenew}

Note first that by the law of large numbers for \mdep{} sequences,
\begin{align}\label{llnh}
S_n/n=S_n(h)/n\asto \E h(X_1)=\nu 
\qquad \text{as \ntoo.}
\end{align}
Moreover, as in \eqref{tg6}, Donsker's theorem for \mdep{} sequences
\cite{Billingsley56} yields
\begin{align}\label{ukr}
  n\qqw {S\nt(h-\nu)}
=
  n\qqw \bigpar{S\nt(h)- \fnt\nu}\dto W_h(t)
\qquad\text{in $D\ooo$}
\end{align}
for a continuous centered Gaussian process $W_h(t)$.
As a simple consequence, we have the following
(the case $N_+$ is in \cite[Theorems 2.1 and 2.2]{SJ29}),
which extends the well known
case of independent $X_i$, see \eg{} \cite[Sections 3.4 and 3.10]{Gut-SRW}. 
\begin{lemma}\label{LN}
  As $x\to\infty$,
  \begin{align}
    N_\pm(x)/x&\asto 1/\nu,\label{ln1}
\\
S_{N_\pm(x)}(h) &= x + \op\bigpar{x\qq}.\label{ln2}
  \end{align}
\end{lemma}
\begin{proof}
  Note that, by the definitions \eqref{N-}--\eqref{N+}, 
  \begin{align}\label{S-+}
S_{N_-(x)}(h)\le x < S_{N_-(x)+1} (h)
\qquad\text{and}\qquad
S_{N_+(x)-1}(h)\le x < S_{N_+(x)}(h).    
  \end{align}
Then \eqref{ln1} follows easily from \eqref{llnh}.
Furthermore, \eqref{ln2} implies that as \Ntoo, 
$S_{N+1}-S_N=\op \bigpar{N\qq}$,
and \eqref{ln2} follows.
We omit the standard details.
\end{proof}

\begin{proof}[Proof of \refT{TUUN}]
Note that \eqref{ukr} is the special (unconstrained) case $f=h$,
$\cD=()$,
$\ell=b=1$
of \eqref{tg1}.
By joint convergence in \refT{TG} for $(f,\cD)$ and $h$, 
we thus have \eqref{tg2} jointly with
\eqref{ukr}. We use again
the Skorohod coupling theorem 
and assume that \eqref{tg2}, \eqref{ukr}, and \eqref{ln2} hold a.s.

Take $n:=\ceil{x}$ and $t:=N_\pm(x)/n$, and let \xtoo.
Then, $t\to1/\nu$ a.s.\ by  \eqref{ln1}, and thus \eqref{tg2} implies,
a.s.,
\begin{align}\label{uk1}
  U_{N_\pm(x)}(f;\cD)&
=\frac{\mu_\cD}{b!} \Npmx^b+Z(\nuqw) n^{b-1/2}+o\bigpar{n^{b-1/2}}
\notag\\&
=\frac{\mu_\cD}{b!} \Npmx^b+Z(\nuqw) x^{b-1/2}+o\bigpar{x^{b-1/2}}
.\end{align}
Similarly, \eqref{ukr}  implies, a.s.,
\begin{align}\label{uk2}
  S_{N_\pm(x)}(h)= \Npmx\nu+W_h(\nuqw) x^{1/2}+o\bigpar{x^{1/2}}.
\end{align}
By \eqref{uk2} and \eqref{ln2}, we have \as{}
\begin{align}\label{uk3}
  \nu\Npmx 
=
  S_{N_\pm(x)}(h)-W_h(\nuqw) x^{1/2}+o\bigpar{x^{1/2}}
=
  x-W_h(\nuqw) x^{1/2}+o\bigpar{x^{1/2}}.
\end{align}
Thus, by the binomial theorem, a.s.,
\begin{align}\label{uk4}
 (\nu\Npmx)^b 
=
  x^b-b W_h(\nuqw) x^{b-1/2}+o\bigpar{x^{b-1/2}}.
\end{align}
Hence, \eqref{uk1} yields, a.s.,
\begin{align}\label{uk5}
  U_{N_\pm(x)}(f;\cD)
=\frac{\mu_\cD}{\nu^b b!}\bigpar{x^b-b W_h(\nuqw)x^{b-1/2}}
+Z(\nuqw) x^{b-1/2}+o\bigpar{x^{b-1/2}},
\end{align}
which yields \eqref{cvtauu} with
\begin{align}\label{ukgam}
  \gamxx=\Var\Bigsqpar{Z(\nuqw)-\frac{\mu_\cD}{\nu^b(b-1)!}W_h(\nuqw)}.
\end{align}
The exactly constrained case and
joint convergence follow similarly.
\end{proof}

\begin{proof}[Proof of \refT{TUUN0}]
We may as in the proof of \refT{TUUN} assume that 
\eqref{tg2}, \eqref{ukr},  and \eqref{uk4} hold a.s.
Recall that the proofs above
use the decomposition \eqref{Ucsum}
and \refL{L2} applied to every $\cDz$ there, with
$b:=b(\cD)$, $D$ given by \eqref{ld}, $M:=D+1$ (for definiteness),
and  $Y_i$ defined by \eqref{YX}.
Furthermore, \eqref{obi} holds with $g=g_\cD:\cS^b\to\bbR$ given by
\eqref{glamis}; in \eqref{tg1}--\eqref{tg2}, 
we thus have the same limit $Z(t)$ for 
$U_n(f;\cD ;(X_i)\xooo)$ and $U_{n}\xpar{g;(Y_i)\xooo}$.
We may assume that this limit holds a.s.\ also for $g$. 

Recall that $h:\cS\to\bbR$. We abuse notation and extend it to $\cS^M$ by
$h(x_1,\dots,x_M):=h(x_1)$; thus $h(Y_i)=h(X_i)$. In particular, 
$S_n(h;(X_i))=S_n(h;(Y_i))$, so we may write $S_n(h)$ without ambiguity. 
We define $H:(\cS^M)^b\to\bbR$ by
\begin{align}\label{H}
  H(y_1,\dots,y_b):=\sumjb h(y_j).
\end{align}
Note that \eqref{mu} and \eqref{fi1} applied to the function $H$ yield
\begin{align}\label{h1}
  \mu_H&:=\E H\bigpar{\hY_1,\dots,\hY_b}
=b\nu,
\\\label{h2}
H_j(y)&\phantom: =h(y)+(b-1)\nu-\mu_H
=h(y)-\nu.
\end{align}
(Since $H$ is symmetric, $H_j$ is the same for every $j$.)
 
In the unconstrained sum \eqref{U},
there are $\binom{n-1}{\ell-1}$ terms that contain $X_i$, for each $i\in[n]$.
Applying this to $H$ and $(Y_i)$, we obtain by \eqref{H}
\begin{align}\label{unH}
  U_n(H;(Y_i))=\binom{n-1}{b-1} S_n(h).
\end{align}
Hence, for each fixed $t>0$, by \eqref{ukr},
a.s.,
\begin{align}\label{uk6}
  U\nt(H;(Y_i))-\E U\nt(H;(Y_i))&
=\binom{\fnt-1}{b-1}S\nt(h-\nu)
\notag\\&
= \frac{(nt)^{b-1}}{(b-1)!}n\qq W_h(t)+o\bigpar{n^{b-1/2}}.
\end{align}
Combining \eqref{tg1} (for $g$ and $(Y_i)$)
and \eqref{uk6}, we obtain that, a.s.,
\begin{align}\label{vk1}
\frac{  U\nt\bigpar{g-\frac{\mu_\cD}{\nu} H}
- \E U\nt\bigpar{g-\frac{\mu_\cD}{\nu} H}}
{n^{b-1/2}}
=Z(t)-\frac{\mu_\cD t^{b-1}}{\nu(b-1)!}W_h(t)+o(1).
\end{align}
Taking $t=\nuqw$, we see that this converges to the random variable in
\eqref{ukgam}.
Let $G:=g-\mu_\cD\nu\qw H$. Then a comparison with \refT{TUM} (applied to
$G$)
shows that
\begin{align}
  \gamxx=t^{2b-1}\gss(G)=\nu^{1-2b}\gss(G).
\end{align}
In particular, $\gamxx=0$ if and only if $\gss(G)=0$,
and the result follows by \refT{T0D}, noting that
$G_j:=g_j-\mu_\cD\nu\qw H_j=g_j+\mu_\cD-\mu_\cD\nu\qw h$ by \eqref{h2}  
\end{proof}

\begin{proof}[Proof of \refT{TUUN2}]
  This can be proved as \cite[Theorem 3.13]{SJ332}, by first stopping at
  $N_+(x_-)$, with $x_-:=\floor{x-\ln x}$, and then continuing to $N_-(x)$;
we therefore only sketch the details.
Let $R(x):=S_{N_+(x)}-x>0$ be the overshoot at $x$,
and let $\gD(x):=x-S_{N_+(x_-)}=x-x_--R(x_-)$.
It is well known, see e.g.\ \cite[Theorem 2.6.2]{Gut-SRW}
that $R(x)$ converges in distribution as \xtoo. 
In particular, $\gD(x)\pto+\infty$ and thus $\P\xsqpar{\gD(x)>0}\to1$.
Since $N_+(x_-)$ is a stopping time and $(X_i)$ are independent, 
the increments of the random walk $S_n$
after $N_+(x_-)$ are independent of $U_{N_+(x_-)}(f)$, 
and it follows that
the overshoot $R(x-1)$ is asymptotically independent of $U_{N_+(x_-)}(f)$.
The event $\set{S_{N_-(x)}=x}$ equals $\set{R(x-1)=1}$, and thus 
the asymptotic distribution of $U_{N_+(x_-)}(f)$ conditioned on $S_{N_-(x)}=x$
is the same as without conditioning, and given by \eqref{cvtauu}.
Finally, the difference between $U_{N_+(x_-)}(f)$ and $U_{N_\pm(x)}(f)$ is
negligible, e.g.\ as a consequence of \eqref{tg2}.
\end{proof}

In the remainder of the section, we prove moment convergence.
Since we here consider different exponents $p$ simultaneously,
we let $C_p$ denote constants that may depend on $p$.

\begin{lemma}\label{LS}
  Assume that $\nu:=\E h(X_1)>0$ and that $\E |h(X_1)|^p<\infty$ for every
  $p<\infty$. 
Then, for every $p<\infty$, $A\ge 2/\nu$ and $x\ge1$, we have
\begin{align}\label{ls}
  \P\bigsqpar{N_\pm(x)\ge Ax}\le C_p(Ax)^{-p}.
\end{align}
\end{lemma}

\begin{proof}
  We have $N_+(x)\le N_-(x)+1$; hence it suffices to consider $N_-(x)$.

Let $A\ge 2/\nu$.
If $N_-(x)\ge Ax$, then 
\eqref{S-+} implies
\begin{align}
  S_{N_-(x)}(h-\nu) \le x-N_-(x)\nu\le (1-A\nu)x \le -(A\nu/2) x.
\end{align}
Hence, for any $p\ge 2$, using \eqref{tq1} for $h$ 
(which is a well-known consequence of Doob's and Burkholder's, or
Rosenthal's, inequalities), 
\begin{align}\label{ls2}
\P\bigsqpar{Ax \le N_-(x) \le 2Ax}&
\le (A x\nu/2)^{-p} \E\bigsqpar{\bigabs{S_{N_-(x)}(h-\nu)}^p\indic{N_-(x)\le 2Ax}}
\notag\\&
\le (A \nu x/2)^{-p} \E\bigsqpar{\bigabs{S^*_{\floor{2A x}}(h-\nu)}^p}
\le C_p (Ax)^{-p} (2Ax)^{p/2}
\notag\\&
\le C_p (Ax)^{-p/2}.
\end{align}
We replace $p$ by $2p$ and 
$A$ by $2^k A$ in \eqref{ls2}, and sum for $k\ge0$; this yields
\eqref{ls}.
\end{proof}

\begin{lemma}\label{LM}
Assume that \AAA{p} and $\E |h(X_1)|^p<\infty$ hold for every
  $p<\infty$, and that
 $\nu:=\E h(X_1)>0$.
Then, for every $p\ge1$ and $x\ge1$,
\begin{align}\label{lmg}
  \Bignorm{U_\Npmx(f;\cD)-\frac{\mu_\cD}{\nu^bb!}x^b}_p
\le C_p x^{b-1/2}.
\end{align}
\end{lemma}

\begin{proof}
Let $V_n:=U_n(f;\cD)-\frac{\mu_\cD}{b!}n^b$, let $B:=2/\nu$, and choose $q:=2bp$.
Then the \CSineq, \eqref{tqq}, and \refL{LS} yield, with 
$V^*_x:=\sup_{n\le x}|V_n|$,
\begin{align}\label{lma}
\hskip1em&\hskip-1em 
\E\Bigabs{U_\Npmx(f;\cD)-\frac{\mu_\cD}{b!}\Npmx^b}^p
=  \E|V_\Npmx|^p
\notag\\&
=\E\bigsqpar{|V_\Npmx|^p\indic{\Npmx\le Bx}}
+\sumk \E\bigsqpar{|V_\Npmx|^p\indic{2^{k-1}Bx<\Npmx\le 2^kBx}}
\notag\\&
\le \E\bigsqpar{|V^*_{Bx}|^p}
+\sumk \E\bigsqpar{|V^*_{2^kBx}|^p\indic{\Npmx> 2^{k-1}Bx}}
\notag\\&
\le \E\bigsqpar{|V^*_{Bx}|^p}
+\sumk \E\bigsqpar{|V^*_{2^kBx}|^{2p}}\qq\P\bigsqpar{\Npmx> 2^{k-1}Bx}\qq
\notag\\&
\le C_p x^{p(b-1/2)}
+ \sumk C_p (2^kx)^{p(b-1/2)}C_q (2^{k-1}x)^{-q/2}
\le C_p x^{p(b-1/2)}
.\end{align}
In other words,
\begin{align}\label{lmb}
  \Bignorm{U_\Npmx(f;\cD)-\frac{\mu_\cD}{b!}\Npmx^b}_p
\le C_p x^{b-1/2}.
\end{align}
We may here replace $U_n(f;\cD)$ by $S_n(h)$ 
(and thus $b$ by 1 and $\mu_\cD$ by $\nu$). This yields
\begin{align}\label{lmc}
\bignorm{S_\Npmx(h)-\nu\Npmx}_p
\le C_p x^{1/2}
.\end{align}
By the same proof, this holds also if we replace $\Npmx$ by $\Npmx\mp1$.
Using \eqref{S-+}, it follows that
\begin{align}\label{lmd}
\bignorm{x-\nu\Npmx}_p
\le C_p x^{1/2}
.\end{align}
In particular, by Minkowski's inequality,
\begin{align}\label{lme}
\bignorm{\nu\Npmx}_p
\le x+\bignorm{x-\nu\Npmx}_p
\le C_p x
.\end{align}
Consequently, by Minkowski's and \Holder's inequalities, \eqref{lmd} and
\eqref{lme}, 
\begin{align}\label{lmf}
\bignorm{ (\nu\Npmx)^b-x^b}_p&
=\lrnorm{\sum_{k=0}^{b-1} (\nu\Npmx-x)(\nu\Npmx)^kx^{b-1-k}}_p
\notag\\&
\le
\sum_{k=0}^{b-1} \bignorm{\nu\Npmx-x}_{(k+1)p}\bignorm{\nu\Npmx}_{(k+1)p}^kx^{b-1-k}
\notag\\&
\le C_p x^{b-1/2}.
\end{align}
Combining \eqref{lmb} and \eqref{lmf}, we obtain \eqref{lmg}.
\end{proof}

\begin{proof}[Proof of \refT{TUUNp}]
We have shown that the \lhs{} of \eqref{cvtauu} is uniformly bounded in
$L^p$ for $x\ge1$.
By replacing $p$ with $2p$, say, this implies that these \lhs{s} are
uniformly $p$-th integrable, for every $p<\infty$,
which implies convergence of all moments in \eqref{cvtauu}.

The proof of \refT{TUUN2} shows that, under the assumptions there,
 $\P(S_{N_-(x)}=x)$ converges to a
positive limit as \xtoo; hence
$\P(S_{N_-(x)}=x)>c$ for some $c>0$ and all large $x$. This implies that the
uniform $p$-th integrability holds also after conditioning (for large $x$), 
and thus all moments converge in \eqref{cvtauu} also after conditioning.
\end{proof}

\section{Constrained pattern matching in words}\label{Sword}
As said in \refS{S:intro},
\citet{FlajoletSzV} studied the following problem;
see also  \citet[Chapter 5]{JacSzp}.
Consider a random string $\Xi_n=\xi_1\dotsm\xi_n$, where the letters $\xi_i$
are \iid{} random elements in some finite alphabet $\cA$.
(We may regard $\Xi_n$ as the initial part of an infinite string
$\xi_1\xi_2\dots$ of \iid{} letters.)
Consider also  a fixed word $\bw=w_1\dotsm w_\ell$
from the same alphabet.
(Thus, $\ell\ge1$ denotes the length of $w$; we keep $\bw$ and $\ell$ fixed.)
Let $N_n(\bw)$ be the (random) number of occurrences of $\bw$ in $\Xi_n$.
More generally, for any constraint $\cD=(d_1,\dots,d_{\elli})$, 
let $N_n(\bw;\cD)$ be the number of constrained occurrences.
This is a special case of the general setting in \eqref{U}--\eqref{Uc}, 
with
$X_i=\xi_i$, and, \cf{} \eqref{occur},
\begin{align}\label{wf}
  f(x_1,\dots,x_\ell)=\indic{x_1,\dots,x_\ell=\bw}
=\indic {x_i=w_i\;\forall i\in[\ell]}.
\end{align}
Consequently,
\begin{align}\label{Uword}
N_n(\bw;\cD)= U_n(f;\cD;(\xi_i))
\end{align}
with $f$ given by \eqref{wf}. 

Denote the distribution of the individual letters by
\begin{align}\label{px}
  p(x):=\P(\xi_1=x),
\qquad x\in\cA;
\end{align}
We will, without loss of generality, assume $p(x)>0$ for every $x\in\cA$.
Then, \eqref{Uword} and the general results above yield
the following result from \cite{FlajoletSzV},
with $b=b(\cD)$
given by \eqref{b}.
The unconstrained case (also in \cite{FlajoletSzV}) is a special case.
Moreover, the theorem holds also for
the exactly constrained case,
with $\mu_\cDq=\prod_i p(w_i)$ and  some $\gss(\bw;\cDq)$;
we leave the detailed statement to the reader.
A formula for $\gss$ is given in \cite[(14)]{FlajoletSzV};
we show explicitly that $\gss>0$ except in trivial (non-random) cases,
which seems omitted from \cite{FlajoletSzV}.

\begin{theorem}[{\citet{FlajoletSzV}}] \label{Tword}
With notations as above, as \ntoo,
\begin{align}\label{ts1}
\frac{  N_n(\bw;\cD)-\frac{\mu_\cD}{b!}n^b}{n^{b-\xfrac12}}
\dto N\bigpar{0,\gss}
\end{align}
for some $\gss=\gss(\bw;\cD)\ge0$, with
\begin{align}\label{ts2}
  \mu_\cD:= \prod_{d_j<\infty} d_j \cdot\prod_{i=1}^\ell p(w_i) 
.\end{align}
Furthermore, the all moments converge in \eqref{ts1}.

Moreover,  
if $|\cA|\ge2$, 
then $\gss>0$.
\end{theorem}


\begin{proof}
By \eqref{Uword}, the convergence
\eqref{ts1} is an instance of \eqref{tumd2} in \refT{TUMD}
together with \eqref{te}  in \refT{TE}.  
The formula \eqref{ts2} follows from \eqref{0muD} since,
by \eqref{wf} and independence,
\begin{align}\label{w4}
  \mu:=\E f(\xi_1,\dots,\xi_\ell)=  
\P\bigpar{\xi_1\dotsm\xi_\ell=w_1\dotsm w_\ell}
=\prod_{i=1}^\ell\P(\xi_i=w_i).
\end{align}

Moment convergence follows by \refT{Tmom}; note that \AAA{p} is trivial,
since $f$ is bounded.

Finally, assume $|\cA|\ge2$ and
suppose that $\gss=0$. Then \refT{T0D} says that \eqref{skamma}
holds, and thus, for each $n>D$, the sum $S_{n}(g_j)$ is independent of
$\xi_{D+1},\dots,\xi_n$. We  consider only $j=1$.
Choose $a\in\cA$ with $a\neq w_1$. 
Consider first an exact constraint $\cDq$. Then $g_\cDq$ is given by
\eqref{han}. 
Since $f(x_1,\dots,x_\ell)=0$ whenever $x_1=a$, it follows from
\eqref{han} that $g(y_1,\dots,y_b)=0$ whenever
$y_1=(y_{1k})_{k=1}^M$ has $y_{11}=a$.
hence, \eqref{fi1} shows that
\begin{align}\label{w5}
  g_1(y_1)=-\mu_{\cDq}=-\mu,
\qquad\text{if }y_{11}=a.
\end{align}
Consequently, on the event $\xi_1=\dots=\xi_{n}=a$, 
we have, recalling \eqref{YX} and $M=D+1$, 
$g_1(Y_k)=g_1(\xi_k,\dots,\xi_{k+D})=-\mu$ for every $k\in[n]$. Thus,
\begin{align}\label{w6a}
  S_n(g_1)=-n\mu
\qquad\text{if}\quad 
\xi_1=\dots =\xi_n=a.
\end{align}
On the other hand, as noted above, the assumption $\gss=0$ implies that
$S_{n}(g_1)$ is independent of $\xi_{D+1}, \dots,\xi_n$.
Consequently, \eqref{w6a} implies
\begin{align}\label{w6}
  S_n(g_1)=-n\mu
\qquad\text{if}\quad 
\xi_1=\dots =\xi_D=a,
\end{align}
regardless of $\xi_{D+1}\dots,\xi_{n+D}$.
This is easily shown to lead to contradiction. For example,
we have, by \eqref{h0},
\begin{align}
  \label{w7}
\E g_1(Y_k)=\E g_1(\xi_k,\dots,\xi_{k+D})=0
,\end{align}
and thus, conditioning on $\xi_1,\dots,\xi_D$,
\begin{align}
 \E \bigpar{S_n(g_1)\mid \xi_1=\dots =\xi_D=a}
&=
\sumkn  
\E \bigpar{g_1(\xi_k,\dots,\xi_{k+D})
 \mid \xi_1=\dots =\xi_D=a}
\notag\\& 
=O(1),\label{w8}
\end{align}
since all terms with $k>D$ are unaffected by the conditioning and
thus vanish by \eqref{w7}; this contradicts
\eqref{w6} for large $n$, since $\mu>0$.
This contradiction shows that $\gss>0$ for an exact constraint $\cDq$.

Alternatively, instead of using the expectation as in \eqref{w8}, one might
easily show that if $n>D+\ell$, then $\xi_{D+1},\dots,\xi_n$ can be chosen 
such that \eqref{w6} does not hold.

For a constraint $\cD$, $g=g_\cD$ is given by a sum \eqref{glamis} of
exactly constrained cases $\cDz$.
Hence, by summing \eqref{w5} for these $\cDz$, it follows that \eqref{w5}
holds also for $g_\cD$ (with $\mu$ replaced by $\mu_{\cD}$). This leads to a
contradiction exactly as above.
\end{proof}

\refT{Tword} shows that, except in trivial cases, the asymptotic variance
$\gss>0$ for a subsequence count $N_n(\bw;\cD)$, and thus \eqref{ts1} 
yields a non-degenerate limit, and thus really shows asymptotic normality.
By the same proof, see also \refR{RCW}, \refT{Tword} extends to linear
combinations of different 
subsequence counts (in the same random string $\Xi_n$),
but in this case, it may happen that $\gss=0$, and then \eqref{ts1} 
has a degenerate limit and thus
yields only convergence in probability to 0.
(We consider only linear combinations with coefficients not depending on $n$.)
One such degenerate
example with constrained subsequence counts is discussed  in \refE{E0}.
There are also degenerate examples in the unconstrained case.
In fact, the general theory of degenerate (in this sense) \Ustats{} 
based on independent $(X_i)\xoo$
is well understood; 
for symmetric \Ustat{s} this case 
was characterized by \cite{Hoeffding} and studied in detail by
\cite{RubinVitale}, and their results were extended to the asymmetric case
relevant here in \cite[Chapter 11.2]{SJIII}.
In \refApp{Aword} we apply these general results to 
string matching
and give a rather detailed treatment of 
the degenerate cases of linear combinations of unconstrained
subsequence counts.
See also \cite{Even-ZoharEtAl2020} for further algebraic aspects 
of both non-degenerate and degenerate cases.

\begin{problem}\label{Pcword0}
\refApp{Aword} considers only the unconstrained case.  
\refE{E0} shows that for linear combinations of 
constrained pattern counts, there are further
possibilities to have $\gss=0$. It would be interesting to 
extend the results in \refApp{Aword} and characterize
these cases, and also to obtain limit theorems for such cases, extending 
\refT{Tworddeg} (in particular the case $k=2$); note again that the limit 
in \refE{E0} is of a different type than the ones occuring in unconstrained
cases (\refT{Tworddeg}).
We leave this as open problems.
\end{problem}

\section{Constrained pattern matching in permutations}
\label{Sperm}

Consider now random permutations.
As usual, we generate a random permutation $\bpi=\bpi\nn\in\fS_n$ 
by taking a sequence $(X_i)_1^n$
of \iid{}  random variables with a uniform distribution $X_i\sim U(0,1)$,
and then replacing the values $X_1,\dots,X_n$, in increasing order, 
by $1,\dots,n$.
Then, the number $N_n(\tau)$ of occurrences of a fixed permutation
$\tau=\tau_1\dotsm\tau_\ell$ 
in
$\bpi$ is given by the \Ustat{} $U_n(f)$ defined by \eqref{U} with
\begin{align}\label{ftau}
  f(x_1,\dots,x_\ell) := \prod_{1\le i<j\le \ell} \indic{x_i<x_j\iff \tau_i<\tau_j}.
\end{align}
Similarly, 
for any constraint $\cD=(d_1,\dots,d_{\elli})$, we have
for the number of constrained occurrences   of  $\tau$,
with the same $f$ given by \eqref{ftau},
\begin{align}\label{Uperm}
N_n(\tau;\cD)=U_n(f;\cD).
\end{align}

Hence, \refTs{TUM} and \ref{TUMD} yield the following result showing
asymptotic normality of the number of (constrained) occurrences.
As said in the introduction, the unconstrained case was shown by 
\citet{Bona-Normal}, the case $d_1=\dots=d_{\elli}=1$ by \citet{Bona3}
and the general vincular case by \citet{Hofer};
we extend it to general constrained cases.
The fact that $\gss>0$ 
was shown in \cite{Hofer} (in vincular cases); 
we give a shorter proof based on \refT{T0D}.
Again, the theorem
holds also for
the exactly constrained case, 
with $\mu_\cDq=1/\ell!$ and some $\gss(\tau;\cDq)$.

\begin{theorem}[largely \citet{Bona-Normal,Bona3} and \citet{Hofer}] 
\label{Tperm}
For any fixed permutatation $\tau\in\fS_\ell$ and constraint 
$\cD=(d_1,\dots,d_{\elli})$,
as \ntoo,
\begin{align}\label{tp1}
\frac{  N_n(\tau;\cD)-\frac{\mu_\cD}{b!}n^b}{n^{b-\xfrac12}}
\dto N\bigpar{0,\gss}
\end{align}
for some $\gss=\gss(\tau;\cD)\ge0$ and
\begin{align}\label{tp2}
  \mu_\cD:= \frac{1}{\ell!}\prod_{d_j<\infty} d_j. 
\end{align}
Furthermore, all moments converge in \eqref{tp1}.

Moreover,  
if $\ell\ge2$, 
then $\gss>0$.
\end{theorem}

\begin{proof}
  This is similar to the proof of \refT{Tword}. 
By \eqref{Uperm}, the convergence
\eqref{tp1} is an instance of \eqref{tumd2} 
together with \eqref{te}. 
The formula \eqref{tp2} follows from \eqref{0muD} since
$  \mu:=\E f(X_1,\dots,X_\ell)$ by \eqref{ftau} is the probability that
$X_1,\dots,X_\ell$ have the same order as $\tau_1,\dots,\tau_\ell$, i.e.,
$1/\ell!$. 
Moment convergence follows by \refT{Tmom}.

Finally, suppose that $\ell\ge2$ but $\gss=0$. 
Then \refT{T0D} says that \eqref{skamma}
holds, and thus, for each $j$ and each $n>D$, 
the sum $S_{n}(g_j)$ is independent of
$X_{D+1},\dots,X_n$; we want to show that this leads to a
contradiction. 
We  choose again $j=1$, 
but we now consider two cases separately.

\pfcase{$d_1<\infty$}
Recall the notation in  
\eqref{jedl}--\eqref{jedv}, and note that in this case
\begin{align}\label{aw0}
  \ell_1>1,
\qquad t_{11}=0, \qquad t_{12}=d_1,
\qquad v_1=0
.\end{align}
Assume for definiteness that $\tau_1>\tau_2$. 
(Otherwise, we may exchange  $<$ and $>$ in the argument below.)
Then \eqref{ftau} implies that
\begin{align}\label{aw1}
  f(x_1,\dots,x_\ell)=0 
\qquad \text{if}\quad x_1<x_2.
\end{align}
Consider first the exact constraint $\cDq$. Then $g_\cDq$ is given by
\eqref{han}. Hence, \eqref{aw1} and \eqref{aw0}
imply that
\begin{align}\label{aw2}
g_\cDq(y_1,\dots,y_b)=0
\qquad\text{if}\quad
y_1=(y_{1,k})_{k=1}^M \text{ with } y_{1,1}<y_{1,1+d_1}.
\end{align}
In particular,
\begin{align}\label{aw3}
g_\cDq(y_1,\dots,y_b)=0
\qquad\text{if}\quad
y_{1,1}<y_{1,2}<\dots<y_{1,M}.
\end{align}
By \eqref{glamis}, the same holds for the constraint $\cD$.
Hence, \eqref{fi1} shows that, for $g=g_\cD$,
\begin{align}\label{tp5}
  g_1(y_1)=-\mu_{\cD}
\qquad\text{if}\quad
y_{1,1}<y_{1,2}<\dots<y_{1,M}.
\end{align}
Consequently, on the event $X_1<\dots<X_{n+D}$, 
we have, recalling \eqref{YX} and $M=D+1$, 
$g_1(Y_k)=g_1(X_k,\dots,X_{k+D})=-\mu_\cD$ for every $k\in[n]$, and thus
\begin{align}\label{tpp}
S_n(g_1)=-n\mu_\cD
\qquad\text{if}\quad
X_1<\dots<X_{n+D}. 
\end{align}
On the other hand, as noted above, the assumption $\gss=0$ implies that
$S_{n}(g_1)$ is independent of $X_{D+1}, \dots,X_n$.
Consequently, \eqref{tpp} implies that a.s.
\begin{align}\label{tp6}
  S_n(g_1)=-n\mu_\cD
\qquad\text{if}\quad X_1<\dots <X_D<X_{n+1}<\dots<X_{n+D}
.\end{align}
However, in analogy with \eqref{w7}--\eqref{w8}, 
we have
$\E g_1(Y_k)=0$
by \eqref{h0}, and thus
\begin{align}
&  \E \bigpar{S_n(g_1)\mid X_1<\dots <X_D<X_{n+1}<\dots<X_{n+D}}
\notag\\&\quad
=
\sumkn  
\E \bigpar{g_1(X_k,\dots,X_{k+D})
\mid X_1<\dots <X_D<X_{n+1}<\dots<X_{n+D}}
\notag\\&\quad
=O(1),\label{tp8}
\end{align}
since all terms with $D<k\le n-D$ are unaffected by the conditioning and
thus vanish. 
But \eqref{tp8} contradicts
\eqref{tp6} for large $n$, since $\mu_\cD>0$.
This contradiction shows that $\gss>0$ when $\ell\ge2$ and $d_1<\infty$.

\pfcase{$d_1=\infty$}
In this case, 
$\ell_1=1$.
Consider again first $\cDq$. 
Since $(X_i)$ are \iid, then \eqref{han} and \eqref{fi1} yield, 
choosing $j_i:=(D+1)i$, say,
\begin{align}
  g_1(y_1)
=\E g\bigpar{y_1,Y_{j_2},\dots,Y_{j_b}}-\mu
=\E f\bigpar{y_{11},X_2,\dots,X_\ell}-\mu
=f_1(y_{11}).
\end{align}
(With $\mu=1/\ell!$.)
Thus, recalling \eqref{YX},
\begin{align}\label{aw6}
  S_n(g_1)=\sumkn g_1(Y_k) = \sumkn f_1(X_k).
\end{align}
By \refT{T0D}, the assumption $\gss=0$ thus implies that the final sum in
\eqref{aw6} is independent of $X_{D+1}$, for any $n\ge D+1$.
Since $(X_i)$ are independent, this is possible only if 
$f_1(X_{D+1})=c$ \as{} for some constant $c$, \ie, 
if $f_1(x)=c$ for \aex{} $x\in(0,1)$. 

However, by \eqref{ftau},
$f\bigpar{x,X_2,\dots,X_\ell}=1$ if and only if $\tau_1-1$ prescribed $X_j$
are in $(0,x)$ and in a specific order, and the remaining $\ell-\tau_1$ ones
are in $(x,1)$ and in a specific order. Hence, \eqref{fi1} yields
\begin{align}
f_1(x)
  =\frac{1}{(\tau_1-1)!\,(\ell-\tau_1)!}x^{\tau_1-1}(1-x)^{\ell-\tau_1}-\mu.
\end{align}
Since $\ell\ge2$, $f_1(x)$ is a non-constant polynomial in $x$.

This is a contradiction, and shows that $\gss>0$ also when $d_1=\infty$. 
\end{proof}

\begin{remark}\label{Rperm0}
  Although, $\gss>0$ for each pattern count $N_n(\tau;\cD)$ 
with $\ell>1$, 
non-trivial linear combinations might have $\gss=0$, and thus variance of
lower order, even in the unconstrained case.
(Similarly to the case of patterns in strings in \refS{Sword} and
\refApp{Aword}.)
In fact, for the unconstrained case, it is shown in \cite{SJ287} that
for  permutations $\tau$ of a given length $\ell$, 
the $\ell!$ counts $N_n(\tau)$ converge jointly, after normalization as
above, to a multivariate normal distribution of dimension only $(\ell-1)^2$,
meaning that there is a linear space of dimension $\ell!-(\ell-1)^2$ 
of linear combinations 
that have $\gss=0$. This is further analyzed in
\cite{Even-Zohar}, where the spaces of linear combinations of $N_n(\tau)$
having variance $O\bigpar{n^{2\ell-r}}$ are characterized for each
$r=1,\dots,\ell-1$, using the representation theory of the symmetric group.
In particular, the highest degeneracy, 
with variance $\Theta\bigpar{n^{\ell+1}}$,
is obtained for the sign statistic
$U_n(\sgn)$, where $\sgn(x_1,\dots,x_\ell)$ is the sign of the permutation
defined by the order of $(x_1,\dots,x_n)$; in other words,
$U_n(\sgn)$ is the sum of the signs of the $\binom n\ell$ subsequences of
length $\ell$ of a random permutation $\bpi\nn\in\fS_n$.
For $\ell=3$, the asymptotic distribution of 
$n^{-(\ell+1)/2}U_n(\sgn)$ is of the type
in \eqref{Z2}, see \cite{FisherLee} and \cite[Remark 2.7]{SJ287}. 
For larger $\ell$, the asymptotic distribution can 
by the methods in \refApp{Aword}
be expressed as a
polynomial of degree $\ell-1$ in  infinitely many independent
normal variables, as in \eqref{tworddeg}; however, we do not know any
concrete such representation.


We expect that, in analogy with \refE{E0}, for linear combinations of 
constrained pattern counts, there are further
possibilities to have $\gss=0$. 
We have not pursued this, and we leave it as an open problem to characterize
these cases with $\gss=0$; moreover, it would also be interesting  
to extend the results of \cite{Even-Zohar}
characterizing cases with higher degeneracies to constrained cases.
\end{remark}

\section{Further comments}\label{Sfurther}

We discuss here briefly some possible extensions of the present work.
We have not pursued them, and they are left as open problems.

\subsection{Mixing and Markov input}
We have in this paper 
studied \Ustats{} based on a sequence $(X_i)$ that is allowed to be dependent,
but only under the rather strong assumption of $m$-dependence
(partly motivated by our application to constrained \Ustats).
It would be interesting to extend the results to 
weaker assumptions on $(X_i)$, for example that it is stationary with some
type of mixing property.
(See \eg{} \cite{Bradley} for various mixing conditions and central limit
theorems under some of them.)

Alternatively (or possibly as a special case of mixing conditions), 
it would be interesting to consider $(X_i)$ that form a stationary Markov
chain (under suitable assumptions).

In particular, it seems interesting to study
constrained \Ustats{} under such assumptions, since the mixing or Markov
assumptions typically imply strong dependence for sets of variables $X_i$ with
small gaps between the indices, but not if the gaps are large.

Markov models
are popular models for random strings.
Substring counts, \ie, the completely constrained case of subsequence counts
(see \refR{Rfull}) have been treated for Markov sources by \eg{}
\cite{RegSzp98},
\cite{NicodemeSF}
and \cite{JacSzp}.

A related model for random strings is a probabilistic dynamic source,
see \eg{} \cite[Section 1.1]{JacSzp}. 
For substring counts,
asymptotic normality has been shown by \cite{BourdonVallee06}.
For (unconstrained or constrained)
subsequence counts, 
asymptotic results on mean and variance 
are
special cases of \cite{BourdonVallee02} and
\cite[Theorem 5.6.1]{JacSzp}; 
we are not aware of any results on asymptotic normality in this setting.

\subsection{Generalized \Ustats}

\emph{Generalized \Ustats} (also called \emph{multi-sample \Ustats}
are defined similarly to \eqref{U}, but are based on two (for simplicity) 
sequences $(X_i)_1^{n_1}$ and $(Y_j)_1^{n_2}$ of random variables, with the
sum in \eqref{U} replaced by a sum over all $i_1<\dots<i_{\ell_1}\le n_1$ 
and $j_1<\dots<j_{\ell_2}\le n_2$, and $f$ now a function
of $\ell_1+\ell_2$ variables.
Limit theorems, including asymptotic normality, under  suitable conditions
are shown in \cite{Sen-generalized}, and extensions to asymmetric cases are
sketched in \cite[Example 11.24]{SJIII}.
We do not know any extensions to \mdep{} or constrained cases, but we expect
that such extensions are straightforward.

\appendix
\section{Linear combinations  for unconstrained subsequence counts} \label{Aword}

As promised in \refS{Sword}, we consider here unconstrained subsequence counts
in a random string $\Xi_n$ with \iid{} letters,
normalized as in \refT{Tword}, 
and study further the case of linear combinations of such normalized counts
(with coefficients not depending on $n$);
in particular, we study in some detail such linear combinations
that are
degenerate in the sense that the asymptotic variance $\gss=0$. 

The results are based on  the orthogonal decomposition 
introduced in the symmetric case by \citet{Hoeffding-LLN},
see also \citet{RubinVitale}; this is extended to the asymmetric case 
in \cite[Chapter 11.2]{SJIII}, but the treatment there
uses a rather heavy formalism, and we therefore give here a direct treatment 
in the present special case. (This case is
somewhat simpler than the general case since we only have
to consider finite-dimensional vector spaces below,
but otherwise the general case is similar.)
See also \cite{Even-ZoharEtAl2020},
which contains a much deeper algebraic study of
the asymptotic variance $\gss(f)$ and the
vector spaces below, 
and
in particular a spectral decomposition that refines \eqref{majk}.

Fix $\cA$ and the random string $(\xi_i)\xoo$.
 Assume, as in \refS{Sword}, that $p(x)>0$ for every $x\in\cA$.
Let $A:=|\cA|$, the number of different letters.

We fix also $\ell\ge1$ and consider all unconstrained subsequence counts
$N_n(\bw)$ with $|\bw|=\ell$.
There are $A^\ell$ such words $\bw$, and it follows from \eqref{Uword} and
\eqref{wf} that the linear combinations of these counts are precisely the
asymmetric \Ustat{s} \eqref{U} for all $f:\cA^\ell\to\bbR$, by the relation
\begin{align}\label{fU}
  \sum_{\bw\in\cA^\ell} f(\bw)N_n(\bw)=U_n(f).
\end{align}
Note that \refT{TUM}
applies to every $U_n(f)$
and thus \eqref{tum1} and \eqref{tum2} hold for some $\gss=\gss(f)\ge0$.
(As said above, this case of \refT{TUM}  with 
\iid{} $X_i$, \ie, the case $m=0$,
is treated also in
\cite[Corollary 11.20]{SJIII} and \cite{SJ332}.)

Let $V$ be the linear space of all functions $f:\cA^\ell\to\bbR$. 
Thus $\dim V=A^\ell$. 
Similarly, let $W$ be the linear space of all functions $h:\cA\to\bbR$,
\ie, all functions of a single letter;
thus $\dim W=A$.
Then $V$ can be identified with the tensor product $W^{\tensor\ell}$, with
the identification
\begin{align}\label{ht}
  h_1\tensor\dotsm\tensor h_\ell(x_1,\dots,x_\ell)=\prod_1^\ell h_i(x_i).
\end{align}
We regard $V$ as a (finite-dimensional) Hilbert space with inner product
\begin{align}\label{ipV}
  \innprod{f,g}_V
:=\E\bigsqpar{f(\Xi_n)g(\Xi_n)},
\end{align}
and, similarly, $W$ as a Hilbert space with inner product
\begin{align}\label{ipW}
  \innprod{h,k}_W
:=\E\bigsqpar{h(\xi_1)k(\xi_1)}.
\end{align}

Let $W_0$ be the subspace of $W$ defined by
\begin{align}\label{maj1}
  W_0:=\set{1}^\perp=\set{h\in W:\innprod{h,1}_W=0}
=\set{h\in W:\E h(\xi_1)=0}.
\end{align}
Thus, $\dim W_0=A-1$.

For a subset $\cB\subseteq \cA$, let $V_\cB$ be the subspace of $V$ spanned by 
all functions $h_1\tensor\dotsm\tensor h_\ell$ as in \eqref{ht} 
such that $h_i\in W_0$ if $i\in \cB$, and $h_i=1$ if $i\notin \cB$.
In other words, if we for a given $\cB$ define $W'_i:=W_0$ when $i\in \cB$ and
$W'_i=\bbR$ when $i\notin \cB$, then
\begin{align}\label{maj2}
  V_\cB=W'_1\tensor\dotsm\tensor W'_\ell
\cong W_0^{\tensor|\cB|}
.\end{align}
It is easily seen that these $2^A$ subspaces of $V$ are orthogonal, and that
we have an orthogonal decomposition
\begin{align}\label{majb}
  V=\bigoplus_{\cB\subseteq\cA} V_\cB.
\end{align}
Furthermore, for $k=0,\dots,\ell$, define
\begin{align}\label{maja}
  V_k:=\bigoplus_{|\cB|=k} V_\cB.
\end{align}
Thus, we have also an orthogonal decomposition
(as in \cite{Hoeffding-LLN} and \cite{RubinVitale})
\begin{align}\label{majk}
  V=\bigoplus_{k=0}^\ell V_k.
\end{align}
Note that, by \eqref{maj2} and \eqref{maja},
\begin{align}\label{maj4}
  \dim V_\cB = (A-1)^{|\cB|},
\qquad
  \dim V_k = \binom{\ell}{k}(A-1)^{k}.
\end{align}
Let $\Pi_\cB$ and $\Pi_k=\sum_{|\cB|=k}\Pi_\cB$ be the orthogonal projections of $V$ onto $V_\cB$
and $V_k$. Then, for any $f\in V$, we may consider its components 
$\Pi_kf\in V_k$.

First, $V_0$ is the 1-dimensional space of constant functions in $V$.
Trivially, if $f\in V_0$, then $U_n(f)$ is non-random, 
so $\Var U_n(f)=0$ for every $n$, and $\gss(f)=0$.
More interesting is that for any $f\in V$, we have 
\begin{align}
  \label{Pi0}
\Pi_0f=\E f(\xi_1,\dots,\xi_\ell)=\mu.
\end{align}

Next, it is easy to see that  taking $\cB=\set{i}$ yields
the projection $f_i$ defined by \eqref{fi1}, except that $\Pi_{\set i} f$ is
defined as a function on $\cA^\ell$; to be precise,
\begin{align}\label{Pi1}
  \Pi_{\set i}f(x_1,\dots,x_\ell)=f_i(x_i).
\end{align}
Recalling \eqref{fU}, this
leads to the following characterization of 
degenerate linear combinations of unconstrained subsequence counts.
\begin{theorem}
  \label{Tword0}
With notations and assumptions as above,
if $f:\cA^\ell\to\bbR$, then 
the following are equivalent.
\begin{romenumerate}
\item \label{Tword0a}
$\gss(f)=0$.
\item \label{Tword0b}
$f_i=0$ for every $i=1,\dots,\ell$.
\item \label{Tword0c}
$\Pi_1f=0$.
\end{romenumerate}
\end{theorem}
\begin{proof}
  The equivalence \ref{Tword0a}$\iff$\ref{Tword0b}
is a special case of \refT{T0}; as said above, this special case is also
given in \cite[Corollary 3.5]{SJ332}. 

The equivalence \ref{Tword0b}$\iff$\ref{Tword0c}
follows by \eqref{maja} and \eqref{Pi1}, which give
\begin{align}\label{peking}
  \Pi_1f=0 
\iff \Pi_{\set i}f=0\; \forall i
\iff f_i=0\; \forall i.
\end{align}
\end{proof}

Note that \eqref{maja} and \eqref{Pi1}
also yield
\begin{align}\label{maj}
  \Pi_1 f (x_1,\dots,x_\ell)
= \sum_{i=1}^\ell   \Pi_{\set i} f (x_1,\dots,x_\ell)
= \sum_{i=1}^\ell f_i(x_i)
.\end{align}

\begin{corollary}
  \label{Cword}
The $A^\ell$ different unconstrained subsequence counts $N_n(\bw)$ with
$\bw\in\cA^\ell$ converge joíntly, after normalization as in \eqref{ts1},
to a centered multivariate normal distribution in $\bbR^{A^\ell}$
whose support is a subspace of dimension $\ell(A-1)$.
\end{corollary}
\begin{proof}
  By \refR{RCW}, we have joint convergence in \refT{Tword} to some centered
  multivariate normal distribution in $V=\bbR^{A^\ell}$.
Let $L$ be the support of this distribution; then $L$ is a subspace of
$V$. Let $f\in V$.
Then, by \refTs{Tword} and \ref{Tword0},
\begin{align}\label{cw1}
  f\perp L& 
\iff
\sum _{\bw\in\cA^\ell}f(\bw) \frac{N_n(\bw)-\E N_n(\bw)}{n^{\ell-1/2}} \dto0
\iff \gss(f)=0
\notag\\&
\iff \Pi_1 f=0
\iff f\perp V_1.
\end{align}
Hence $L=V_1$, and the result follows by \eqref{maj4}.
\end{proof}

What happens in the degenerate case 
when $\Pi_1f=0$ and thus $\gss(f)=0$? 
For symmetric \Ustat{s}, this was considered
by \citet{Hoeffding} (variance)
and  \citet{RubinVitale} (asymptotic distribution), see also \citet{DynkinM}.
Their results extend to the present asymmetric situation as follows.
We make a final definition of a special subspace of $V$: let
\begin{align}
  V_{\ge k}:=\bigoplus_{i=k}^\ell V_i
=
\set{f\in V: \Pi_i f=0\text{ for } i=0,\dots,k-1}.
\end{align}
In particular, $V_{\ge 1}$ consists of all $f$ with $\E f(\Xi_n)=0$.
Note also that $\fx$ in \eqref{hi} by \eqref{Pi0} and \eqref{maj}
equals $f-\Pi_0 f-\Pi_1 f\in V_{\ge2}$.

\begin{lemma}\label{Ldeg}
Let\/ $0\le k\le\ell$.   
If\/ $f\in V_{\ge k}$, then
$\E U_n(f)^2 = O\bigpar{n^{2\ell-k}}$.
Moreover, if\/ $f\in V_{\ge k}\setminus V_{\ge k+1}$, then
$\E U_n(f)^2 = \Theta\bigpar{n^{2\ell-k}}$.
\end{lemma}
\begin{proof}
  This is easily seen using the expansion \eqref{b1} without the constraint
$\cD$,
and similar to the symmetric case in \cite{Hoeffding};
cf.\ also (in the more complicated \mdep{} case) the cases $k=1$ in \eqref{l1}
and $k=2$ in \eqref{tzol}.
We  omit the details.
\end{proof}

We can now state a general limit theorem that also include degenerate
cases. 
\begin{theorem}\label{Tworddeg}
Let $k\ge1$ and suppose that $f\in V_{\ge k}$. Then
\begin{align}\label{tworddeg}
  n^{k/2-\ell} U_n(f) \dto Z,
\end{align}
where $Z$ is some  polynomial of degree $k$
in independent normal variables (possibly infinitely many).
Moreover, $Z$ is not degenerate unless $f\in V_{\ge k+1}$.
\end{theorem}

\begin{proof}
This follows by \cite[Theorem 11.19]{SJIII}.
As noted in  \cite[Remark 11.21]{SJIII}, it can also be reduced to the
symmetric case in
\cite{RubinVitale}
by the following trick.
Let $(\eta_i)\xoo$ be an \iid{} sequence, independent of
$(\xi_i)\xoo$, with $\eta_i\sim U(0,1)$; then
\begin{align}\label{Usym}
  U_n\bigpar{f;(\xi_i)}
\eqd \sumx_{i_1,\dots,i_\ell\le n} f\bigpar{\xi_{i_1},\dots,\xi_{i_\ell}}
\indic{\eta_{i_1}<\dots<\eta_{i_\ell}},
\end{align}
where $\sum^*$ denotes summation over all distinct $i_1,\dots,i_\ell\in[n]$,
and the sum in \eqref{Usym} can be regarded as a symmetric \Ustat{} based on
$(\xi_i,\eta_i)\xoo$. 
The result \eqref{tworddeg} then follows by 
\cite{RubinVitale}.
\end{proof}

\begin{remark}\label{Rworddeg}
 The case $k=1$ in \refT{Tworddeg} is just a combination of 
\refT{Tword} (in the unconstrained case) and \refT{Tword0}; then $Z$ is
simply a normal variable. When $k=2$, there is a canonical representation
(where the number of terms is finite or infinite)
\begin{align}\label{Z2}
  Z=\frac{1}{2(\ell-2)!}\sum_i \gl_i (\zeta_i^2-1),
\end{align}
where $\zeta_i$ are \iid{} $\N(0,1)$ random variables
and $\gl_i$ are the non-zero eigenvalues (counted with multiplicity) 
of a compact self-adjoint integral operator on 
$L^2(\cA\times\oi, \nu\times \ddx t)$, 
where $\nu:=\cL(\xi_1)$ is the distribution of a single letter
and $\ddx t$ is Lebesgue measure; the kernel $K$
of this integral operator can be constructed from $f$ by applying
\cite[Corollary 11.5(iii)]{SJIII} to the symmetric \Ustat{} in \eqref{Usym}.
We omit the details, but note that in the particular case
$k=\ell=2$, this kernel $K$ is given by
\begin{align}\label{jun1}
  K\bigpar{(x,t),(y,u)} = f(x,y)\indic{t<u} + f(y,x)\indic{t>u},
\end{align}
and thus the integral operator is
\begin{align}\label{jun2}
  h\mapsto Th(x,t)
:= \E\int_0^t f(\xi_1,x)h(\xi_1,u)\dd u + \E\int_t^1 f(x,\xi_1)h(\xi_1,u)\dd u
.\end{align}

When $k\ge3$, the limit $Z$ can be represented as a multiple stochastic
integral
\cite[Theorem 11.19]{SJIII},
but we do not know any canonical representation of it.
See also \cite{RubinVitale} and \cite{DynkinM}.
\end{remark}

We give two simple examples of limits in degenerate cases; in both cases $k=2$.
The second example shows that although the space $V$ has finite dimension, 
the representation \eqref{Z2} might require infinitely many terms.
(Note that the operator $T$ in \eqref{jun2} acts in an infinite-dimensional
space.)

\begin{example}
  \label{E21}
Let $\Xi_n$ be  a symmetric binary string, \ie,
$\cA=\setoi$ and $p(0)=p(1)=1/2$. Consider
\begin{align}\label{as1}
N_n(00)+N_n(11)-N_n(01)-N_n(10)
=U_n(f)
,\end{align}
with
\begin{align}\label{as2}
  f(x,y)&:=
\indic{xy=00}+\indic{xy=11}-\indic{xy=01}-\indic{xy=10}
\notag\\&\phantom:
=\bigpar{\indic{x=1}-\indic{x=0}}\bigpar{\indic{y=1}-\indic{y=0}}.
\end{align}
For convenience, we change notation and consider instead the 
letters $\hxi_i:=2\xi_i-1\in\set{\pm1}$; then $f$ corresponds to
\begin{align}\label{as3}
  \hf(\hat x,\hat y):=\hat x \hat y.
\end{align}
Thus
\begin{align}\label{as4}
  U_n\bigpar{f;(\xi_i)}&
= U_n\bigpar{\hf;(\hxi_i)}
=\sum_{1\le i<j\le n}\hxi_i\hxi_j
=
\frac12\lrpar{\lrpar{\sumin \hxi_i}^2-\sumin\hxi_i^2}
\notag\\&=
\frac12\lrpar{\sumin \hxi_i}^2-\frac{n}2.
\end{align}
By the central limit theorem, $n\qqw\sumin\hxi_i\dto \zeta\sim \N(0,1)$,
and thus \eqref{as4} implies
\begin{align}\label{as5}
  n\qw U_n(f) \dto \tfrac12(\zeta^2-1)
.\end{align}
This is an example of \eqref{tworddeg}, with $k=\ell=2$
and limit given by \eqref{Z2}, in this case with a single term in the sum
and $\gl_1$=1. 

Note that in this example, the function $f$ is  symmetric, so \eqref{as1} is
an example of a symmetric \Ustat{} and thus the result \eqref{as5} is also
an example of the limit result in \cite{RubinVitale}.
\end{example}

\begin{example}\label{E4}
  Let $\cA=\set{a,b,c,d}$, with $\xi_i$ having the symmetric distribution
  $p(x)=1/4$ for each $x\in\cA$.
Consider
\begin{align}\label{sw1}
N_n(ac)-N_n(ad)-N_n(bc)+N_n(bd)
=U_n(f)
,\end{align}
with, writing $\etta_y(x):=\indic{x=y}$,
\begin{align}\label{sw2}
  f(x,y)&:=
\bigpar{\etta_a(x)-\etta_b(x)} \bigpar{\etta_c(x)-\etta_d(x)}
\end{align}
Then, $\Pi_0f=\Pi_1 f=0$ by symmetry, so $f\in V_{\ge 2}=V_2$ (since
$\ell=2$).

Consider the integral operator $T$ on $L^2(\cA\times\oi)$ defined by
\eqref{jun2}. Let $h$ be an eigenfunction with eigenvalue $\gl\neq0$,
and write $h_x(t):=h(x,t)$.
The eigenvalue equation $Th=\gl h$ then is equivalent to, using \eqref{jun2}
and \eqref{sw2},
\begin{align}
  \gl h_a(t) &= \frac{1}{4}\int_t^1\bigpar{h_c(u)-h_d(u)}\dd u, \label{sw3a}
\\
  \gl h_b(t) &= \frac{1}{4}\int_t^1\bigpar{-h_c(u)+h_d(u)}\dd u, \label{sw3b}
\\
  \gl h_c(t) &= \frac{1}{4}\int_0^t\bigpar{h_a(u)-h_b(u)}\dd u, \label{sw3c}
\\
  \gl h_d(t) &= \frac{1}{4}\int_0^t\bigpar{-h_a(u)+h_b(u)}\dd u.\label{sw3d}
\end{align}
These equations hold \aex, but we can redefine $h_x(t)$ by these equations
so that they hold for every $t\in\oi$. Moreover, although originally we
assume only $h_x\in L^2\oi$, it follows from \eqref{sw3a}--\eqref{sw3d} that
the functions $h_x(t)$ are continuous in $t$, and then by induction that
they are infinitely differentiable on $\oi$.
Note also that \eqref{sw3a} and \eqref{sw3b} yield $h_b(t)=-h_a(t)$, and
similarly $h_d(t)=-h_c(t)$. Hence, we may reduce the system to
\begin{align}
  \gl h_a(t) &= \frac{1}{2}\int_t^1h_c(u)\dd u, \label{sw4a}
\\
  \gl h_c(t) &= \frac{1}{2}\int_0^th_a(u)\dd u. \label{sw4c}
\end{align}
By differentiation, for $t\in(0,1)$,
\begin{align}
  h_a'(t) &= -\frac{1}{2\gl}h_c(t),\label{sw6a}
\\
  h_c'(t) &= \frac{1}{2\gl}h_a(t)\label{sw6c}.
\end{align}
Hence, with $\go:=1/(2\gl)$,
\begin{align}\label{hc''}
h_c''(t)=-\go^2h_c(t).   
\end{align}
Furthermore, \eqref{sw4c} yields $h_c(0)=0$, and thus \eqref{hc''} has the
solution (up to a constant factor that we may ignore)
\begin{align}
  h_c(t) = \sin \go t = \sin \frac{t}{2\gl}.
\end{align}
By \eqref{sw6c}, we then obtain
\begin{align}
  h_a(t) = \cos \go t = \cos \frac{t}{2\gl}.
\end{align}
However, \eqref{sw4a} also yields $h_a(1)=0$ and thus we must have
$\cos(1/2\gl)=0$; hence
\begin{align}\label{swgl}
  \gl = \frac{1}{(2N+1)\pi},
\qquad N\in\bbZ.
\end{align}
Conversely, for every $\gl$ of the form \eqref{swgl}, the argument can be
reversed to find an eigenfunction $h$ with eigenvalue $\gl$.
It follows also that all these eigenvalues are simple.
Consequently, \refT{Tworddeg} and \eqref{Z2} yield
\begin{align}\label{swtor}
  n\qw U_n(f) \dto 
\frac{1}{2\pi}\sum_{N=-\infty}^\infty \frac{1}{2N+1}(\zeta_N^2-1)
=
\frac{1}{2\pi}\sum_{N=0}^\infty \frac{1}{2N+1}(\zeta_N^2-\zeta_{-N-1}^2)
\end{align}
where, as above, $\zeta_N$ are \iid{} and $\N(0,1)$.
A simple calculation, using the product formula for cosine
\cite[\S12]{EulerE61},
\cite[4.22.2]{NIST}, shows
that the moment generating function of the limit distribution $Z$ in
\eqref{swtor} is 
\begin{align}\label{swmgf}
  \E e^{s Z} = \frac{1}{\cos\qq(s/2)},
\qquad |\Re s|<\pi.
\end{align}

It can be shown that $Z\eqd \frac12\intoi B_1(t)\dd B_2(t)$
if $B_1(t)$ and $B_2(t)$ are two independent standard Brownian motions;
this is, for example, a consequence of \eqref{swmgf} and the 
calculation, using 
\cite{CM1945}
or
\cite[page 445]{RY} for the
final equality, 
\begin{align}
 \E{ e^{s\intoi B_1(t)\dd B_2(t)}}&=
\E  \E\bigsqpar{ e^{ s\intoi B_1(t)\dd B_2(t)}\mid B_1}
=\E e^{\frac{s^2}{2}\intoi B_1(t)^2\dd t}
\notag\\&
=\cos\qqw(s),
\qquad |\Re s|<\pi/2.
\end{align}
We omit the details, but note that this representation of the limit $Z$ is
related to the special form \eqref{sw2} of $f$;
we may, intuitively at least, 
interpret $B_1$ and $B_2$ as limits (by Donsker's theorem) of partial
sums of $\etta_a(\xi_i)-\etta_b(\xi_i)$ and $\etta_c(\xi_i)-\etta_d(\xi_i)$.
In fact, in this example
it is possible to give a rigorous proof of $n\qw U_n(f)\dto Z$ by this 
approach; again we omit the details.
\end{example}

\newcommand\AAP{\emph{Adv. Appl. Probab.} }
\newcommand\JAP{\emph{J. Appl. Probab.} }
\newcommand\JAMS{\emph{J. \AMS} }
\newcommand\MAMS{\emph{Memoirs \AMS} }
\newcommand\PAMS{\emph{Proc. \AMS} }
\newcommand\TAMS{\emph{Trans. \AMS} }
\newcommand\AnnMS{\emph{Ann. Math. Statist.} }
\newcommand\AnnPr{\emph{Ann. Probab.} }
\newcommand\CPC{\emph{Combin. Probab. Comput.} }
\newcommand\JMAA{\emph{J. Math. Anal. Appl.} }
\newcommand\RSA{\emph{Random Structures Algorithms} }
\newcommand\DMTCS{\jour{Discr. Math. Theor. Comput. Sci.} }

\newcommand\AMS{Amer. Math. Soc.}
\newcommand\Springer{Springer-Verlag}
\newcommand\Wiley{Wiley}

\newcommand\vol{\textbf}
\newcommand\jour{\emph}
\newcommand\book{\emph}
\newcommand\inbook{\emph}
\def\no#1#2,{\unskip#2, no. #1,} 
\newcommand\toappear{\unskip, to appear}

\newcommand\arxiv[1]{\texttt{arXiv:#1}}
\newcommand\arXiv{\arxiv}

\newcommand\xand{and }
\renewcommand\xand{\& }

\def\nobibitem#1\par{}
\renewcommand\MR[1]{\relax}

\section*{Acknowledgement}
I thank Wojciech Szpankowski for stimulating discussions 
on patterns in random strings
during the last 20 years.
Furthermore, I thank Andrew Barbour and Nathan Ross for help with references,
and the anonymous referees for helpful comments and further references;
in particular, \refT{Trate} is based on suggestions from a referee.


\begin{thebibliography}{99}

\bibitem{AlsmeyerH01}
Gerold Alsmeyer \xand Volker Hoefs: 
Markov renewal theory for stationary $m$-block factors. 
\emph{Markov Process. Related Fields} \vol7 (2001), no. 2, 325--348. 
\MR{1856500}

\bibitem{BaldiR}
Pierre Baldi \xand Yosef Rinott:
On normal approximations of distributions in terms of dependency graphs.
\emph{Ann. Probab.} \vol{17} (1989), no. 4, 1646--1650.
\MR{1048950}


\bibitem{BKR}
A. D. Barbour, Micha{\l} Karo{\'n}ski \xand  Andrzej Ruci{\'n}ski:
A central limit theorem for decomposable random variables with applications
to random graphs. 
\emph{J. Combin. Theory Ser. B} \vol{47} (1989), no. 2, 125--145.
\MR{1047781}

\bibitem{BenderK}
Edward A. Bender \xand Fred Kochman:
The distribution of subword counts is usually normal.
\emph{European J. Combin.} \vol{14} (1993), no. 4, 265--275.
\MR{1226574}

\bibitem{Billingsley56}
Patrick Billingsley:
The invariance principle for dependent random variables.
\emph{Trans. Amer. Math. Soc.} \vol{83} (1956), 250--268.
\MR{0090923}

\bibitem{Blom}
Gunnar Blom:
Some properties of incomplete $U$-statistics.
\emph{Biometrika} \vol{63} (1976), no. 3, 573--580.
\MR{0474582}

\bibitem[B\'ona(2007)]{Bona-Normal}
Mikl\'os B\'ona:
The copies of any permutation pattern are asymptotically normal.
Preprint, 2007.
\arXiv{0712.2792}

\bibitem[B\'ona(2008)]{Bona-d}
Mikl\'os B\'ona:
Generalized descents and normality.
\emph{Electron. J. Combin.} \vol{15} (2008), no. 1, Note 21, 8 pp.
\MR{2426166}

\bibitem[B\'ona(2010)]{Bona3}
Mikl\'os B\'ona:
On three different notions of monotone subsequences.  
\emph{Permutation Patterns}, 89--114,  
London Math. Soc. Lecture Note Ser., 376, 
Cambridge Univ. Press, Cambridge, 2010.
\MR{2732825} 


\bibitem{BourdonVallee02}
J\'er\'emie Bourdon
\xand Brigitte Vall\'ee:
Generalized pattern matching statistics.  
\emph{Mathematics and Computer Science, II (Versailles, 2002)}, 249--265,
Birkh\"auser, Basel, 2002.
\MR{1940140}

\bibitem{BourdonVallee06}
J\'er\'emie Bourdon
\xand Brigitte Vall\'ee:
Pattern matching statistics on correlated sources. 
\emph{LATIN 2006: Theoretical informatics}, 224--237,
Lecture Notes in Comput. Sci., {3887}, Springer, Berlin, 2006. 
\MR{2256334}


\bibitem[Bradley(2007)]{Bradley}
Richard C. Bradley:
\emph{Introduction to Strong Mixing Conditions. Vol. 1--3}. 
Kendrick Press, Heber City, UT, 2007. 
\MR{2325294--2325296} 

\bibitem{CM1945} 
R. H. Cameron \xand  W. T. Martin:
Transformations of Wiener integrals under a general class of linear
transformations. 
\emph{Trans. Amer. Math. Soc.} \vol{58} (1945), 184--219.
\MR{0013240}

\bibitem{ChenGS}
Louis H. Y. Chen,  Larry Goldstein \xand Qi-Man Shao:
\book{Normal Approximation by Stein's Method}. 
Springer, Berlin, 2011. 
\MR{2732624}

\bibitem{DehlingW}
Herold Dehling \xand Martin Wendler:
Central limit theorem and the bootstrap for $U$-statistics of strongly
mixing data.
\emph{J. Multivariate Anal.} \vol{101} (2010), no. 1, 126--137.
\MR{2557623}

\bibitem[Dynkin and Mandelbaum(1983)]{DynkinM}
E. B.  Dynkin \xand  A. Mandelbaum:
Symmetric statistics, Poisson point processes, and multiple Wiener integrals.
\emph{Ann. Statist.} \vol{11} (1983), no. 3, 739--745. 
\MR{0707925}

\bibitem{EulerE61}
Leonhard Euler:
De summis serierum reciprocarum ex potestatibus numerorum naturalium ortarum
dissertatio altera, in qua eaedem summationes ex fonte maxime diverso
derivantur.
\emph{Miscellanea Berolinensia} \vol7 (1743), 172--192.
Reprinted in
\emph{Opera Omnia}, Series 1, Volume 14, 138--155,
Teubner, Leipzig, 1925.

\bibitem{Even-Zohar}
Chaim Even-Zohar: 
Patterns in random permutations. 
\emph{Combinatorica} \vol{40} (2020), no. 6, 775--804.
\MR{4208097}

\bibitem{Even-ZoharEtAl2020} 
Chaim Even-Zohar, Tsviqa Lakrec \xand Ran J. Tessler:
Spectral analysis of word statistics.
\emph{S\'em. Lothar. Combin.} \vol{85B} (2021), Art. 81, 12 pp.
\MR{4311962}

\bibitem{Fang} 
Xiao Fang: 
A multivariate CLT for bounded decomposable
random vectors with the best known rate. 
\emph{J. Theoret. Probab.} 29 (2016), no. 4, 1510--1523.
\MR{3571252}

\bibitem{FisherLee}
N. I. Fisher \xand A. J.  Lee:
Nonparametric measures of angular-angular association.
\emph{Biometrika} \vol{69} (1982), no. 2, 315--321.
\MR{0671969}

\bibitem[Flajolet and Sedgewick(2009)]{FlajoletS} 
Philippe~Flajolet \& Robert~Sedgewick:
\emph{Analytic Combinatorics}.
Cambridge Univ. Press, Cambridge, UK, 2009.
\MR{2483235}

\bibitem[Flajolet, Szpankowski and  Vall\'ee(2006)]{FlajoletSzV}
Philippe Flajolet,
Wojciech Szpankowski
\xand
Brigitte Vall\'ee:
Hidden word statistics. 
\emph{J. ACM} \textbf{53} (2006), no. 1, 147--183. 
\MR{2212002} 

\bibitem[Gut(2009)]{Gut-SRW}
Allan Gut:
\book{Stopped Random Walks}.
2nd ed.,
Springer, New York, 2009.
\MR{2489436} 

\bibitem{Gut}
Allan Gut:
\emph{Probability: A Graduate Course},
2nd ed., Springer, New York, 2013. 
\MR{2977961}

\bibitem{HanQ}
Fang Han \xand Tianchen Qian:
On inference validity of weighted U-statistics under data heterogeneity.
\emph{Electron. J. Stat.} \vol{12} (2018), no. 2, 2637--2708.
\MR{3849897}

\bibitem[Hoeffding(1948)]{Hoeffding}
  Wassily Hoeffding:
  A class of statistics with asymptotically normal
 distribution. 
\emph{Ann. Math. Statistics} \textbf{19} (1948), 293--325.
\MR{0026294}

\bibitem[Hoeffding(1961)]{Hoeffding-LLN}
Wassily Hoeffding:
The strong law of large numbers for   $U$-statistics.
Institute of Statistics, Univ. of North Carolina, Mimeograph
      series 302 (1961).
\url{https://repository.lib.ncsu.edu/handle/1840.4/2128}


\bibitem[Hofer(2018)]{Hofer}
Lisa Hofer: 
A central limit theorem for vincular permutation patterns.
\emph{Discrete Math. Theor. Comput. Sci.} \vol{19} (2018), no. 2, 
Paper No. 9, 26~pp. 
\MR{3783365}
%


\bibitem{HsingW}
Tailen Hsing \xand  Wei Biao Wu:
On weighted $U$-statistics for stationary processes.
\emph{Ann. Probab.} \vol{32} (2004), no. 2, 1600--1631.
\MR{2060311}

\bibitem[Jacquet and Szpankowski(2015)]{JacSzp}  
Philippe Jacquet \xand Wojciech Szpankowski: 
\book{Analytic Pattern Matching. From DNA to Twitter.} 
Cambridge University Press, Cambridge, 2015. 
\MR{3524836}

\bibitem{SJ48}
Sreenivasa Rao Jammalamadaka \xand Svante Janson:
Limit theorems for a triangular scheme of $U$-statistics with
applications to inter-point distances.
\emph{Ann. Probab.} \vol{14} (1986), 1347--1358.
\MR{0866355}


\bibitem{SJ29}
Svante Janson: 
Renewal theory for $M$-dependent variables. 
\emph{Ann. Probab.} \vol{11} (1983), no. 3, 558--568.
\MR{0704542}

\bibitem{SJ58}
Svante Janson:
Normal convergence by higher semi-invariants with applications to sums of
dependent random variables and random graphs.
\emph{Ann. Probab.} \vol{16} (1988), no. 1, 305--312.
\MR{0920273}

\bibitem{SJIII}
Svante Janson:
\emph{Gaussian Hilbert Spaces},
Cambridge Univ. Press, Cambridge, UK, 1997.
\MR{1474726}


\bibitem{SJ286}
Svante Janson:
On degenerate sums of $m$-dependent variables.
\emph{J. Appl. Probab.} \vol{52} (2015), no. 4, 1146--1155.
\MR{3439177}

\bibitem{SJ332}
Svante Janson:
Renewal theory for asymmetric $U$-statistics.
\emph{Electron. J. Probab.} \vol{23} (2018), Paper No. 129, 27 pp.
\MR{3896866}

\bibitem{SJ361}
Svante Janson:
A central limit theorem for $m$-dependent variables.
Preprint, 2021.
\arXiv{2108.12263}

\bibitem{SJ367}
Svante Janson:
The number of occurrences of patterns in a random tree or forest permutation.
Preprint, 2022.
\arXiv{2203.04182}

\bibitem[Janson,  Nakamura and Zeilberger(2015)]{SJ287}
Svante Janson, Brian Nakamura \xand Doron Zeilberger:
On the asymptotic statistics of the number of occurrences of multiple
permutation patterns.
\emph{Journal of Combinatorics} \textbf{6} (2015), no. 1-2, 117--143.
\MR{3338847}
  
\bibitem{SJ349}
Svante Janson \xand  Wojciech Szpankowski:
Hidden words statistics for large patterns.
\emph{Electronic J. Combinatorics},
\vol{28}:2 (2021), Article P2.36.

\bibitem{Kallenberg}
Olav Kallenberg:
\book{Foundations of Modern Probability}.
2nd ed., Springer, New York, 2002. 
\MR{1876169}


\bibitem{Major}
P\'eter Major:
Asymptotic distributions for weighted $U$-statistics.
\emph{Ann. Probab.} \vol{22} (1994), no. 3, 1514--1535.
\MR{1303652}

\bibitem[Malevich and  Abdalimov(1982)]{MalevichA}
T. L. Malevich \xand B. Abdalimov:
Refinement of the central limit theorem for $U$-statistics of $m$-dependent
variables. (Russian.)
\emph{Teor. Veroyatnost. i Primenen.} \vol{27} (1982), no. 2, 369--373.
English translation: 
\emph{Theory Probab. Appl.} \vol{27} (1982), no. 2, 391--396.
\MR{0657936}

\bibitem[Miller and Sen(1972)]{MillerSen}
R. G. Miller, Jr. \xand Pranab Kumar Sen:
Weak convergence of $U$-statistics and von Mises' differentiable statistical
functions.
\emph{Ann. Math. Statist.} \vol{43} (1972), 31--41.
\MR{0300321}

\bibitem{NicodemeSF}
Pierre  Nicod\'eme,  Bruno Salvy \xand Philippe Flajolet:
Motif statistics.
\emph{Theoret. Comput. Sci.} \vol{287} (2002), no. 2, 593--617.
\MR{1930238}

\bibitem{NIST}
\emph{NIST Handbook of Mathematical Functions}. 
Edited by Frank W. J. Olver, Daniel W. Lozier, Ronald F. Boisvert \xand
Charles W. Clark. 
Cambridge Univ. Press, 2010. \\
Also available as 
\emph{NIST Digital Library of Mathematical Functions},
\url{http://dlmf.nist.gov/}
\MR{2723248}


\bibitem{ONeilR}
Kevin A. O'Neil \xand  Richard A. Redner:
Asymptotic distributions of weighted $U$-statistics of degree 2.
\emph{Ann. Probab.} \vol{21} (1993), no. 2, 1159--1169.
\MR{1217584}

\bibitem[Orey(1958)]{Orey}
Steven Orey: 
A central limit theorem for $m$-dependent random variables. 
\emph{Duke Math. J.} \vol{25} (1958), 543--546.
\MR{0097841}

\bibitem[Peligrad(1996)]{Peligrad}
Magda  Peligrad:
On the asymptotic normality of sequences of weak dependent random variables.
\emph{J. Theoret. Probab.} \vol9 (1996), no. 3, 703--715.
\MR{1400595}

\bibitem[Pike(2011)]{Pike}
John Pike:
 Convergence rates for generalized descents. 
\emph{Electron. J. Combin.} \vol{18} (2011), no. 1, Paper 236, 14 pp. 
\MR{2880686}

\bibitem{Raic}
M. Rai{\v c}:
A multivariate CLT for decomposable random vectors with finite second moments.
\emph{J. Theoret. Probab.} \vol{17} (2004), no. 3, 573--603.
\MR{2091552}

\bibitem{RegSzp98}
Mireille R{\'e}gnier \xand Wojciech Szpankowski:
On pattern frequency occurrences in a Markovian sequence.
\emph{Algorithmica} \vol{22} (1998), no. 4, 631--649.
\MR{1701633}

\bibitem[Revuz and Yor(1999)]{RY}  
Daniel Revuz \xand  Marc Yor:
\book{Continuous Martingales and Brownian Motion},
$3^{rd}$ edition,
\Springer, Berlin, 1999.
\MR{1725357}


\bibitem{Rinott}
Yosef Rinott:
On normal approximation rates for certain sums of dependent random variables.
\emph{J. Comput. Appl. Math.} \vol{55} (1994), no. 2, 135--143.
\MR{1327369}


\bibitem{RinottR}
Yosef Rinott \xand Vladimir Rotar:
On coupling constructions and rates in the CLT for dependent summands with
applications to the antivoter model and weighted $U$-statistics. 
\emph{Ann. Appl. Probab.} \vol7 (1997), no. 4, 1080--1105.
\MR{1484798}

\bibitem[Rubin and Vitale(1980)]{RubinVitale}
H. Rubin \xand  R. A. Vitale:
Asymptotic distribution of symmetric statistics.
\emph{Ann. Statist.} \vol8 (1980), no. 1, 165--170.
\MR{0557561}

\bibitem[Sen(1960)]{Sen-LLN}
Pranab Kumar Sen:
On some convergence properties of $U$-statistics.
Calcutta Statist. Assoc. Bull. 10 (1960), 1--18.
\MR{0119286}

\bibitem[Sen(1963)]{Sen-m}
Pranab Kumar Sen:
On the properties of $U$-statistics when the observations are not independent.
I. 
Estimation of non-serial parameters in some stationary stochastic process. 
\emph{Calcutta Statist. Assoc. Bull.} \vol{12} (1963), 69--92.
\MR{0161417}

\bibitem{Sen-*}
Pranab Kumar Sen:
Limiting behavior of regular functionals of empirical distributions for
stationary mixing processes.
\emph{Z. Wahrscheinlichkeitstheorie und Verw. Gebiete} \vol{25} 
(1972/73), 
71--82.
\MR{0329003}

\bibitem{Sen-generalized}
Pranab Kumar Sen:
Weak convergence of generalized $U$-statistics. 
\emph{Ann. Probability} \vol2 (1974), no. 1, 90--102.
\MR{0402844}


\bibitem{ShapiroH}
C. P. Shapiro \xand Lawrence Hubert:
Asymptotic normality of permutation statistics derived from weighted sums of
bivariate functions. 
\emph{Ann. Statist.} \textbf7 (1979), no. 4, 788--794.
\MR{0532242}

\bibitem{Szp}
 Wojciech Szpankowski: 
\emph{Average Case Analysis of Algorithms on  Sequences}. 
Wiley-Interscience, New York, 2001. 
\MR{1816272}

\bibitem[Yoshihara(1976)]{Yoshihara76}
 Ken-ichi Yoshihara:
Limiting behavior of $U$-statistics for stationary, absolutely regular
processes. 
\emph{Z. Wahrscheinlichkeitstheorie und Verw. Gebiete} \vol{35} (1976),
no. 3, 237--252. 
\MR{0418179}

\bibitem[Yoshihara(1992)]{Yoshihara92}
Ken-ichi Yoshihara:
Limiting behavior of $U$-statistics for strongly mixing sequences.
\emph{Yokohama Math. J.} \vol{39} (1992), no. 2, 107--113.
\MR{1150042}

\bibitem{Zhou}
Zhou Zhou:
Inference of weighted $V$-statistics for nonstationary time series and its
applications.
\emph{Ann. Statist.} \vol{42} (2014), no. 1, 87--114.
\MR{3161462}

\end{thebibliography}
\end{document}